\documentclass[12pt]{article}
\usepackage{amsmath,amssymb,color}
\usepackage{enumerate}
\usepackage[utf8]{inputenc}
\usepackage{amsthm,hyperref,mathtools,}
\usepackage{enumitem}
\usepackage[mathscr]{euscript}
\usepackage{bm}
\usepackage{babel}

\usepackage{nccmath} 

\newtheorem{thm}{Theorem}[section]
\newtheorem{definition}[thm]{Definition}

\newtheorem{corollary}[thm]{Corollary}

 \newtheorem{cor}[thm]{Corollary}
 \newtheorem{prop}[thm]{Proposition}
 \newtheorem{remark}[thm]{Remark}
 \newtheorem{lemma}[thm]{Lemma}

 \newcommand {\R}{\mathbb{R}}
 
  \newcommand {\SH}{{\bf H}}
 \newcommand {\M}{{\cal M }}

 \newcommand {\RN}{\mathbb{R}^N}

 \newcommand{\be}{\begin{equation}}
 \newcommand{\ee}{\end{equation}}

 \newcommand{\bd}{\begin{description}}
 \newcommand{\ed}{\end{description}}
 \newcommand{\e}{\varepsilon}

 \newcommand{\bdp}{\begin{displaymath}}
 \newcommand{\edp}{\end{displaymath}}
 \newcommand{\bea}{\begin{eqnarray}}
 \newcommand{\eea}{\end{eqnarray}}
 \newcommand{\brr}{\begin{eqnarray*}}
 \newcommand{\err}{\end{eqnarray*}}

\begin{document}
\title{ Nonlinear Schr\"odinger systems with all attractive forces}
\author{{ Jaeyoung Byeon$^1$, Junyoung Heo$^1$, Zhi-Qiang Wang$^{2,3}$}\\
{\it\small $^1$ Department of Mathematical Sciences, KAIST}\\
{\it\small 291 Daehak-ro, Yuseong-gu,  Daejeon 305-701,  Republic of Korea}\\
{\it\small $^2$ College of Mathematics and Statistics,
	Fujian Normal University}\\
{\it\small
	Fuzhou, Fujian 350117, P. R. China}\\
{\it\small $^3$ Department of Mathematics and Statistics,
 	Utah State University}\\
{\it\small 
 	Logan, Utah 84322, USA}}
 
\maketitle

\begin{abstract} In this paper we investigate in a systematic way the solution structure of nonnegative solutions for coupled nonlinear Schr\"odinger systems in the all attractive regime. 
For all frequencies equal case we obtain results on Morse index of synchronized positive vector solutions and nonnegative semi-vector solutions as well as their kernel of linearized systems at these solutions. We establish a variational characterization of synchronized positive vector solutions and as applications we give a new existence result about positive vector solutions for the general systems with arbitrary frequencies.
We also examine the synchronization phenomenon of positive vector solutions, and prove that if the interaction matrix has exactly one positive eigenvalue, any positive vector solution is synchronized. Finally we provide examples of domains for which synchronization of positive solutions fails to hold if  the interaction matrix has at least two positive eigenvalues. 
Our results reveal the effect  of the spectral information of the interaction matrix of the couplings and geometry of a domain on the solution structure.

\end{abstract}

\tableofcontents

\section{Introduction}
In this paper, we study the structure of nonnegative solutions for the following 
coupled nonlinear Schr\"odinger system 
\begin{equation}\label{eq1}
- \Delta u_i + \lambda_i u_i = \sum_{j=1}^n a_{ij}u_iu_j^2 \quad \text{in } \Omega, \quad i=1,\cdots,n,\end{equation}
where $u_i \in H \equiv H^1_0(\Omega)$ for $1\leq i\leq n$ and $\Omega$ is a 
domain in $\mathbb{R}^N$ for $N \le 3.$
Here the parameters $a_{ij}$ representing an interaction between components $u_i$ and $u_j$ satisfy $a_{ij}=a_{ji} \ge 0$ for $1\leq i,j\leq n$. 
If $\Omega = \mathbb{R}^N$, we define 
$H$ to be the subspace of radially symmetric functions in $H^1(\mathbb{R}^N)$. 
We denote the space of $N-$times product of $H$ by $\mathbf{H} \equiv H^n$.  A solution $\vec{u} = (u_1,\dots,u_n)$ of \eqref{eq1} is called 
a vector solution of \eqref{eq1} if $u_i \ne 0$ for each $i =1,\dots,n$. 
A vector 
solution $\vec{u} \in \mathbf{H}$ is called a positive vector solution of 
\eqref{eq1} if $u_i > 0$ in $\Omega$ for each $i =1,\dots,n$. A solution $\vec{u}$ 
of \eqref{eq1} is called a semi-vector nonnegative solution if $u_i \ge 0$ in $
\Omega$ for all $1 \le i \le n$, with $u_i > 0$ and $u_j = 0$ for some distinct 
indices $i,j \in \{1,\dots,n\}$. We note that a semi-vector nonnegative solution
is a positive vector solution of the same type of system with a smaller number of equations. Thus, we are naturally interested in the positive vector solutions.
Systems of this type arise in mathematical 
physics, most notably in the theory of multi-species Bose-Einstein 
condensations \cite{CLLL, MS, R}. Physically, the positive coupling constants 
$a_{ij}$ ($i\neq j$) represent attractive interactions between components $u_i$ and $u_j$, whereas negative coupling 
constants represent repulsive interactions. Following the pioneering work of Lin 
and Wei \cite{LW1}, an extensive body of literature has explored cases involving 
purely attractive, purely repulsive, and mixed couplings (see, e.g., \cite{AC, 
ACR, BDW, BW, BWW, BKS, CZ1, Corr1, Corr2, CFT, DWW, dFL, IT, LW2, LcyW1, LLC, 
LLW, LWz1, LWz2, MMP, NR, NTTV, PW, SW1, S, S1, ST, TTVW, TV, TW1, WW1, 
WW11, WW12, WY, ZW} and the references therein), revealing a rich and complex 
solution structure. From a mathematical perspective, understanding the 
structure of positive vector solutions to \eqref{eq1} poses significant 
challenges. These stem primarily from the coexistence of numerous semi-vector solutions and the fact that positive vector solutions may fail to exist for 
certain regimes of coupling parameters. Early research successfully 
established the existence and classification of positive vector solutions, though 
predominantly for two-component systems ($n=2$). For purely attractive 
couplings, results concerning systems with a large number of components 
remain scarce; thus for general $n \ge 2,$ it is very challenging to find the optimal conditions of $\lambda_i, i=1,\cdots,n,$ and $a_{ij}, 1\le i,j \le n,$ for an existence of positive vector solutions(refer to a recent survey paper \cite{B2}).

In this paper, we focus on purely attractive forces and 
systematically investigate the solution structure according to the algebraic property of the interaction matrix $A_{n\times n}=(a_{ij})$.
We mainly consider a special yet pivotal class where the linear coefficients $
\lambda_j$ are identical, though as applications of our results for this special class we will also consider the general case with different $\lambda_i's$ 
give both perturbation and non-perturbation type results. We set $\lambda_1 = \dots = 
\lambda_n = \lambda> 0$ and reduce \eqref{eq1} to
\begin{equation} \label{eq3}- \Delta u_i + 
\lambda u_i = \sum_{j=1}^n a_{ij}u_iu_j^2 \quad \text{in } \Omega, \quad u_i \in H, \quad 
1\leq i\leq n.\end{equation}
Solutions to the system \eqref{eq3} correspond to critical points of 
the associated $C^2$ energy functional $I: \mathbf{H} \to \mathbb{R}$, defined 
by $$I(\vec{u}) \equiv \frac{1}{2} \|\vec{u}\|_{\Omega}^2 - \frac{1}{4} \int_{\Omega} 
\sum_{i,j=1}^{n} a_{ij} u_i^2 u_j^2 \, dx,$$
where for $p \ge 1$ and $
\vec{u} = (u_1,\dots,u_n) \in \mathbf{H}$, we use notations
\[ |u|_{p, \Omega}^p = \int_{\Omega}|u|^p\, dx, \quad \|u\|_{\Omega}^2 = 
|\nabla u|_{2, \Omega}^2 + \lambda |u|_{2, \Omega}^2, \quad \|\vec{u}\|_\Omega = \left(  \sum_{i=1}^n \|u_i\|_{\Omega}^2 \right)^{1/2}.\] 
The Morse index of a solution $\vec{u}$ 
to \eqref{eq3} is defined as the maximal dimension of a subspace $W \subset 
\mathbf{H}$ such that $$\{\vec{\phi} \in W \setminus \{0\} \mid I^{\prime\prime}
(\vec{u})(\vec{\phi},\vec{\phi}) < 0\}.$$
Let $A=(a_{ij})_{1\leq i,j\leq n}$ denote the 
interaction matrix. Under the assumption of purely attractive forces $a_{ij}\geq 
0$ for all $1\leq i,j\leq n$, we explore the interplay between the eigenvalues and eigenfunctions of $A$ and the Morse index of the nonnegative solutions 
of \eqref{eq3}.
 We also study the synchronization phenomenon, wherein the 
components of a solution are mutually proportional by positive constants. 
More precisely, let $U \in H$ be a positive solution to the scalar equation
\begin{equation}\label{singleeq}\Delta u - \lambda u + u^3 = 0 \text{ in } 
\Omega, \quad u = 0 \text{ on } \partial\Omega.\end{equation}
We define the subspaces $\mathbb{R} U = \{cU \mid c\in\mathbb{R}\} \subset H$ 
and $\mathbb{R}^n U = \{\vec{c}U 
\mid \vec{c}\in\mathbb{R}^n\} \subset \mathbf{H}$.  A positive vector solution $\vec{u}$ of \eqref{eq3} is said to be synchronized if $\vec{u} \in \mathbb{R}^n U$.
It is easy to check that a synchronized positive vector solution $\vec{c}U \in \mathbb{R}^n U$
exists if and only if there exists a vector $\vec{c}=(c_1,\cdots,c_n) \in \mathbb{R}^n$ satisfying $c_1,\cdots,c_n > 0$ and 
\begin{equation}
\label{liposol} 
A
\begin{pmatrix}
c_1^2\\
c_2^2\\
\vdots\\
c_n^2
\end{pmatrix}
=
\begin{pmatrix}
1\\
1\\
\vdots\\
1
\end{pmatrix}.
\end{equation}

We have two primary objectives to accomplish in this paper. On one hand, we give precise results on the Morse index of the synchronized positive vector solutions and the synchronized non-negative semi-vector solutions in terms of the eigenvalues and eigenvectors associated with the interaction matrix $A$, and we also completely characterize the kernel of the linearization at these solutions. We establish a variational formulation for synchronized positive vector solutions to be ground state solutions. Then as applications of these results, we consider more general systems for which the linear frequencies $\lambda_i'$s are not assumed to be the same and give a new sufficient condition for the ground states to be positive vector solutions. Also as applications of the Morse index results we provide local bifurcation results in terms of interaction parameters. 
  On the other hand, with the existence of the synchronized positive vector solutions, we want to examine whether these special solutions are the only positive vector solutions, i.e., whether all positive vector solutions are the synchronized ones, leading to uniqueness of positive solutions in radially symmetric domains. We prove a very general result that if $A$ has exactly one positive eigenvalue and if every row of $A$ is a non-zero vector then any positive vector solution must be synchronized. This in turn gives the uniqueness of positive vector solutions for \eqref{eq3} if \eqref{liposol} has a  solution and non-existence of positive vector solutions if \eqref{liposol} has no solution. Conversely for a large class of non-singular matrices $A$ having at least two eigenvalues we construct non-synchronized positive vector solutions on some domains for which the uniqueness of positive solutions does not hold for the scalar equation (\ref{singleeq}).
  
  In the following, among other things for reader's convenience we introduce some of our main results informally, and we refer the exact statements to the specific sections.
  
\vskip .1in

\noindent {\bf Morse Index of synchronized positive vector solutions and semi-vector solutions.} 
In Section 2, we 
establish explicit relationship between the spectral and algebraic properties of the interaction matrix $A$ and the Morse index of both synchronized positive vector solutions and nonnegative synchronized semi-vector  solutions of \eqref{eq3} [Theorems \ref{T1}, \ref{T1-1}]. 

\noindent{\bf Theorem \ref{T1}} 
{\it Suppose that there exists a positive vector solution $\vec{u} =(c_1,c_2,\cdots,c_n)U =\vec{c}U \in\R^nU$ to the system \eqref{eq3}.  Then the following results hold.
\begin{itemize}
\item [(a)]
The Morse index of $\vec{u}$ equals the number of positive eigenvalues $\lambda_1 \ge \cdots \ge \lambda_k > 0$ of $A$, up to their multiplicity. 
Moreover,  a maximal negative space $W$ of $I^{\prime\prime}(\vec{u})$ satisfying 
\[
 I''(\vec{u})(\vec{\phi},\vec{\phi})<0\ \mathrm{for}\ \vec{\phi}\in W\setminus \{0\} 
\]
is characterized by a linear span of
\[ \Big \{(\vec{\gamma} \oslash \vec{c})U \ | \  A\vec{\gamma} = \lambda_i \vec{\gamma}
 \ i  =1,\cdots,k\Big \}.\]
 \item[(b)] The vector solution $\vec{u}$ is nondegenerate in $\SH$ if and only if $U$ is nondegenerate in $H$ and $\det (A) \neq 0$. \end{itemize}}(see details in Theorem \ref{T1} in Section 2.2)

 When the vector solution $\vec{u}$ is degenerate,  we also give a complete characterization of
the kernel of a linearized system \eqref{eq3} at $\vec{u}\in \R^nU$ in Theorem  \ref{T1}. 

\noindent{\bf Theorem \ref{T1-1}}.
{\it Let $1\leq m< n$ and $c_1,c_2,\cdots,c_m>0$ and $\vec{u}=(c_1,c_2,\cdots,c_m,0,\cdots,0)U\in\R^nU$ be a semi-vector nonnegative solution to a system \eqref{eq3}. Define a matrix $B$ by
\begin{equation}
\label{B}
b_{ij}=
\begin{cases}
    2a_{ii}c_i^2+\sum_{k=1}^n a_{ik}c_k^2\ \mathrm{if}\ i=j,\\
    2a_{ij}c_ic_j\ \mathrm{if}\ i\neq j, \nonumber
\end{cases}
\end{equation} $A_m=(a_{ij})_{1\leq i,j\leq m}$, $B_{m}=(b_{ij})_{1\leq i,j\leq m}$, and $f(x) = \max\{i \in \mathbb{N} \ | \ \omega_i < x\},$ where
 $   f(x) = 0$ if $\{i \in \mathbb{N} \ | \ \omega_i < x\} = \emptyset.$ Denote by $\omega_1 = 1 < \omega_2 \le \cdots$ the eigenvalues of multiplicity one with the corresponding eigenfunctions $\{\psi_i\}_{i=1}^\infty$ satisfying $
    \Delta\psi_i-\lambda \psi_i+\omega_i U^2\psi_i=0, \psi_i\in H.
$
Then the following results hold.
\begin{itemize}
\item[(a)] The Morse index of $\vec{u}$ equals
    \begin{equation}\label{eigenformula}
       the \  number \ of \ the \ positive \ eigenvalues \ of A_m +\sum_{i=m+1}^n f(b_{ii}).
    \end{equation}
\item[(b)] $\vec{u}$ is nondegenerate if $U$ is nondegenerate in $H$, $\det(A_m) \ne 0$ and $b_{kk} \notin \{\omega_1,\omega_2,\cdots\}$ for $k =m+1,\cdots,n.$
\end{itemize}}
When the semi-vector solution $\vec{u}$ is degenerate we also give a complete characterization of
the kernel of a linearized system \eqref{eq3} at $\vec{u}\in \R^nU$ in terms of the linearized kernel of $U$ and the eigenvectors corresponding to eigenvalue $0$ of $A$ in Theorem \ref{T1-1}. 
If a degeneracy occurs for a synchronized positive vector solution, we find a necessary condition for a bifurcation of positive vector solutions under a perturbation of the interaction matrix [Theorem \ref{B2}], describing the local structure of synchronized vector positive solutions.

Section 3 is devoted to Morse index results for systems of large interspecies force (namely the coupling constants $a_{ij}$ for $i\neq j$ are relatively larger than $a_{ii}$ for $i=1,...,n$). We establish optimal upper bounds on Morse index of synchronized vector solutions in Theorems \ref{C3} \ref{MM3}.
 
\vskip .1in

\noindent{\bf Variational construction of positive vector solutions.}
We establish a sufficient and necessary condition on the interaction matrix $A$ for a ground state solution to be a synchronized positive vector solution.
\\
\noindent{\bf Theorem \ref{V1}}.
{\it Assume that  $\lambda_1=\cdots=\lambda_n = \lambda.$ 
The following two assertions are equivalent.\\
\begin{itemize}
\item[(a)] The interaction matrix $A=(a_{ij})_{1\leq i,j\leq n}$ with $a_{ij} \ge 0$ has eigenvalues $\mu_1 > 0 \ge  \mu_2 \ge \cdots \ge \mu_n$ and the convex hull
of column vectors $\{\vec{a}_{i}\}_{i=1}^n$ with $\vec{a}_i = (a_{i1},\cdots,a_{in})$
contains an interior point $(e,\cdots,e)$ for some $e > 0;$\\
\item[(b)] The following minimization
 \[ M_\lambda \equiv \inf\Big \{I(\vec{u}) \ \Big | \ I^\prime(\vec{u})(\vec{u}) = 0 \ \text{ for } \ \vec{u} \in \SH \setminus\{0\} \Big \}\]
 is attained by  a positive vector solution $\vec{c} U \in \R^n U $ of \eqref{eq3} where $U$ is a least energy solution  of \eqref{singleeq}.
 \end{itemize}}
 While in Theorem \ref{V1} the synchronized positive vector solution is characterized as a global minimizer of a co-dimension $1$ manifold, we also prove in Theorem \ref{V2} that under the conditions of $A$ being nonsingular and $U$ being nondegenerate a synchronized positive vector solution $\vec{c} U$ is also a local minimizer of a co-dimension $n$ manifold.
 
 

Using the variational characterizations of the synchronized positive vector solutions in Theorems \ref{V1}, we give a new sufficient condition in Theorem \ref{LET} for the minimizer on the co-dimension $1$ manifold to be a positive vector solution of the more general systems \eqref{eq1} where the linear frequencies $\lambda_i'$s are not assumed to be the same. We point out that this is not a perturbation type result and contains some known existing results as special cases for the general existence of positive vector solution. \\
\noindent{\bf Theorem \ref{LET}}. {\it 
Suppose that the symmetric interaction matrix $A=(a_{ij})_{1\leq i,j\leq n}$ with $a_{ij} >  0$ has eigenvalues $\mu_1 > 0 >  \mu_2 \ge \cdots \ge \mu_n$ and
there exist $c_1,\cdots,c_n > 0$ satisfying
 \[A \begin{pmatrix}
    c_1^2\\ \vdots\\ c_n^2
\end{pmatrix} = \begin{pmatrix}
    1\\ \vdots\\ 1\end{pmatrix}.\]
Then, for $\lambda_1,\cdots,\lambda_n > 0,$ there exists a positive vector solution of \eqref{eq1}with the least energy among all nontrivial solutions
if  \begin{equation}\label{VarCondi}\Big(\frac{\max\{\lambda_1,\cdots,\lambda_n\}}{ \min\{\lambda_1,\cdots,\lambda_n\}}\Big)^{2-\frac{N}{2}}
< 1+ |\mu_2|\frac{\min_{1\le j\le n} \big(\sum_{i=1}^n(-1)^{i+j}\det(( a_{pq})_{p\neq i,q\neq j})\big)^2 }{\det(A)\sum_{i,j=1}^n(-1)^{i+j}\det(( a_{pq})_{p\neq i,q\neq j})}.\end{equation}
}
Theorem \ref{LET} is a generalized version of the previous results in  \cite{CFT,LLC,LWz2}, where the authors consider restricted classes of interaction matrices $A =(a_{ij}),$ typically $a_{ij} = a$ for $i\ne j$. 

\vskip .1in

\noindent{\bf Synchronization and nonsynchronization of positive vector solutions.} While the question on synchronization of positive vector solutions has been completely settled for the case of two equations in \cite{CLWY} and \cite{WY}, the question remains largely open for the case $n\geq 3$.
We establish a general result here. \\
\noindent{\bf Theorem \ref{SYN}}. {\it
Assume that a symmetric matrix $A=(a_{ij})_{1\leq i,j\leq n}$ has exactly one positive eigenvalue and satisfies $a_{ij} \ge 0$ for any $1\leq i,j\leq n$, $\sum_{j=1}^na_{ij} > 0$ for each $i=1,\cdots,n.$
Then, any positive vector solution $\vec{u} =(u_1,\cdots,u_n)$ of \eqref{eq3}  is synchronized.}

This in turn gives a uniqueness result of positive vector solutions when the scalar equation \eqref{singleeq} has a unique positive solution.
Also we have a general result on non-existence of vector positive solutions.

\noindent{\bf Corollary \ref{non}}. {\it 
Suppose that $A=(a_{ij})_{1\leq i,j\leq n}$ has only one positive eigenvalue and and satisfies $a_{ij} \ge 0$ for any $1\leq i,j\leq n$, $\sum_{j=1}^na_{ij} > 0$ for each $i=1,\cdots,n.$
Then, \eqref{eq3} has no positive vector solutions if there exists no vector $\vec{c}=(c_1,\cdots,c_n) \in \R^n$ satisfying $\sum_{j}a_{ij}c_j^2 =1 $ and $c_i \ne 0$ for each $i =1,\cdots,n.$}

The synchronization fails for general $n\ge2$ when the interaction matrix $A$ has  more than two positive eigenvalues as we see in the following (we refer the precise conditions and statement to Section 5).

\noindent{\bf Theorem \ref{NSYN}}. {\it 
Let $N=2,3$. For an interaction matrix $A=(a_{ij})_{1\le i,j\le n}$ in a dense subset of the class of nonsingular matrices with at least two positive eigenvlaues, there exists a  domain $\Omega\subset\R^N$  such that  \eqref{eq3}  has a nonsynchronized positive vector solution.}
\\

Throughout the paper, we use the following notations.
For two vectors $\vec{a}=(a_1,\cdots,a_n)$, $\vec{b}=(b_1,\cdots,b_n) \in \RN$
the Hadamard product ${\vec{a}}\odot{\vec{b}}$ is given by $(a_1b_1,\cdots,a_nb_n)$,
and the Hadamard division ${\vec{a}}\oslash{\vec{b}}$ is given by $(a_1/b_1,\cdots,a_n/b_n)$ whenever they make sense.

\section{Morse index and bifurcation of synchronized  solutions}
In this section, we assume that $U \in H$ is a  least energy positive solution of
\eqref{singleeq};
thus the Morse index of the solution $U$ is $1.$
The argument for the calculation of the Morse index for a solution 
$\vec{u} \in \R^n U$ in this section can be adapted to get the Morse index of a solution $\vec{u} \in \R^n U$ for any positive solution $U$ of \eqref{singleeq} with a higher Morse index.

\subsection{Morse index of synchronized positive vector solutions}
In this section, we aim to calculate the Morse index of a positive vector solution $\vec{u}\in\R^nU$ to the system \eqref{eq3} when the interaction matrix $A=(a_{ij})_{1\leq i,j\leq n}$ satisfies $a_{ij}\geq0$ for $1\leq i,j\leq n$.
We note that since $U$ is a positive solution of \eqref{singleeq}, 
$\vec{u}=(c_1,\cdots,c_n)U$ is a positive vector solution to the system \eqref{eq3} if and only if $c_1,\cdots,c_n>0$ and
\begin{equation} 
\label{Ac=1}
A
\begin{pmatrix}
c_1^2\\
c_2^2\\
\vdots\\
c_n^2
\end{pmatrix}
=
\begin{pmatrix}
1\\
1\\
\vdots\\
1
\end{pmatrix}.
\end{equation}
The equation \eqref{Ac=1} has a solution $(c_1,\cdots,c_n), c_i > 0$ if and only if
$(1,\cdots,1)$ is in the interior of the minimal convex cone containing column vectors of $A.$
If $\det(A) \ne 0,$ the solvability of the equation \eqref{Ac=1} is equivalent to
\[ \text{det}(A) \sum_{i=1}^nC_{ij}= \text{det}(A) \sum_{i=1}^nC_{ji} > 0, \ \ j = 1,\cdots,n,\]
where $C_{ij}$ is a cofactor defined by $(-1)^{i+j}\text{det}(( a_{pq})_{p\neq i,q\neq j})$.
The equivalence also holds without the assumption $a_{ij}\geq 0$ for $1\leq i,j\leq n$.
We prepare some Lemmas. The following is a result for symmetric matrices with nonnegative entries, which is similar to the  Perron-Frobenius theorem for matrices with positive entries.
\begin{lemma}
\label{M0}
Let $B=(b_{ij})_{1\leq i,j\leq n}$  be a nontrivial symmetric matrix satisfying $b_{ij}=b_{ji}\geq0$ for $1\leq i,j\leq n$. 
If there exist a positive eigenvalue $\nu$ and a corresponding eigenvector $(c_1,\cdots,c_n)$ of $B$ with $c_1,\cdots,c_n > 0,$ then
 \[ -\nu < \mu \le \nu \ \textup { for any eigenvalue } \ \mu \textup{ of } B.\] 
\end{lemma}
\begin{proof}
We define 
\[ \nu_{max} \equiv  \max \{\langle Bx,x\rangle \ | \ |x| =1\}.\]
Since $b_{ij} \ge0,$ $\nu_{max}$ is attained by a nonzero vector $(d_1,\cdots,d_n)$ in $\R^n$ with $d_i \ge 0.$
If $\nu_{max} \ne \nu,$ we see that $(c_1,\cdots,c_n)$ is orthogonal to 
$(d_1,\cdots,d_n)$; this is not possible since $c_i > 0, i=1,\cdots,n.$
Thus,   $\nu_{max} =\nu.$
Let $\mu$ be an eigenvalue of $B$.
It is obvious that $ \mu \le \nu.$
Suppose that $\mu < -\nu.$
For a unit eigenvector $x=(x_1,\cdots,x_n) \in \R^n$ of $\mu$,
let $y = (|x_1|,\cdots,|x_n|)$, which is also a unit vector.
Then we see that
\begin{equation} \label{maes} \nu < |\mu| = | \langle Bx,x \rangle | \le \langle By,y \rangle \le \nu. \end{equation}
This is a contradiction; thus $-\nu \le \mu \le \nu.$
Suppose that $\mu = -\nu.$
In this case, by a coordinate change, we may assume that for some $k \in \{1,\cdots,n-1\}$ and $l  \in \{k+1,\cdots,n\},$
\[ x_1,\cdots,x_k < 0 =x_{k+1} = \cdots = x_{l-1} < x_l,\cdots,x_n.\]
As in \eqref{maes}, we see that
\[ \nu = |\mu| = | \langle Bx,x \rangle | \le \langle By,y \rangle \le \nu. \]
This implies that $b_{ij} = 0$ for $i = 1,\cdots,k$ and $j=l,\cdots,n.$
Then, the first $k$ components of $Bx$ are not positive and the last $(n-l+1)$ are not negative. This contradicts that $Bx=-\nu x.$
Thus, we conclude that $ -\nu < \mu \le \nu.$
This completes the proof.
\end{proof}
We now prove a lemma that enables us to considerably simplify the calculation of the Morse index of $\vec{u}$ when the interaction matrix $A$ has nonnegative entries.
\begin{lemma}
\label{MI}
Let $\vec{u}\in\R^nU$ be a positive vector solution to a system \eqref{eq3} with an interaction matrix $A=(a_{ij})_{1\leq i,j\leq n}$ satisfying $a_{ij}=a_{ji}\geq0$ for $1\leq i,j\leq n$. Then, the Morse index of $\vec{u}$ equals the dimension of $\{\vec{\phi}\in\R^nU\ |\ I''(\vec{u})(\vec{\phi},\vec{\phi})<0\}.$
\end{lemma}
\begin{proof}

Let $\vec{u}=(c_1,c_2,\cdots,c_n)U$. The second derivative of the energy functional $I$ with respect to $\vec{\phi}=(\phi_1,\phi_2,\cdots,\phi_n)\in \SH$ is given by
\begin{align}
I''(\vec{u})(\vec{\phi},\vec{\phi})&=\lVert\vec{\phi}\rVert_{\Omega}^2-\int_{\Omega}U^2\sum_{i,j=1}^n a_{ij}(c_j^2\phi_i^2+2c_ic_j\phi_i\phi_j)dx\\
&=\lVert\vec{\phi}\rVert_{\Omega}^2-\int_{\Omega}U^2\sum_{i,j=1}^n\phi_i b_{ij}\phi_jdx,
\end{align}
where
\begin{equation} \label{matrixB}
b_{ij}=
\begin{cases}
    2a_{ii}c_i^2+\sum_{k=1}^n a_{ik}c_k^2 \ \ \ \mathrm{if} \ \ \ i=j,\\
    2a_{ij}c_ic_j \ \ \ \mathrm{if} \ \ \ i\neq j.
\end{cases}
\end{equation}
We define $B=(b_{ij})_{1\leq i,j\leq n}$. Then, \eqref{Ac=1} implies $b_{ii}=2a_{ii}c_i^2+1$ for $i=1,2,\cdots,n$, and
\begin{equation} \label{3ev}
(B-I)
\begin{pmatrix}
c_1\\
c_2\\
\vdots\\
c_n
\end{pmatrix}
=
\begin{pmatrix}
2a_{ij}c_ic_j
\end{pmatrix}_{1\leq i,j\leq n}
\begin{pmatrix}
c_1\\
c_2\\
\vdots\\
c_n
\end{pmatrix}
=
\begin{pmatrix}
2\sum_{j=1}^n a_{1j}c_1c_j^2\\
2\sum_{j=1}^n a_{2j}c_2c_j^2\\
\vdots\\
2\sum_{j=1}^n a_{nj}c_nc_j^2
\end{pmatrix}
=
\begin{pmatrix}
2c_1\\
2c_2\\
\vdots\\
2c_n
\end{pmatrix}.
\end{equation}
Therefore, number $3$ is an eigenvalue of $B$ with corresponding eigenvector $(c_1,c_2,\cdots,c_n)$, which is a vector with positive components. Since $B$ is a symmetric matrix with nonnegative entries, all other eigenvalues $\mu$ of $B$ must satisfy $-3 < \mu\leq 3$ by Lemma \ref{M0}.

Next, we decompose $\vec{\phi}$ as in \cite{BKS}. Let $\{\vec{\eta}^1,\vec{\eta}^2,\cdots,\vec{\eta}^n\}$ be an orthonormal system of eigenvectors corresponding to eigenvalues $\mu_1$, $\mu_2$, $...$, $\mu_n$ of $B$, respectively. Denote the $j$-th component of $\vec{\eta}^i$ as $\eta_j^i$, and define
\begin{equation}
\label{decompose}
\chi_i=\sum_{j=1}^{n}\eta_j^i\cdot \phi_j,\quad 1\leq i\leq n.
\end{equation}
Then, $\vec{\phi}$ is expressed as
\begin{equation}
\label{decomposephi}
\vec{\phi}=\sum_{i=1}^{n}\chi_{i}\vec{\eta}^i,
\end{equation}
so that
\begin{equation}
\label{decomposeresult}
\begin{split}
I''(\vec{u})(\vec{\phi},\vec{\phi})&=\sum_{i}\lVert\chi_i\rVert_{\Omega}^2-\int_{\Omega}U^2\cdot\sum_{i}\mu_i\chi_i^2dx\\
&=\sum_{i}\int_{\Omega}\lvert\nabla\chi_i\rvert^2+\lambda\chi_i^2-\mu_i U^2\chi_i^2dx
\end{split}
\end{equation}
holds. Therefore, the Morse index of $\vec{u}$ is the sum of the maximal dimension of a subspace $W_i\subset H$ satisfying
\begin{equation}
\label{wi}
\int_{\Omega}\lvert\nabla\chi\rvert^2+\lambda\chi^2-\mu_i U^2\chi^2dx<0\quad\mathrm{for}\quad\chi\in W_i\setminus\{0\},
\end{equation}
for $i=1,2,\cdots,n$. Now, since the Morse index of $U$ is 1, the maximal dimension of $W\subset H^1_0(\Omega)$ satisfying 
\begin{equation}
\int_{\Omega}\lvert\nabla\chi\rvert^2+\lambda\chi^2-3 U^2\chi^2dx<0\quad\mathrm{for}\quad\chi\in W\setminus\{0\}
\end{equation}
is 1. Since $U$ is the first eigenfunction of the linear eigenvalue problem $-\Delta w+ \lambda w=\tilde{\lambda} U^2w$ with eigenvalue $\tilde{\lambda}=1$, we see that
\begin{equation}
\label{UU}
\int_{\Omega}\lvert\nabla U\rvert^2+\lambda U^2-U^2\cdot U^2dx=0
\end{equation}
and
\begin{equation}
\label{Uchi}
\int_{\Omega}\lvert\nabla \chi\rvert^2+\lambda\chi^2-U^2\chi^2dx\geq0
\end{equation}
holds for any $\chi\in H$. Then we conclude that the maximal dimension of subspace $W_i\subset H$ satisfying \eqref{wi} is 1 if $1<\mu_i\leq 3$ and 0 if $\mu_i\leq 1$. This implies that the Morse index of $\vec{u}$ equals the number of eigenvalues $\mu$ of $B$ such that $1<\mu\leq 3$.

We now consider the maximal dimension of a subspace $W'\subset \R^nU$ satisfying
\begin{equation}
\label{w'}
I''(\vec{u})(\vec{\phi},\vec{\phi})<0\quad\mathrm{for}\quad\vec{\phi}\in W'\setminus\{0\}.
\end{equation}
By \eqref{decompose} and \eqref{decomposephi}, we have $\vec{\phi}\in \R^nU$ if and only if $\chi_i\in\R U$ for $i=1,2,\cdots, n$. From the decomposition \eqref{decompose}, we see that the maximal dimension of such $W'$ equals the sum of the maximal dimension of a subspace $W_i'\subset\R U$ satisfying
\begin{equation}
\label{wi'}
\int_{\Omega}\lvert\nabla\chi\rvert^2+\lambda\chi^2-\mu_i U^2\chi^2dx<0\quad\mathrm{for}\quad\chi\in W_i'\setminus\{0\},
\end{equation}
for $i=1,2,\cdots,n$. By \eqref{UU}, the maximal dimension of such $W_i'$ is 1 if $1<\mu_i\leq 3$ and 0 if $\mu_i\leq1$. Therefore, the maximal dimension of $W'\subset \mathbb{R}^nU$ satisfying \eqref{w'} equals the number of eigenvalues $\mu$ of $B$ such that $1<\mu\leq 3$. Thus, the Morse index of $\vec{u}$ equals the dimension of $\{\vec{\phi}\in\R^nU\ |\ I''(\vec{u})(\vec{\phi},\vec{\phi})<0\}.$
\end{proof}

We now establish a theorem that enables us to directly calculate the Morse index of a positive vector solution $\vec{u}\in\R^nU$ in terms of the eigenspaces of the interaction matrix $A$.
Since $A$ is symmetric, there are $n$ eigenvalues of $A$.
Since we consider a symmetric matrix $A$ with nonnegative entries, there exists at least one positive eigenvalue as far as $A \ne 0.$
For $r > 0,$  we define \[S_r^+ \equiv\{(c_1,\cdots,c_n) \in \R^n \ | \ c_1^2+\cdots + c_n^2 = r, c_i > 0, i = 1,\cdots,n\}.\]
\begin{thm} 
\label{T1}
Suppose that there exists a positive vector solution $\vec{u} =(c_1,c_2,\cdots,c_n)U =\vec{c}U \in\R^nU$ to a system \eqref{eq3} with an interaction matrix $A=(a_{ij})_{1\leq i,j\leq n}$ satisfying $a_{ij}=a_{ji}\geq0$ for $1\leq i,j\leq n$. Then the following conclusions hold.
\begin{itemize}
\item [(i)]
The Morse index of $\vec{u}$ equals the number of positive eigenvalues $\lambda_1 \ge \cdots \ge \lambda_k > 0$ of $A$, up to their multiplicity. 
Moreover,  a maximal negative space $W$ of $I^{\prime\prime}(\vec{u})$ satisfying 
\[
 I''(\vec{u})(\vec{\phi},\vec{\phi})<0\ \mathrm{for}\ \vec{\phi}\in W\setminus {\{0\} }
\]
is characterized by a linear span of
\[ \Big \{(\vec{\gamma} \oslash \vec{c})U \ | \  A\vec{\gamma} = \lambda_i \vec{\gamma}
 \ i  =1,\cdots,k\Big \}.\]
\item[(ii)] If $U$ is nondegenerate in $H$ and $\det (A) = 0$,
the linearized kernel of the system \eqref{eq3} at $\vec{u}\in \R^nU$ is given
by $\{ (\vec{\gamma} \oslash \vec{c})) U \ | \ A \vec{\gamma} = 0\};$ 
 in this case, for $K \equiv \{ \vec{\gamma}=(\gamma_1,\cdots,\gamma_n)  \ | \ A \vec{\gamma} = 0\} $,
 \[ L \equiv \big\{(d_1,\cdots,d_n) \ | \ (d_1,\cdots,d_n)U \in \R^n U \text{ is a positive vector solution of } \eqref{eq3} \big \} \]
is characterized by
\[
L = \big \{ \big( (c_1^2+\gamma_1)^{1/2}, \cdots, (c_n^2+\gamma_n)^{1/2}\big)\ |\  \vec{\gamma}\in K \text{ with } \ c_i^2+\gamma_i >0 , i=1,\cdots,n \big \}
\]
and is a connected smooth $\mathrm{dim}(K)$-dimensional manifold in $ S^+_{c_1^2+\cdots+c_n^2}$  with \[I(\vec{u}) =I((d_1,\cdots,d_n)U), \ \ \partial L\subset \partial S^+_{c_1^2+\cdots+c_n^2}. \]
\item [(iii)]
If $U$ is degenerate in $H$ and $\det (A) \ne 0$,
the linearized kernel of equation \eqref{eq3} at $\vec{u}\in\R^nU$ contains 
\[\{ (c_1,\cdots,c_n) \psi \ | \   \Delta \psi - \lambda\psi + 3U^2 \psi = 0, \psi \in H\}.\]
Moreover, if $U$ is degenerate in $H$,   $\det (A) \ne 0$
and  $a_{ij} > 0$ for all $1 \le i,j \le n,$
the linearized kernel of equation \eqref{eq3} at $\vec{u}\in\R^nU$ is spanned by 
\[\{ (c_1,\cdots,c_n) \psi \ | \   \Delta \psi - \lambda\psi + 3U^2 \psi = 0, \psi \in H\}.\]
\item[(iv)] The vector solution $\vec{u}$ is nondegenerate in $\SH$ if and only if $U$ is nondegenerate in $H$ and $\det (A) \neq 0$. 
\end{itemize}
\end{thm}
\begin{proof}
The second derivative of the energy functional $I$ at $\vec{u}=(c_1,c_2,\cdots,c_n)U$ in the direction of $\vec{\phi}=(\phi_1,\phi_2,\cdots,\phi_n)\in \SH$ is given by
\begin{align} \label{twoderi}
I''(\vec{u})(\vec{\phi},\vec{\phi})&=\lVert\vec{\phi}\rVert_{\Omega}^2-\int_{\Omega}U^2\sum_{i,j=1}^{n}a_{ij}(c_j^2\phi_i^2+2c_ic_j\phi_i\phi_j)dx\\
&=\lVert\vec{\phi}\rVert_{\Omega}^2-\int_{\Omega}U^2\left(\sum_{i=1}^n \phi_i^2+\sum_{i,j=1}^n c_i\phi_i\cdot 2a_{ij}\cdot c_j\phi_j\right)dx.
\end{align}
In particular, for any $\vec{\phi}\in \R^nU$, this expression is simplified to
\begin{equation}
\label{quadphi}
I''(\vec{u})(\vec{\phi},\vec{\phi})=-\int_{\Omega}U^2\sum_{i,j=1}^n c_i\phi_i\cdot 2a_{ij}\cdot c_j\phi_jdx.
\end{equation}
We now decompose $(c_1\phi_1,\cdots,c_n\phi_n)$ similar as in \cite{BKS}. The difference is that we decompose with respect to the eigenvectors of $A$, instead of those of $B$ given in \eqref{matrixB}. Let $\{\vec{\gamma}^1,\vec{\gamma}^2,\cdots,\vec{\gamma}^n\}$ be an orthonormal system of eigenvectors corresponding to eigenvalues $\lambda_1$, $\lambda_2$, $...$, $\lambda_n$ of $A$, respectively. Denote the $j$-th component of $\vec{\gamma}^i$ as $\gamma_j^i$. We define
\begin{equation}
\psi_i=\sum_{j=1}^{n}\gamma_j^i\cdot c_j\phi_j,\quad 1\leq i\leq n.
\end{equation}
Then, it holds that
\begin{equation}
I''(\vec{u})(\vec{\phi},\vec{\phi})=-\int_{\Omega}U^2\sum_i 2\lambda_i \psi_i^2dx.
\end{equation}
Note that $\vec{\phi}\in \R^nU$ if and only if $\psi_i\in\R U$ for $i=1,2,\cdots, n$. Then, we see from \eqref{quadphi} that a maximal space $W\subset \R^nU$ satisfying
\begin{equation} \label{negativespace}
 I''(\vec{u})(\vec{\phi},\vec{\phi})<0\ \mathrm{for}\ \vec{\phi}\in W\setminus\{0\}
\end{equation}
is characterized by the linear span of
\begin{equation} \label{maximalspace}
 \Big \{(\vec{\gamma}\oslash \vec{c})U \ | \ A\vec{\gamma} = \lambda \vec{\gamma},  \ \lambda  > 0\Big \}.\end{equation}
Thus the dimension of $W$ equals the number of positive eigenvalues of $A$. By Lemma \ref{MI}, the Morse index of $\vec{u}$ equals the number of positive eigenvalues of $A$, up to their multiplicity. This proves $(i)$

We now move on to the linearized kernel of the solution $\vec{u}$. The linearized system of \eqref{eq3} at $\vec{u}$ is given by
\begin{equation}
\label{linsys}
\Delta\varphi_i-\lambda\varphi_i+U^2\left(\sum_{j=1}^{n}a_{ij}c_j^2\varphi_i+\sum_{j=1}^n 2a_{ij}c_ic_j\varphi_j\right)=0\ \mathrm{for}\ 1\leq i\leq n,\ \varphi_i\in H.
\end{equation}
We define 
\begin{equation}
b_{ij}=
\begin{cases}
    2a_{ii}c_i^2+\sum_{k=1}^n a_{ik}c_k^2\ \mathrm{if}\ i=j,\\
    2a_{ij}c_ic_j\ \mathrm{if}\ i\neq j.
\end{cases}
\end{equation}
and a matrix  $B=(b_{ij})_{1\leq i,j\leq n}$. 
Writing in terms of the matrix $B$, \eqref{linsys} is rewritten as
\begin{equation}
\Delta\varphi_i-\lambda\varphi_i+U^2\left(\sum_{j=1}^n b_{ij}\varphi_j\right)=0\ \mathrm{for}\ 1\leq i\leq n,\quad \varphi_i\in H.
\end{equation}
As in the proof of Lemma \ref{MI},
let $\{\vec{\eta}^1,\cdots,\vec{\eta}^n\}$ be an orthonormal system of eigenvectors corresponding to eigenvalues $\mu_1$, $\mu_2$, $...$, $\mu_n$ of $B$, respectively. Denote the $j$-th component of $\vec{\eta}^i$ as $\eta_j^i$, and define
\begin{equation}
\chi_i=\sum_{j=1}^{n}\eta_j^i\cdot \phi_j,\quad 1\leq i\leq n.
\end{equation}
Since 
$\vec{\phi}=\sum_{i=1}^{n}\chi_{i}\vec{\eta}^i,$
$\vec{\phi}$ is a solution of \eqref{linsys} if and only if
\begin{equation}
\label{sle}
\Delta\chi_i-\lambda\chi_i+\mu_i U^2\chi_i=0\ \mathrm{for}\ 1\leq i\leq n,\quad \chi_i\in H,
\end{equation}
where $\{\mu_i\}$ are eigenvalues of $B$.
By \eqref{3ev} and Lemma \ref{M0}, we see that $-3 < \mu_i \le 3.$ 
Since \eqref{sle} has the first eigenvalue $1$ and a corresponding positive eigenfunction $U$, we see that if $\mu_i < 1,$ $\chi_i = 0.$
Since $U$ is a least energy solution of \eqref{singleeq},
$\mu_i =1$ or $3$ if $\mu_i \in [1,3].$
Since $B-I=(2a_{ij}c_ic_j)_{1\leq i,j\leq n}$, $1$ is an eigenvalue of $B$ if and only if \[\det (A)=(2^nc_1^2c_2^2\cdots c_n^2)^{-1}\det(B-I)=0.\]
Since $3$ is an eigenvalue of $B$, and $\mu_i=3$ is an eigenvalue of \eqref{sle} if and only if $U$ degenerates in $H$. 
Thus, if $ \det (A) = 0$ and  $U$ does not degenerate in $H$, 
since $\vec{\phi}=\sum_{i=1}^{n}\chi_{i}\vec{\eta}^i$ and the eigenvectors of $B$ corresponding to the eigenvalue $1$ relates to the kernel of $A$, 
the linearized kernel at $\vec{u}$ is given by
\[\Big\{ (\vec{\gamma}\oslash \vec{c}) U \ \Big | \ A \vec{\gamma} = 0\Big \}.\]
To characterize the set $L,$ we note that
$d_1,\cdots,d_n>0$. Then, $(d_1,\cdots,d_n)U$ is a positive vector solution of \eqref{eq3} if and only if
$$
A\begin{pmatrix}
d_1^2\\ \vdots \\ d_n^2
\end{pmatrix}=\begin{pmatrix}
1\\ \vdots\\ 1
\end{pmatrix}=A\begin{pmatrix}
c_1^2\\ \vdots\\ c_n^2
\end{pmatrix}.
$$
This condition is equivalent to
$$
\begin{pmatrix}
d_1^2-c_1^2\\ \vdots \\ d_n^2-c_n^2
\end{pmatrix}=\begin{pmatrix}
k_1\\ \vdots\\ k_n
\end{pmatrix}
$$
for some $(k_1,\cdots,k_n)\in K$.
In this case, for $t \in [0,1]$, $(\sqrt{c_1^2+tk_1},\cdots,\sqrt{c_n^2+tk_n})U$ is a positive vector solution \eqref{eq3}. Thus $L$ is connected and we get the characterization of $L$.
Furthermore, since $A$ is symmetric, $K$ is perpendicular to the image of $A$. In particular, $K$ is perpendicular to $(1,\cdots,1)$. This implies
$$
\sum_{i=1}^{n}(c_i^2+k_i)=\sum_{i=1}^{n}c_i^2
$$
for any $(k_1,\cdots,k_n)\in K$ so that $L\subset S^+_{c_1^2+\cdots+c_n^2}$.
The energy property $I((d_1,\cdots,d_n)U) = I(\vec{u})$ follows from the fact that $I^\prime((d_1,\cdots,d_n)U) = 0$ for any $(d_1,\cdots,d_n) \in L$. 
Lastly, for $(c_1,\cdots,c_n) \in L$ and $(k_1,\cdots,k_n) \in K,$
$(c_1+tk_1,\cdots,c_n+tk_n) \in L$ for small $|t| \ge 0;$ thus $\partial L \subset \partial S^+_{c_1^2+\cdots+c_n^2}$.
This completes the proof $(ii)$.

If $ \det (A) \ne 0$ and  $U$  is degenerate in $H$, 
since $(c_1,\cdots,c_n)$ is the eigenvector of $B$ corresponding to the eigenvalue $3,$
the linearized kernel at $\vec{u}$ contains
 \[\{ (c_1,\cdots,c_n) \psi \ | \   \Delta \psi -\lambda \psi + 3U^2 \psi = 0, \psi \in H\}.\]
If we further assume that $a_{ij} > 0$ for any $i,j \in \{1,\cdots,n\},$
we see from the Perron-Frobenius theorem that the maximum eigenvalue $3$ of $B$ is simple; in this case,
the linearized kernel at $\vec{u}$ is given by
 \[\{ (c_1,\cdots,c_n) \psi \ | \   \Delta \psi - \lambda \psi + 3U^2 \psi = 0, \psi \in H\}.\]
This proves $(iii).$

The last assertion $(iv)$ follows from the results and arguments in $(i)$,$(ii)$,$(iii).$
\end{proof}

By Theorem \ref{T1}, it is immediate that the Morse index of a positive vector solution $\vec{u}\in\R^nU$ is at most $n$. Furthermore, as we see in the following Theorem, for any $1\leq m\leq n$, we can find a matrix $A$ with symmetric, nonnegative entries such that $\det (A)\neq 0$ and there exists a positive vector solution $\vec{u}\in\R^nU$ to a system \eqref{eq3} of which the Morse index equals $m$.

\begin{thm}
\label{T2}
Let $1\leq m\leq n$. Then, there exists a matrix $A=(a_{ij})_{1\leq i,j\leq n}$ satisfying $a_{ij}=a_{ji}\geq 0$ for $1\leq i,j\leq n$ such that $\det (A)\neq 0$ and there exists a positive vector solution $\vec{u}\in\R^nU$ to a system \eqref{eq3} which the Morse index of $\vec{u}$ equals $m$. 
\end{thm}

\begin{proof}
First, we consider the case where $m=n$. In this case, let $A$ be the $n\times n$ identity matrix. Then, $\det (A)\neq 0$ and $(1,1,\cdots,1)U\in\R^nU$ is a positive vector solution to a system \eqref{eq3}. Furthermore, the Morse index of $\vec{u}$ equals $n$ by Theorem \ref{T1}.

From now on, let $1\leq m<n$. Consider an matrix
\begin{equation}
A=
\begin{pmatrix}
    I_{m-1} & 0_{(m-1)\times (n-m+1)}\\
    0_{(n-m+1)\times (m-1)} & 1_{(n-m+1)\times (n-m+1)}-I_{n-m+1}
\end{pmatrix}
\end{equation}
where $I_p$ is the $p\times p$ identity matrix, $1_{p\times q}$ is the $p\times q$ matrix where all the terms are 1, and $0_{p\times q}$ is the $p\times q$ matrix where all the terms are 0. 
Then, 
\begin{equation}
A
\begin{pmatrix}
    1_{(m-1)\times 1}\\
    \frac{1}{n-m}\cdot 1_{(n-m+1)\times 1}
\end{pmatrix}
=
\begin{pmatrix}
    1\\
    1\\
    \vdots\\
    1
\end{pmatrix}
\end{equation}
holds, which is simply \eqref{Ac=1} with $c_1=c_2=\cdots=c_{m-1}=1>0$ and $c_m=c_{m+1}=\cdots=c_n=(n-m)^{-1/2}>0$. Therefore, $\vec{u}=(c_1,c_2,\cdots,c_n)U\in\R^nU$ is a positive vector solution to a system \eqref{eq3}.

Now, let us denote the $i$-th unit vector in $\R^n$ as $e_i$. Then, $e_1,e_2,\cdots,e_{m-1}$ are $(m-1)$ eigenvectors of $A$ with eigenvalues 1. Also, $e_{m}-e_{m+1},e_{m}-e_{m+2},\cdots,e_{m}-e_{n}$ are $(n-m)$ eigenvectors of $A$ with eigenvalues $-1$. Finally, $e_{m}+e_{m+1}+\cdots+e_{n}$ is an eigenvector of $A$ with eigenvalue $(n-m)$. There are no other eigenvectors and $\det (A)\neq 0$ since 0 is not an eigenvalue of $A$. Now we see that there are exactly $m$ positive eigenvalues of $A$. By Theorem \ref{T1}, the Morse index of $\vec{u}$ is $m$, which completes the proof.
\end{proof}

\begin{remark}
If $a_{ij} = 1$ for any $1 \le i,j\le n,$ we see that $L=S^+_{1}.$
If $n=4$ and 
\[ A\equiv (a_{ij}) = \begin{pmatrix}
0 & 1 & 1/2 &1/2 \\
1  & 0 & 1/2 &1/2 \\
1/2 &1/2 & 0 & 1 \\
1/2 & 1/2 & 1 & 0
\end{pmatrix}, \]
 $L= \{(\cos \theta, \cos \theta,\sin \theta,\sin \theta) \ | \ \theta \in (0,\pi/2) \}$ is an open subset of a circle.
If $n=4$ and 
\[ A \equiv (a_{ij}) = \begin{pmatrix}
0 & 0 & 1 &1 \\
0  & 0 & 1 & 1 \\
1 & 1 & 0 & 0 \\
1 & 1 & 0 & 0
\end{pmatrix}, \]
$L= \{(\cos \theta, \sin \theta,\cos \theta^\prime,\sin \theta^\prime) \ | \ \theta, \theta^\prime \in (0,\pi/2) \}$ is an open subset of a torus
\end{remark}

\subsection{Morse index of semi-vector nonnegative solutions}
    We now consider semi-vector nonnegative solutions. Let $m<n$. For positive real numbers $c_1$, $c_2$, ..., $c_m$, assume that $\vec{u}=(u_1,u_2,\cdots,u_n)=(c_1,c_2,\cdots,c_m,0,0,\cdots,0)U$ is a nonnegative solution of \eqref{eq3} with interaction matrix $A=(a_{ij})_{1\leq i,j\leq n}$. Let $A_m=(a_{ij})_{1\leq i,j\leq m}$, and $c_{m+1}=c_{m+2}=\cdots=c_n=0$. For $\vec{\phi}=(\phi_1,\phi_2,\cdots,\phi_n)\in \SH$, we have
\begin{align}
I''(\vec{u})(\vec{\phi},\vec{\phi})&=\lVert\vec{\phi}\rVert_{\Omega}^2-\int_{\Omega}U^2\sum_{i,j=1}^n a_{ij}(c_j^2\phi_i^2+2c_ic_j\phi_i\phi_j)dx\\
&=\lVert\vec{\phi}\rVert_{\Omega}^2-\int_{\Omega}U^2\sum_{i,j=1}^n \phi_i b_{ij}\phi_jdx,
\end{align}
where
\begin{equation}
\label{B}
b_{ij}=
\begin{cases}
    2a_{ii}c_i^2+\sum_{k=1}^n a_{ik}c_k^2\ \mathrm{if}\ i=j,\\
    2a_{ij}c_ic_j\ \mathrm{if}\ i\neq j.
\end{cases}
\end{equation}
Let us define $B=(b_{ij})_{1\leq i,j\leq n}$. Then, $b_{ij}=0$ for all $i\neq j$ satisfying either $i>m$ or $j>m$. As before,
since $\Delta (c_iU)-\lambda c_iU=-c_iU^3$, it holds that
\begin{equation}
\label{Bm1}
A_m
\begin{pmatrix}
c_1^2\\
c_2^2\\
\vdots\\
c_m^2
\end{pmatrix}
=
\begin{pmatrix}
1\\
1\\
\vdots\\
1
\end{pmatrix}.
\end{equation}
This implies $b_{ii}=2a_{ii}c_i^2+1$ for $i=1,2,\cdots,m$. Therefore, $B$ is of form
\begin{equation}
\label{Bform}
    B=
    \begin{pmatrix}
        B_m& O\\
        O^T& \tilde{B}
    \end{pmatrix}
\end{equation}
where $B_m$ is a $m\times m$ matrix, $\tilde{B}$ is a $(n-m)\times (n-m)$ diagonal matrix, and $O$ is the $m\times (n-m)$ zero matrix.
Let $\omega_1 = 1 < \omega_2 \le \cdots$ be the eigenvalues of multiplicity one with the corresponding eigenfunctions $\{\psi_i\}_{i=1}^\infty$ satisfying 
   \begin{equation}
    \label{speceq}
    \Delta\psi_i-\lambda \psi_i+\omega_i U^2\psi_i=0,\quad  \psi_i\in H.
    \end{equation}
 Then, we get the following result for the semi-vector nonnegative solution $\vec{u}$.
\begin{thm}
\label{T1-1}
    Let $c_1,c_2,\cdots,c_m>0$ and $\vec{u}=(c_1,c_2,\cdots,c_m,0,\cdots,0)U\in\R^nU$ be a semi-vector nonnegative solution to a system \eqref{eq3} with the interaction matrix $A=(a_{ij})_{1\leq i,j\leq n}$ satisfying $a_{ij}=a_{ji} \ge 0 $ for $1\leq i,j\leq n$. Define $B=(b_{ij})_{1\leq i,j\leq n}$ by \eqref{B}, $A_m=(a_{ij})_{1\leq i,j\leq m}$, $B_{m}=(b_{ij})_{1\leq i,j\leq m}$, and $f(x) = \max\{i \in \mathbb{N} \ | \ \omega_i < x\},$ where
 $   f(x) = 0$ if $\{i \in \mathbb{N} \ | \ \omega_i < x\} = \emptyset.$
Then, the Morse index of $\vec{u}$ equals
    \begin{equation}\label{eigenformula}
       the \  number \ of \ the \ positive \ eigenvalues \ of A_m +\sum_{i=m+1}^n f(b_{ii}).
    \end{equation}
If $U$ is nondegenerate in $H$,
the linearized kernel of a system \eqref{eq3} at $\vec{u}\in\R^nU$ is spanned
by $\{ (\frac{\gamma_1}{c_1},\cdots,\frac{\gamma_m}{c_m},0\cdots,0) U \ | \ A_m (\gamma_1,\cdots,\gamma_m) = 0\}$ and  
\[E_k \equiv \{ e_k\psi_{i(k)} \ | \  b_{kk} = \omega_{i(k)} \},   k =m+1,\cdots,n,\]
where $\{e_k\}_1^n$ is the standard basis of $\R^n$ and we define
$E_k$ to be empty if $b_{kk} \notin \{\omega_1,\omega_2\cdots \}$;
thus, $\vec{u}$ is nondegerate if $U$ is nondegenerate in $H$, $\det(A_m) \ne 0$ and $b_{kk} \notin \{\omega_1,\omega_2,\cdots\}$ for $k =m+1,\cdots,n.$

\end{thm}

\begin{proof}
Let $\{\vec{\eta}^1,\vec{\eta}^2,\cdots,\vec{\eta}^m\}$ be an orthonormal system of eigenvectors corresponding to eigenv-alues $\mu_1$, $...$, $\mu_m$ of $B_m$. 
By \eqref{Bform},  $e_k$ is the eigenvector of $B$ with an eigenvalue $b_{kk}$ for $k=m+1,\cdot,n.$  We define $(\mu_k,\vec{\eta}^k) = (b_{kk},e_k)$ for $k=m+1,\cdots,n.$

Denote the $j$-th component of $\vec{\eta}^i$ as $\eta_j^i$, and define
\begin{equation*}
\chi_i=\sum_{j=1}^{n}\eta_j^i\cdot \phi_j,\quad 1\leq i\leq n.
\end{equation*}
Then, $\vec{\phi}$ is expressed as
\begin{equation*}
\vec{\phi}=\sum_{i=1}^{n}\chi_{i}\vec{\eta}^i,
\end{equation*}
so that
\begin{equation*}
\begin{split}
I''(\vec{u})(\vec{\phi},\vec{\phi})&=\sum_{i}\lVert\chi_i\rVert_{\Omega}^2-\int_{\Omega}U^2\cdot\sum_{i}\mu_i\chi_i^2dx\\
&=\sum_{i}\int_{\Omega}\lvert\nabla\chi_i\rvert^2+\lambda \chi_i^2-\mu_i U^2\chi_i^2dx
\end{split}
\end{equation*}
holds. Therefore, the Morse index of $\vec{u}$ is the sum of the maximal dimension of a subspace $W_i\subset H$ satisfying
\begin{equation*}
\int_{\Omega}\lvert\nabla\chi\rvert^2+\lambda\chi^2-\mu_i U^2\chi^2dx<0\quad\mathrm{for}\quad\chi\in W_i\setminus\{0\},
\end{equation*}
for $i=1,2,\cdots,n$, which equals $\sum_{i=1}^nf(\mu_i).$
By the same argument as the proof of Theorem \ref{T1}, we see that
$\sum_{i=1}^mf(\mu_i)$ is the number  of  the  positive  eigenvalues of $A_m.$
This proves \eqref{eigenformula}.

Suppose that $U$ is nondegenerate.
An element $(\varphi_1,\cdots,\varphi_n) \in \SH$ in the linearized kernel of a system \eqref{linsys} at $\vec{u} \in  \R^n U$ satisfies
\begin{equation}
\label{linsys2}
\Delta\varphi_i-\lambda\varphi_i+U^2\left(\sum_{j=1}^n b_{ij}\varphi_j\right)=0\ \mathrm{for}\ 1\leq i\leq n,\quad \varphi_i\in H.
\end{equation}
Defining
\begin{equation}
\chi_i=\sum_{j=1}^{n}\eta_j^i\cdot \phi_j,\quad 1\leq i\leq n,
\end{equation}
we see that 
$\vec{\phi}=\sum_{i=1}^{n}\chi_{i}\vec{\eta}^i$
is a solution of \eqref{linsys2} if and only if
\begin{equation*}
\Delta\chi_i-\lambda\chi_i+\mu_i U^2\chi_i
= 0\ \mathrm{for}\ 1\leq i\leq n,\quad \chi_i\in H,
\end{equation*}
where $\mu_1,\cdots,\mu_m$ are eigenvalues of $B_m$
and $\mu_i = \sum_{j=1}^m a_{ij}c_j^2$ for $i=m+1,\cdots,n.$
For $i=1,\cdots,m,$ as in the proof of Theorem \ref{T1}, we see that
\eqref{sle} has a solution $U$ only when $\mu_i = 1.$
Since $B_m-I=(2a_{ij}c_ic_j)_{1\leq i,j\leq m}$, $1$ is an eigenvalue of $B_m$ if and only if \[\det (A)_m=(2^nc_1^2c_2^2\cdots c_m^2)^{-1}\det(B_m-I)=0.\]
Thus, the space 
$\{ (\gamma_1,\cdots,\gamma_m,0,\cdots,0) U \ | \ A_m (c_1\gamma_1,\cdots,c_m\gamma_m) = 0\}$
 is in the linearized kernel of a system \eqref{eq3} at $\vec{u}$.
For $k \in \{m+1,\cdots,n\},$ a space $E_k \equiv \{ e_k\psi_{i} \}$ is 
in the kernel only when
$b_{kk} = \omega_{i}.$
This proves the characterization of the kernel of a system \eqref{eq3} at $\vec{u}$
The linearized nondegeneracy of $\vec{u}$ follows from the characterization when
$U$ is nondegenerate in $H$, $\det(A_m) \ne 0$ and $b_{kk} \notin \{\omega_1\cdots\}$ for $i =m+1,\cdots,n.$
\end{proof}
  \begin{remark}
   We can generate pairs of the interaction matrix $A$ and nonnegative vector solution $\vec{u}$ with large Morse indices. For example, let
    \begin{equation}
        A=\begin{pmatrix}
            1& 0& a\\
            0& 1& 0\\
            a& 0& 0
        \end{pmatrix}
    \end{equation}
    and $\vec{u}=(1,1,0)U$. Then, if we define $B=(b_{ij})_{1\leq i,j\leq n}$ by \eqref{B}, in this case, it holds that $b_{33}=a$. Therefore, by Theorem \ref{T1-1},  the Morse index of $\vec{u}$ also goes to $+\infty$ as $a$ goes to $+\infty$.
If $a=1,$ the linearized kernel of a system \eqref{eq3} at $(1,1,0)U$ is spanned 
by $(0,0,1)U$. Thus, there is a simple bifurcation from $(1,1,0)U$ for $a$ close  to $1$ in the direction of $(0,0,1)U.$
In fact, we have a global solution curve $\{(\frac{1}{\sqrt{a}},1,\frac{\sqrt{a-1}}{a})U \ | \ a \in [1,\infty)\}$ bifurcating from $(1,1,0)U$ in the direction of $(0,0,1)U.$
We could characterize such a phenomenon in a simple case as in Theorem \ref{B2}. It would be interesting to see a bifurcation phenomena in the general case when we have a linearized degeneracy of a  semi-vector nonnegative  solution of \eqref{eq3}.
\end{remark}

\subsection{Bifurcation from the positive vector solutions}

We see from Theorem \ref{T1} that if $U$ is nondegenerate and $\det(A) \ne 0$, any positive vector solution $\vec{u} \in \R^n U$ continues to exist under perturbation of $A$ and no other solutions close to $\vec{u}$ exist.
 When $\det(A) = 0$ and there exists a positive vector solution $\vec{u} \in \R^nU$ of \eqref{eq3}, $(ii)$ of Theorem \ref{T1} says that there is a continuum $L$ of positive vector solutions containing $\vec{u}$ in $\R^nU$.  
To see whether there could be a bifurcation of positive vector solutions
under a perturbation of $A,$ 
we consider an example $a_{ij} =1$ for $1\le i,j \le 2.$
Then, \[L = \{(\cos\theta,\sin\theta)U \ | \ \theta \in (0,\pi/2)\}, K = \{(a,-a) \ | \ a \in \R\}.\] 
For a perturbation $A^\e =A+\e E$ of 
$A = \begin{pmatrix}
1 & 1  \\
1 & 1 
\end{pmatrix}$  by $E  \equiv  \begin{pmatrix}
\alpha & 0  \\
0 & \beta 
\end{pmatrix},$
it is well known that if $\alpha \beta \le 0$ and $\alpha^2+\beta^2 \ne 0,$
there exist no corresponding positive vector solutions of \eqref{eq3} with $A^\e$ replacing $A$.
On the other hand, if $\alpha\beta > 0,$
we get a corresponding positive vector solution 
\[\vec{c}_\e U \equiv \Big (\sqrt {\frac{\beta}{\alpha+\beta+\e\alpha\beta}},
\sqrt{\frac{\alpha}{\alpha+\beta+\e\alpha\beta}} \Big )U \ \ \ \textup{ for small } |\e|  \ge  0.\]
Thus, we see that  for a perturbation $A^\e$ of $A$, we have a bifurcation from a positive vector solution $\Big (\sqrt {\frac{\beta}{\alpha+\beta}},
\sqrt{\frac{\alpha}{\alpha+\beta}} \Big )U$  for small $|\e| $ if and only if $\alpha\beta > 0.$

We study this phenomenon in general cases.
For a perturbation of the interaction matrix $A$ by any symmetric matrix $ E =\{e_{ij}\}_{1\le ij, \le n},$
the following is a key function 
\[G(d_1^2,\cdots,d_n^2) = \sum_{1 \le i,j\le n} e_{ij}d_i^2d_j^2, \ \ (d_1,\cdots,d_n) \in L.\]
We assume that the condition in $(ii)$ of Theorem \ref{T1} holds and
 \[\vec dU=(d_1,\cdots,d_n)U \in \R^nU\]
is a positive vector solution of \eqref{eq3}.
Since
\[
L = \big \{ \big( (d_1^2+\gamma_1)^{1/2}, \cdots, (d_n^2+\gamma_n)^{1/2}\big)\ |\  \vec{\gamma}\in K \text{ with } \ d_i^2+k_i >0 , i=1,\cdots,n \big \},
\]
we may regard the function $G$ is defined on a neighborhood of $0$ in $K.$
Then, for $\vec{\gamma} = (\gamma_1,\cdots,\gamma_n),\vec{\eta}=(\eta_1,\cdots,\eta_n) \in K$, the gradient and Hessian of $G$ on $K$ are expressed by
\[ \nabla_K G(d_1^2,\cdots,d_n^2)(\vec{\gamma}) = 2\sum_{1 \le i,j\le n} e_{ij}\gamma_i d_j^2\]
and
\[ \nabla^2_K G(d_1^2,\cdots,d_n^2)(\vec{\gamma},\vec{\eta})=2\sum_{1\le i,j\le n} e_{ij}\gamma_i\eta_j. \]
Therefore, if there exists $\vec{d} = (d_1,\cdots,d_n) \in L$ such that for any $\vec{\gamma}\in K\setminus\{0\}$, $E\vec{\gamma} \perp (d_1^2,\cdots,d_n^2)$ but there exists $\vec{\eta}\in K$ such that $E\vec{\gamma}\cdot\vec{\eta}\neq 0$, then we see that \[ \nabla_K G(d_1^2,\cdots,d_n^2) = 0 \ \text{  and } \ \nabla^2_K G(d_1^2,\cdots,d_n^2)  \ \text{ is nondegenerate. } \]
 In the sense of a bifurcation of positive vector solutions for a perturbation of $A$ by  $A+\e E,$ we need the following  definition.
\begin{definition}
For a symmetric matrix $A$ satisfying the assumptions in $(ii)$ of Theorem \ref{T1} with a corresponding $k$-dimensional kernel $K$ of $A$ and a connected smooth $k$-dimensional manifold 
$L \subset S^+_r$ with $r =\sqrt{c_1^2+\cdots+c_n^2}$,
we say that a symmetric matrix 
$E =\{e_{ij}\}_{1\le i, j \le n}$ is complementary to $A$ at $\vec{d} \in L$ if  for any $\vec{\gamma}\in K\setminus\{0\}$, $E\vec{\gamma} \perp (d_1^2,\cdots,d_n^2)$ but there exists $\vec{\eta}\in K$ such that $E\vec{\gamma}\cdot\vec{\eta}\neq 0$. The point $\vec{d} = (d_1,\cdots,d_n) \in L$ is called a complementary point to $A$ and $E$.
\end{definition}

\begin{thm} \label{B2}
Suppose that $a_{ij} = a_{ji} \ge 0$ for $1\le i, j \le n$, a solution $U$ of \eqref{singleeq} is nondegenerate in $H$,  the kernel  $K$ of $A$ is $k$-dimensional with $k \in \{1,\cdots,n-1\}$ and there  exists a positive vector solution $(c_1,\cdots,c_n)U \in \R^n U$ of \eqref{eq3}.
Let $L \subset S^+_r$ with $r =\sqrt{c_1^2+\cdots+c_n^2}$ be the connected smooth $k$-dimensional manifold given in Theorem \ref{T1}-(ii).
Let  $E = (e_{ij})_{1 \le i,j\le n}$ be a symmetric matrix complementary to $A$ at $\vec{c}$ and define
$A^E_\e=A+\e E$.
Then, for any small $\e > 0,$ there exists a unique positive vector solution $\vec{u}_\e$ of \eqref{eq3} with $A$ replaced by $A^E_\e$ near $\vec{c}$U such that $\vec{u}_\e \in\R^nU$  and it converges to the positive vector solution $\vec{c}U$ as $\e\rightarrow0$.
\end{thm}
\begin{proof}
For convenience, let us denote the $i$-th coordinate of a vector $\vec{w}$ by $w_i$ in this proof. Recall that we defined the Sobolev space $\SH=(H^1_0(\Omega))^n$ if $\Omega$ is a bounded domain in $\R^N$. 
If $\Omega = \R^N,$ the space $\SH$ was defined by the radially symmetric functions in $(H^1(\R^N))^n$.
We define a map $S_\e : \SH \to \SH^*=\SH$
by 
$$
\langle S_{\e}(\vec{u}),\vec{v}\rangle_{\SH}=\int_{\Omega}\sum_{i,j=1}^n(a_{ij}+\e e_{ij})u_iv_iu_j^2dx.
$$
This map is well-defined by the continuous embedding $\SH\to(L^4(\Omega))^n$ and the Riesz representation theorem. We will solve \[\vec{u}-S_{\e}(\vec{u})=(0,\cdots,0),\] which is equivalent to solving \eqref{eq3} with $A$ replaced by $A^E_\e$. We apply the Lyapunov-Schmidt reduction to solve $\vec{u}-S_{\e}(\vec{u})=(0,\cdots,0)$ near $\vec{c}U$. The directional derivative of $S_{\e}(\vec{u})$ in the direction of $\vec{\varphi}=(\varphi_1,\cdots,\varphi_n) \in \SH$ is characterized by
\begin{equation}
\label{linearS}
\langle S'_{\e}(\vec{u})(\vec{\varphi}),\vec{v}\rangle_{\SH}=\int_{\Omega}\sum_{i,j=1}^n(a_{ij}+\e e_{ij})(\varphi_iv_iu_j^2+2u_iv_iu_j\varphi_j)dx.
\end{equation}
Let $K$ be the kernel of $A$ and $K^\perp$ the orthogonal complement of $K$ in $\R^n.$
We define 
\[ M \equiv \{\vec{d} U \in \SH \ | \ \vec{d}\in L\}.\]
For each $\vec{d} U  \in M,$ we define  the tangent space $T_{\vec{d}}M$ of $M$ at $\vec{d}U$ in $\SH$, which we could characterize by Theorem \ref{T1} 
\[ T_{\vec{d}}M \equiv \{(\vec{\gamma}\oslash \vec{d}) U \ | \ \vec{\gamma} \in K\}.\]
Let $ M_{\vec{d}}^\perp$ be the orthogonal complement of $T_{\vec{d}}M$ in $\SH$, that is,
\[M_{\vec{d}}^\perp\equiv\left\{\vec{u}\in\SH\ |\ \langle\vec{u},(\vec{\gamma}\oslash \vec{d}) U\rangle_{\SH}=0 \text{ for all } \vec{\gamma}\in K\right\}\]
and let $\pi_{\vec{d}}^\perp$ be the orthogonal projection of $\SH$ onto $M_{\vec{d}}^\perp$.

For each $\vec{d} \in L,$  we now define $\mathscr{L}_{\vec{d},\e}:M_{\vec{d}}^\perp\to M_{\vec{d}}^\perp$ by
$$
\mathscr{L}_{\vec{d},\e}(\vec{\varphi})=\vec{\varphi}-\pi_{\vec{d}}^\perp(S'_{\e}(\vec{d}U)(\vec{\varphi}))$$
and show that for each open neighborhood $L^\prime$ of  $\vec{c}$ in $L$ with $\overline{L^\prime} \subset L,$  there exist  $\e_0 > 0$ and $\gamma> 0$  such that for any $|\e| \leq \e_0$, $\vec{d}\in L^\prime$  and $\vec{\varphi}\in M_{\vec{d}}^\perp$, 
\begin{equation} \label{invertable}
\lVert \mathscr{L}_{\vec{d},\e}\vec{\varphi}\rVert_{\SH}\geq\gamma\lVert\vec{\varphi}\rVert_{\SH}.
\end{equation}
Suppose that it is not true. Then, there exist sequences 
$\{\vec{d}_l = (d^l_1,\cdots,d_n^l)\}_l \subset L^\prime$,
$\{\e_l\} \subset \R$ with $\lim_{l \to \infty}\e_l = 0$ and a sequence $\vec{\varphi}_l= (\varphi^l_1,\cdots,\varphi_n^l)\in M_{\vec{d}_l}^{\perp}$ satisfying 
$\lVert\vec{\varphi}_l\rVert_{\SH}=1$ and $\lVert \mathscr{L}_{\vec{d}_l,\e_l}(\vec{\varphi}_l)\rVert_{\SH}\rightarrow 0$ as $l\rightarrow\infty$. Passing to a subsequence, we may assume that $\lim_{l \to \infty}\vec{d}_l =\vec{d}_{\infty}\in \overline{L^\prime}$ and  $\vec{\varphi}_l$ converges weakly to $\vec{\varphi}_{\infty}$ in $\SH$ as $l \rightarrow \infty$. 
Now, the fact $\vec{\varphi}_l\in M_{\vec{d}_l}^{\perp}$ implies
\begin{equation*}
\begin{split}
&\langle\vec{\varphi}_{\infty},(\vec{\gamma}\oslash \vec{d}_{\infty})U\rangle_{\SH}=\lim_{l\to\infty}\langle\vec{\varphi}_{l},(\vec{\gamma}\oslash \vec{d}_{\infty})U\rangle_{\SH}\\
=&\lim_{l\to\infty}\langle\vec{\varphi}_{l},(\vec{\gamma}\oslash \vec{d}_{l})U\rangle_{\SH}+\lim_{l\to\infty}\langle\vec{\varphi}_{l},((\vec{\gamma}\oslash \vec{d}_{\infty})-\vec{\gamma}\oslash\vec{d}_l)U\rangle_{\SH}=0
\end{split}
\end{equation*}
for any $\vec{\gamma}\in K$. This implies $\vec{\varphi}_\infty\in M_{\vec{d}_{\infty}}^{\perp}$.
On the other hand, notice from the facts that $U$ is a solution of \eqref{singleeq} and $\sum_{j=1}^n a_{ij}(d_j^l)^2=1, i=1,\cdots,n,$ that for any $\vec{\gamma}\in K$,
\begin{equation*}
\begin{split}
\langle S'_0(\vec{d}_lU)(\vec{\varphi}_l),(\vec{\gamma}\oslash \vec{d}_l) U\rangle_{\SH}
=&\int_{\Omega}\sum_{i,j=1}^n a_{ij}(\varphi_i^l(d_j^l)^2+2d_i^ld_j^l\varphi_j^l)\cdot\frac{\gamma_i}{d_i^l}U^3dx\\
=&\int_{\Omega}\Big(\sum_{i=1}^n \varphi_i^l\cdot\frac{\gamma_i}{d_i^l}U^3+\sum_{j=1}^n\sum_{i=1}^n 2a_{ij}\gamma_i\cdot d_j^l\varphi_j^lU^3\Big)dx\\
=&\langle \vec{\varphi},(\vec{\gamma}\oslash \vec{d}_l) U\rangle_{\SH}.
\end{split}
\end{equation*}
This is equivalent to $\vec{\varphi}_l-S'_0(\vec{d}_lU)(\vec{\varphi}_l)=\pi_{\vec{d}_l}^\perp(\vec{\varphi}_l-S'_0(\vec{d}_lU)(\vec{\varphi}_l))$. Therefore, we can decompose $\vec{\varphi}_l-S'_0(\vec{d}_{\infty}U)(\vec{\varphi}_l)$ as follows:
\begin{equation*}
\begin{split}
    \vec{\varphi}_l-S'_0(\vec{d}_{\infty}U)(\vec{\varphi}_l)
    = & (\vec{\varphi}_l-\pi_{\vec{d}_l}^\perp (S'_{\e_l}(\vec{d}_lU)(\vec{\varphi}_l)))+\pi_{\vec{d}_l}^\perp(S'_{\e_l}(\vec{d}_lU)(\vec{\varphi}_l)-S'_{0}(\vec{d}_lU)(\vec{\varphi}_l))\\ 
    &+(S'_0(\vec{d}_lU)(\vec{\varphi}_l)-S'_0(\vec{d}_{\infty}U)(\vec{\varphi}_l)).
    \end{split}
\end{equation*}
It is easy to see that the limits of the second and third terms equal \[ \lim_{l \to \infty} \pi_{\vec{d}_l}^\perp(S'_{\e_l}(\vec{d}_lU)(\vec{\varphi}_l)-S'_{0}(\vec{d}_lU)(\vec{\varphi}_l)=\lim_{l \to \infty}(S'_0(\vec{d}_lU)(\vec{\varphi}_l)-S'_0(\vec{d}_{\infty}U)(\vec{\varphi}_l)) = 0 
 \ \ \ \textup { in } \ \ \SH.\]
 Since the first term $(\vec{\varphi}_l-\pi_{\vec{d}_l}^\perp (S'_{\e_l}(\vec{d}_lU)(\vec{\varphi}_l)))$ is merely $\mathscr{L}_{\vec{d}_i,\e_i}\vec{\varphi}_i$,
we see that \[\lim_{l \to \infty} (\vec{\varphi}_l-S'_0(\vec{d}_{\infty}U)(\vec{\varphi}_l)) = 0 \ \ \textup{  in } \ \  \SH. \]
Then,  since $\vec{\varphi}_l$ converges weakly to $\vec{\varphi}_{\infty}$ in $\SH$, 
we get
\begin{equation*}
\begin{split}
\langle\vec{\varphi}_{\infty}-S'_0(\vec{d}_{\infty}U)(\vec{\varphi}_{\infty}),\vec{v}\rangle_{\SH}
=&\lim_{l\to\infty}\langle\vec{\varphi}_{l}-S'_0(\vec{d}_{\infty}U)(\vec{\varphi}_{\infty}),\vec{v}\rangle_{\SH}\\
=&\lim_{l\to\infty}\langle S'_0(\vec{d}_{\infty}U)(\vec{\varphi}_{l})-S'_0(\vec{d}_{\infty}U)(\vec{\varphi}_{\infty}),\vec{v}\rangle_{\SH}=0
\end{split}
\end{equation*}
for any $\vec{v}\in\SH$.
This is equivalent to
$$
\vec{\varphi}_{\infty}-S'_0(\vec{d}_{\infty}U)(\vec{\varphi}_{\infty})=(0,\cdots,0).
$$
In this case, since $\vec{\varphi}_{\infty}\in M_{\vec{d}_{\infty}}^{\perp}$, we conclude from $(ii)$ of Theorem \ref{T1} that $\vec{\varphi}_\infty=(0,\cdots,0)$.
If $\Omega$ is bounded, then by the strong $(L^2(\Omega))^n$ convergence of $\vec{\varphi}_l$ to $(0,\cdots,0)$, $\int_{\Omega}\varphi_{j}^lU^2$ converges to $0$ for any $j$ as $l \rightarrow\infty$.  If $\Omega=\R^N$, then by the exponential decay of $U$ and the strong $(L^2(\tilde{\Omega}))^n$ convergence of $\vec{\varphi}_l$ to $(0,\cdots,0)$ on any bounded domain $\tilde{\Omega}$, 
$\int_{\Omega}\varphi_{j}^lU^2$ converges to 0 for any $j$ as $l\rightarrow\infty$. In either case, we see that $S'_{\e_l}(\vec{d}_lU)(\vec{\varphi}_l)\to(0,\cdots,0)$ in $\SH$, because of \eqref{linearS}. However, this contradicts $\lVert\vec{\varphi}_l\rVert_{\SH}=1$ and $\lVert \mathscr{L}_{\vec{d}_l,\e_l}(\vec{\varphi}_l)\rVert_{\SH}\rightarrow 0$ as $l\rightarrow\infty$; 
thus proves the claim \eqref{invertable}.

Now, note that $(\vec{u}+\vec{\varphi})-S_{\e}(\vec{u}+\vec{\varphi})$ can be expressed as
\begin{equation}
\label{linearization of S}
(\vec{u}+\vec{\varphi})-S_{\e}(\vec{u}+\vec{\varphi})=(\vec{u}-S_{\e}(\vec{u}))+(\vec{\varphi}-S'_{\e}(\vec{u})(\vec{\varphi}))-R(\vec{u},\vec{\varphi}),
\end{equation}
where the characterization of the remainder term $R$ is given by
$$
\langle R(\vec{u},\vec{\varphi}),\vec{v}\rangle_{\SH}=\int_{\Omega}\sum_{i,j=1}^n(a_{ij}+\e e_{ij})( u_iv_i\varphi_j^2+2\varphi_iv_iu_j\varphi_j+\varphi_iv_i\varphi_j^2)dx.
$$
For $|\e| \le \e_0$, $\vec{d} \in L^\prime$ and $\vec{\varphi} \in M_{\vec{d}}^\perp,$
we define 
\begin{equation} \label{F}
	F_{\vec{d},\e}(\vec{\varphi}) \equiv -\mathscr{L}_{\vec{d},\e}^{-1}(\pi_{\vec{d}}^\perp (\vec{d}U-S_{\e}(\vec{d}U))-\pi_{\vec{d}}^\perp R(\vec{d}U,\vec{\varphi})) \in M_{\vec{d}}^\perp. \end{equation}
This operator is well-defined by the invertibility of $\mathscr{L}_{\vec{d},\e}$ by \eqref{invertable} and the Fredholm alternative.
For each $\vec{d} \in L^\prime$,  we aim to find a fixed point $\vec{\varphi}_{\vec{d},\e} \in M_{\vec{d}}^\perp$ of $F_{\vec{d},\e}$. By applying $\mathscr{L}^{-1}_{\vec{d},\e}\circ\pi_{\vec{d}}^\perp$ on both sides of \eqref{linearization of S}, we see that finding such fixed point is equivalent to finding $\vec{\varphi}_{\vec{d},\e} \in M_{\vec{d}}^\perp$ which solves 
\begin{equation}
\label{orthsolve}
(\vec{d}U+\vec{\varphi})-\pi_{\vec{d}}^\perp S_{\e}(\vec{d}U+\vec{\varphi})=0.\end{equation}
To find the fixed point, we show that $F_{\vec{d},\e}$ is a contraction near $(0,\cdots,0)\in M_{\vec{d}}^\perp$. First of all, by the continuous embedding $\SH\rightarrow(L^4(\Omega))^n$ and the fact that $U\in L^4(\Omega)$, there exist positive constants $\beta_1$ and $C_1$ such that if $\vec{d}\in L^\prime$ and $\lVert\vec{\varphi}\rVert_{\SH}\leq\beta_1$, then $\lVert R(\vec{d}U,\vec{\varphi})\rVert_{\SH}\leq C_1 \lVert\vec{\varphi}\rVert_{\SH}^2$. 
Since for $\vec{\varphi},\vec{\rho} \in  M_{\vec{d}}^\perp$ and $\vec{v}\in\SH$ it holds that
\begin{equation*}
\begin{split}
& \langle R(\vec{u},\vec{\varphi})-R(\vec{u},\vec{\rho}),\vec{v}\rangle_{\SH}\\
&=\int_{\Omega}\sum_{i,j=1}^n(a_{ij}+\e e_{ij}) \Big(u_i(\varphi_j^2-\rho_j^2)+2u_j(\varphi_i\varphi_j-\rho_i\rho_j)+(\varphi_i\varphi_j^2-\rho_i\rho_j^2) \Big)v_idx\\
&=\int_{\Omega}\sum_{i,j=1}^n (a_{ij}+\e e_{ij}) \Big(u_i(\varphi_j+\rho_j)(\varphi_j-\rho_j)+2u_j((\varphi_i-\rho_i)\varphi_j+\rho_i(\varphi_j-\rho_j))\\
&\quad\quad+((\varphi_i-\rho_i)\varphi_j^2+\rho_i(\varphi_j+\rho_j)(\varphi_j-\rho_j))\Big)v_idx,
\end{split}
\end{equation*}
 there exist positive $\beta_2$, and $C_2$ such that for $\vec{d}\in L^\prime,$ 
 $\lVert\vec{\varphi}\rVert_{\SH}\leq\beta_2$ and $\lVert\vec{\rho}\rVert_{\SH}$ $\leq$ $\beta_2$, then \[\lVert R(\vec{d}U,\vec{\varphi})-R(\vec{d}U,\vec{\rho})\rVert_{\SH}\leq C_2 \max\{\lVert\vec{\varphi}\rVert_{\SH},\lVert\vec{\rho}\rVert_{\SH}\}\lVert\vec{\varphi}-\vec{\rho}\rVert_{\SH}.\] 
 Define $\beta=\min\{\beta_1,\beta_2, \gamma/2C_1,\gamma/2C_2\}$. Since $U\in L^6(\Omega)$ is a solution of \eqref{singleeq}, there exists positive $\e_1$ such that for any $\vec{d}\in L^\prime$ and $|\e| \leq\e_1$, it holds that $\lVert \vec{d}U-S_{\e}(\vec{d}U)\rVert_{\SH}\leq\frac12\beta\gamma$. Then, for any $\vec{d}\in L^\prime$ and $|\e| \leq\min\{\e_0,\e_1\}$ and $\vec{\varphi}\in M_{\vec{d}}^\perp$ satisfying $\lVert\vec{\varphi}\rVert_{\SH}\leq\beta$, 
 the norm of $F_{\vec{d},\e}(\vec{\varphi})$ is bounded by
\begin{equation*}
\begin{split}
\lVert F_{\vec{d},\e}(\vec{\varphi})\rVert_{\SH}&\leq\frac{1}{\gamma}(\lVert \vec{d}U-S_{\e}(\vec{d}U)\rVert_{\SH}+\lVert R(\vec{d}U,\vec{\varphi})\rVert_{\SH})\\
&\leq\frac1{\gamma}\left(\frac12\beta\gamma+C_1\beta^2\right) \leq\frac12\beta+\frac12\beta=\beta.
\end{split}
\end{equation*}
This implies that for those $\vec{d}\in L'$ and $|\e| \le \min\{\e_0,\e_1\},$
$F_{\vec{d},\e}$ is a map from \[ B_{\vec{d},\e}(\beta) \equiv \{\vec{\varphi} \in M_{\vec{d}}^\perp : \lVert\vec{\varphi}\rVert_{\SH}\leq\beta\}\]
 to itself. 
 Note that for those $\vec{d} \in L^\prime$, $|\e| \le \min\{\e_0,\e_1\}$, and $\vec{\varphi}, \vec{\rho} \in B_{\vec{d},\e}(\beta),$
 \begin{equation*}
\begin{split}
\lVert F_{\vec{d},\e}(\vec{\varphi})-F_{\vec{d},\e}(\vec{\rho})\rVert_{\SH}&=\lVert \mathscr{L}^{-1}_{\vec{d},\e}(\pi^\perp_{\vec{d}}(R(\vec{d}U,\vec{\varphi})-R(\vec{d}U,\vec{\rho})))\rVert_{(L^2(\Omega))^n}\\
&\leq\frac{C_2\beta}{\gamma}\lVert\vec{\varphi}-\vec{\rho}\rVert_{\SH}\\
&\leq\frac12\lVert\vec{\varphi}-\vec{\rho}\rVert_{\SH}.
\end{split}
\end{equation*}
This implies that we have a unique fixed point $\vec{\varphi}_{\vec{d},\e}$ 
of $F_{\vec{d},\e}$ in $B_{\vec{d},\e}(\beta).$
 As mentioned above, this point satisfies $(\vec{d}U+\vec{\varphi}_{\vec{d},\e})-\pi_{\vec{d}}^\perp S_{\e}(\vec{d}U+\vec{\varphi}_{\vec{d},\e})=0$. 
 Since $\vec{\varphi}_{\vec{d},\e} = \lim_{k \to \infty}F_{\vec{d},\e}^{k}(0,\cdots,0),$ we see that
 \begin{equation}
 \label{phismall}
\begin{split}
\lVert\vec{\varphi}_{\vec{d},\e}\rVert_{\SH}&\leq\sum_{i=1}^\infty\lVert F_{\vec{d},\e}^{i-1}(F_{\vec{d},\e}(0,\cdots,0))-F_{\vec{d},\e}^{i-1}(0,\cdots,0)\rVert_{\SH}\\
&\leq \sum_{k=1}^\infty \frac{1}{2^{k-1}}\lVert F_{\vec{d},\e}(0,\cdots,0)-(0,\cdots,0)\rVert_{\SH}\\
&=2\lVert \mathscr{L}^{-1}_{\vec{d},\e}(\pi^\perp_{\vec{d}}(\vec{d}U-S_{\e}(\vec{d}U)))\rVert_{\SH}\\
&\leq\frac2{\gamma}\lVert\vec{d}U-S_{\e}(\vec{d}U)\rVert_{\SH}\\
&\leq C\e
\end{split}
\end{equation}
for some $C>0$.

We now aim to show the $C^1$-continuity of $\vec{\varphi}_{\vec{d},\e}$ with respect to $\vec{d}\in L'$ and small $\e$. The function $G:L'\times [-\min\{\e_0,\e_1\},\min\{\e_0,\e_1\}]\times \SH\to \SH$ defined by
$$
G(\vec{d},\e,\vec{\varphi})=(\vec{d}U+\vec{\varphi})-\pi_{\vec{d}}^{\perp}(S_{\e}(\vec{d}U+\pi_{\vec{d}}^{\perp}\vec{\varphi}))
$$
is a $C^1$-map. The directional derivative of $G$ with respect to the last component $\vec{\varphi}$ in the direction $\vec{\psi}$ equals
\begin{equation*}
\begin{split}
\frac{\partial G}{\partial\vec{\varphi}}(\vec{d},\e,\vec{\varphi})(\vec{\psi})=&\vec{\psi}-\pi_{\vec{d}}^{\perp}(S'_{\e}(\vec{d}U+\pi_{\vec{d}}^{\perp}\vec{\varphi})(\pi_{\vec{d}}^{\perp}\vec{\psi}))\\
=&\mathscr{L}_{\vec{d},\e}(\vec{\psi})-\pi_{\vec{d}}^{\perp}\big((S'_{\e}(\vec{d}U+\pi_{\vec{d}}^{\perp}\vec{\varphi})-S'_{\e}(\vec{d}U))(\pi_{\vec{d}}^{\perp}\vec{\psi})\big).
\end{split}
\end{equation*}
This formula, \eqref{invertable}, \eqref{phismall}, and the continuous embedding $\SH\rightarrow(L^4(\Omega))^n$ imply
\begin{equation*}
\begin{split}
\Big\lVert\frac{\partial G}{\partial\vec{\varphi}}(\vec{d},\e,\vec{\varphi}_{\vec{d},\e})(\vec{\psi})\Big\rVert_{\SH}\geq & \gamma\lVert\vec{\psi}\rVert_{\SH}-C_3\lVert\vec{\varphi}_{\vec{d},\e}\rVert_{\SH}\lVert\vec{\psi}\rVert_{\SH}\\
\geq & \gamma(1-C_4\e)\lVert\vec{\psi}\rVert_{\SH}
\end{split}
\end{equation*}
for some positive $C_3,C_4$ when $|\e|\leq\min\{\e_0,\e_1\}$. Since $\vec{\varphi}=\vec{\varphi}_{\vec{d},\e}$ solves $G(\vec{d},\e,\vec{\varphi})=0$, by the implicit function theorem, we conclude that $\vec{\varphi}_{\vec{d},\e}$ is a $C^1$-map with respect to $\vec{d}\in L'$ and $|\e|\leq\min\{\e_0,\e_1,(2C_4)^{-1}\}$.

For an energy functional $I_{\e}:L'\times \{\vec{\varphi}:\lVert\vec{\varphi}\rVert_{\SH}\leq C\e\}\to \R$ given by
$$
I_{\e}(\vec{d},\vec{\varphi})=\frac12\lVert \vec{d}U+\vec{\varphi}\rVert^2_\Omega-\frac14\int_{\Omega}\sum_{i,j=1}^n(a_{ij}+\e e_{ij})(d_iU+\varphi_i)^2(d_jU+\varphi_j)^2dx.
$$
we now aim to prove the existence of $\vec{d}_{\e}\in L^\prime$ such that $\nabla_{L'}I_{\e}(\vec{d}_{\e},\vec{\varphi})|_{\vec{\varphi}=\vec{\varphi}_{\vec{d}_{\e},\e}}=0$.  First of all, since 
\begin{equation*}
\int_\Omega \sum_{i=1}^n \nabla(d_iU)\cdot\nabla\varphi_i+\lambda d_iU\cdot\varphi_idx
=\int_\Omega\sum_{i=1}^nd_iU^3\cdot\varphi_idx=\int_\Omega\sum_{i,j=1}^n a_{ij} d_i^2d_jU^3\cdot\varphi_jdx,
\end{equation*}
we see that for any $\vec{d} \in L^\prime$ and $\lVert\vec{\varphi}\rVert_{\SH}\leq C\e$,
\begin{equation*}
\begin{split}
I_{\e}(\vec{d},\vec{\varphi})=&I(\vec{d}U)+\int_\Omega \sum_{i=1}^n \nabla(d_iU)\cdot\nabla\varphi_i+\lambda d_iU\cdot\varphi_idx+\frac12\lVert\vec{\varphi}\rVert^2_{\Omega}\\
&-\frac14\int_{\Omega}\sum_{i,j=1}^n\e e_{ij}d_i^2d_j^2U^4dx-\int_\Omega\sum_{i,j=1}^n(a_{ij}+\e e_{ij})d_i^2d_jU^3\cdot\varphi_jdx\\
&-\frac14\int_{\Omega}\sum_{i,j=1}^n (a_{ij}+\e e_{ij})(d_i^2U^2\cdot\varphi_j^2+4d_id_jU^2\cdot\varphi_i\varphi_j+2d_iU\cdot\varphi_i\varphi_j^2)dx\\
&-\frac14\int_{\Omega}\sum_{i,j=1}^n (a_{ij}+\e e_{ij})\varphi_i^2(d_j^2U^2+2d_jU\varphi_j+\varphi_j^2)dx\\
=&I(\vec{d}U)+\frac12\lVert\vec{\varphi}\rVert^2_{\Omega}-\frac14\e G(d_1^2,\cdots,d_n^2)\int_{\Omega}U^4dx-\int_\Omega\sum_{i,j=1}^n \e e_{ij}d_i^2d_jU^3\cdot\varphi_jdx\\
&-\frac14\int_{\Omega}\sum_{i,j=1}^n (a_{ij}+\e e_{ij})(d_i^2U^2\cdot\varphi_j^2+4d_id_jU^2\cdot\varphi_i\varphi_j+2d_iU\cdot\varphi_i\varphi_j^2)dx\\
&-\frac14\int_{\Omega}\sum_{i,j=1}^n (a_{ij}+\e e_{ij})\varphi_i^2(d_j^2U^2+2d_jU\varphi_j+\varphi_j^2)dx.
\end{split}
\end{equation*}
Then, a function $f : L'\times [-\min\{\e_0,\e_1,(2C_4)^{-1}\},\min\{\e_0,\e_1,(2C_4)^{-1}\}]\to K^{*}=K$ defined by
$$
f(\vec{d},\e)\cdot\vec{\gamma}=
\begin{dcases}
    \frac{1}{\e}\nabla_{L'}I_{\e}(\vec{d},\vec{\varphi})(\vec{\gamma}\oslash\vec{d})\big|_{\vec{\varphi}=\vec{\varphi}_{\vec{d},\e}}\text{ if }\e\neq 0\\
    -\frac12\nabla_K G(d_1^2,\cdots,d_n^2)(\vec{\gamma})\int_{\Omega}U^4dx\text{ if }\e=0
\end{dcases}
$$
is characterized for $\vec{\varphi}_{\vec{d},\e} \equiv (\tilde{\varphi}_1,\cdots,\tilde{\varphi}_n)$ by
\begin{equation*}
\begin{split}
&f(\vec{d},\e)\cdot\vec{\gamma}\\
=&-\frac12 \nabla_K G(d_1^2,\cdots,d_n^2)(\vec{\gamma})\int_{\Omega}U^4dx-\int_\Omega\sum_{i,j=1}^n e_{ij} \Big(2\gamma_id_j+d_i^2\cdot\frac{\gamma_j}{d_j}\Big)U^3\cdot\tilde{\varphi}_jdx\\
&-\frac{1}{4\e}\int_{\Omega}\sum_{i,j=1}^n (a_{ij}+\e e_{ij})\bigg(2\gamma_iU^2\cdot\tilde{\varphi}_j^2+4\Big(\frac{\gamma_i}{d_i}\cdot d_j+d_i\cdot \frac{\gamma_j}{d_j}\Big)U^2\cdot\tilde{\varphi}_i\tilde{\varphi}_j\\
&\quad\quad\quad\quad\quad\quad\quad\quad\quad\quad+2d_iU\cdot\tilde{\varphi}_i\tilde{\varphi}_j^2\bigg)dx\\
&-\frac{1}{4\e}\int_{\Omega}\sum_{i,j=1}^n (a_{ij}+\e e_{ij})\tilde{\varphi}_i^2\Big(\gamma_j U^2+2\frac{\gamma_j}{d_j}U\tilde{\varphi}_j+\tilde{\varphi}_j^2\Big)dx
\end{split}
\end{equation*}
and is a $C^1$-map from the $C^1$-continuity of $(\vec{d},\e)\mapsto\vec{\varphi}_{\vec{d},\e}$ and the fact that $\vec{\varphi}_{\vec{d},0}=(0,\cdots,0)$. Now, for $\vec{\gamma},\vec{\eta}\in K$, it holds that
$$
\nabla_{L'}f(\vec{c},0)(\vec{\gamma})\cdot\vec{\eta}=-\nabla_K^2G(c_1^2,\cdots,c_n^2)(\vec{\gamma},\vec{\eta})\int_{\Omega}U^4dx.
$$
By the nondegeneracy of $\nabla_K^2(c_1^2,\cdots,c_n^2)$ and the implicit function theorem, we see that for any small $\e>0$, there exists a unique solution $\vec{d}_{\e}$ solving
$$
f(\vec{d}_\e,\e)\cdot\vec{\gamma}=\frac{1}{\e}\nabla_{L'}I_{\e}(\vec{d}_\e,\vec{\varphi})(\vec{\gamma}\oslash\vec{d}_\e)\big|_{\vec{\varphi}=\vec{\varphi}_{\vec{d}_\e,\e}}\equiv0.
$$
Therefore, there exists a unique solution $\vec{u}_\e=\vec{d}_{\e}U+\vec{\varphi}_{\vec{d}_{\e},\e}$ of \eqref{eq3} with $A$ replaced by $A^E_\e$ such that $\vec{u}_\e$ converges to the positive vector solution $\vec{c}U$ as $\e\rightarrow0$. Furthermore, performing this procedure in the setting where $\mathscr{L}_{\vec{d},\e}:M^{\perp}_{\vec{d}}\cap\R^nU\to M^{\perp}_{\vec{d}}\cap\R^nU$ yields another sequence $\vec{u}_{\e}$, but this time additionally satisfying $\vec{u}_{\e}\in\R^nU$. This new $\vec{u}_{\e}$ and the original $\vec{u}_{\e}$ should be the same by the uniqueness of $\vec{u}_{\e}$. We thus conclude that for any small $\e>0$, there exists a unique positive vector solution $\vec{u}_\e$ of \eqref{eq3} with $A$ replaced by $A^E_\e$ such that $\vec{u}_\e\in\R^nU$ and it converges to the positive vector solution $\vec{c}U$ as $\e\rightarrow0$.
\end{proof}
\begin{remark}
The criticality condition $\nabla_K G(d_1^2,\cdots,d_n^2) = 0 $ is certainly necessary
for a bifurcation as we see in the previous example with
 \[
A\equiv  \begin{pmatrix}
1 & 1  \\
1 & 1 
\end{pmatrix},\quad
E\equiv  \begin{pmatrix}
\alpha & 0  \\
0 & \beta
\end{pmatrix}.
\]
In fact, we see in this case that $EK= \langle (\alpha,-\beta)\rangle$ and the point $(d_1,d_2) \in L$  satisfying $EK \perp (d_1^2,d_2^2)$ is given by 
$\Big (\sqrt {\frac{\beta}{\alpha+\beta}},
\sqrt{\frac{\alpha}{\alpha+\beta}} \Big ).$

The non-degeneracy condition  of $\nabla^2_K G(d_1^2,\cdots,d_n^2) $ is also necessary since if $E=A,$ the non-degeneracy does not hold and the continuation of $L$ instead of a bifurcation holds for small $|\e|.$
As a different type of example, for $n=3$, we consider
$$
A=\begin{pmatrix}
0&2&1\\
2&0&1\\
1&1&1
\end{pmatrix},
$$
$\vec{d}=(\frac{1}{\sqrt{6}},\frac{1}{\sqrt{6}},\frac{2}{\sqrt{6}})$, and
$$
E=\begin{pmatrix}
1&0&0\\
0&-1&0\\
0&0&0
\end{pmatrix}.
$$
In this case, the kernel of $K$ of $A$ is spanned by $(1,1,-2)$ and the condition $E\vec{\gamma}\perp(d_1^2,\cdots,d_n^2)$ holds for any $\vec{\gamma}\in K$. However, the inverse of $A+\e E$ equals
$$
\begin{pmatrix}
\e&2&1\\
2&-\e&1\\
1&1&1
\end{pmatrix}^{-1}=\frac{1}{\e^2}\begin{pmatrix}
1+\e&1&-2-\e\\
1&1-\e&-2+\e\\
-2-\e&-2+\e&4+\e^2
\end{pmatrix}
$$
so that
$$
\begin{pmatrix}
\e&2&1\\
2&-\e&1\\
1&1&1
\end{pmatrix}^{-1}\begin{pmatrix}
1\\
1\\
1
\end{pmatrix}=\frac{1}{\e^2}\begin{pmatrix}0\\0\\ \e^2\end{pmatrix}=\begin{pmatrix}0\\0\\1\end{pmatrix}.
$$
Therefore, $\vec{u}_\e=(0,0,1)U$ is the unique solution of \eqref{eq3} with $A$ replaced by $A+\e E$ in $\R^nU$, but $\vec{u}_\e$ does not converge to $\vec{d}U$ as $\e\rightarrow0$. This is because $E(1,1,-2)\cdot\vec{\eta}=0$ for any $\vec{\eta}\in K$. This example shows that the existence of $\vec{\eta}\in K$ satisfying $E\vec{\gamma}\cdot\vec{\eta}\neq 0$ for any $\vec{\gamma}\in K$ is a crucial condition.
\end{remark}



Next, we show that in certain cases if there exists a semi-vector nonnegative solution in $\R^nU$ of Morse index 1, then the positive vector solution in $\R^nU$, if exists, must have Morse index at least 2.

\begin{thm}
\label{T4}
Let us assume that the interaction matrix $A=(a_{ij})_{1\leq i,j\leq n}$ satisfy $a_{ij}=a_{ji} \ge 0 $ for $1\leq i,j\leq n$. For \eqref{eq3}, we assume that there exist a semi-vector nonnegative solution of form $\vec{c}U=(c_1,\cdots,c_n)U \in \R^nU$ which has Morse index 1 and 
$$
A\begin{pmatrix}
    c_1^2\\ \vdots\\ c_n^2
\end{pmatrix}\neq
\begin{pmatrix}
    1\\ \vdots\\ 1
\end{pmatrix}.
$$
Then, if  $\vec{d}U=(d_1,\cdots,d_n)U$ is a positive vector solution of \eqref{eq3}, 
the Morse index of the solution $\vec{d}U$ is at least 2.
\end{thm}
\begin{proof}
Since $\vec{d}U$ is a positive vector solution of \eqref{eq3}, it holds that
$$
A\begin{pmatrix}
    d_1^2\\ \vdots\\ d_n^2
\end{pmatrix}=\vec{1}\equiv\begin{pmatrix}
    1\\ \vdots\\ 1
\end{pmatrix}.
$$
Since the Morse index of $\vec{c}U$ is 1, Theorem \ref{T1-1} implies that
\begin{equation}
A\begin{pmatrix}
    c_1^2\\ \vdots\\ c_n^2
\end{pmatrix}\leq\vec{1}
\end{equation}
componentwise.
We therefore conclude that
\begin{equation} \label{minus}
A\begin{pmatrix}
    d_1^2-c_1^2\\ \vdots\\ d_n^2-c_n^2
\end{pmatrix}\gneq
\begin{pmatrix}
    0\\ \vdots\\ 0
\end{pmatrix}
\end{equation}
componentwise, although the left hand side is not $(0,\cdots,0)$ by the assumption. Recall that
\begin{equation}
I''(\vec{d}U)(\vec{\phi},\vec{\phi})=\lVert\vec{\phi}\rVert_{\Omega}^2-\int_{\Omega}U^2\left(\sum_{i,j=1}^n d_i\phi_i\cdot 2a_{ij}\cdot d_j\phi_j+\sum_{i=1}^n \phi_i^2\right)dx.
\end{equation}
In particular, for $\vec{\phi}\in\R^n U$, we have
\begin{equation}
I''(\vec{d}U)(\vec{\phi},\vec{\phi})=-\int_{\Omega}U^2\left(\sum_{i, j=1}^n d_i\phi_i\cdot 2a_{ij}\cdot d_j\phi_j\right)dx.
\end{equation}
This implies $I''(\vec{d}U)(\vec{d}U,\vec{d}U)<0$. Now, for $k\in \R$, we define $\vec{\phi}_k=(g_1,\cdots,g_n)U$ by $g_i=kd_i+\frac{c_i^2}{d_i}$, $1\leq i\leq n$. Then, we have
\begin{equation}
I''(\vec{d}U)(\vec{\phi}_k,\vec{\phi}_k)=-\int_{\Omega}2U^4 (k\vec{d}\odot\vec{d}+\vec{c}\odot\vec{c})\cdot A(k\vec{d}\odot\vec{d}+\vec{c}\odot\vec{c})dx.
\end{equation}
The discriminant $D$ of the following quadratic polynomial with respect to $k\in \R$
\begin{equation}
\begin{split}
    &(k\vec{d}\odot\vec{d}+\vec{c}\odot\vec{c})\cdot A(k\vec{d}\odot\vec{d}+\vec{c}\odot\vec{c})\\
    =&k^2(\vec{d}\odot\vec{d})\cdot A(\vec{d}\odot\vec{d})+2k(\vec{c}\odot\vec{c})\cdot A(\vec{d}\odot\vec{d})+(\vec{c}\odot\vec{c})\cdot A(\vec{c}\odot\vec{c})\\
    =&k^2(\vec{d}\odot\vec{d})\cdot\vec{1}_n+2k(\vec{c}\odot\vec{c})\cdot\vec{1}_n+(\vec{c}\odot\vec{c})\cdot \vec{1}_n
\end{split}
\end{equation}
satisfies
\begin{equation}
\begin{split}
    \frac{D}{4}&=((\vec{c}\odot\vec{c})\cdot\vec{1}_n)^2-((\vec{d}\odot\vec{d})\cdot\vec{1}_n)((\vec{c}\odot\vec{c})\cdot \vec{1}_n)\\
    &=((\vec{c}\odot\vec{c})\cdot\vec{1}_n)((\vec{c}\odot\vec{c}-\vec{d}\odot\vec{d})\cdot\vec{1}_n)\\
    &=((\vec{c}\odot\vec{c})\cdot\vec{1}_n)((\vec{c}\odot\vec{c}-\vec{d}\odot\vec{d})\cdot A(\vec{d}\odot\vec{d}))\\
    &=((\vec{c}\odot\vec{c})\cdot\vec{1}_n)(A(\vec{c}\odot\vec{c}-\vec{d}\odot\vec{d})\cdot (\vec{d}\odot\vec{d}))\\
    &<0,
\end{split}
\end{equation}
where we use \eqref{minus} in the last inequality.
Therefore, $I''(\vec{d}U)(\vec{\phi}_k,\vec{\phi}_k)<0$ holds for any $k\in\R$. 
Since $\vec{\phi}_k$ is orthogonal to $\vec{d}U$ in $\SH(\Omega)$ for some $k \in \R,$ we see that the Morse index of $\vec{d}U$ is at least 2.
\end{proof}
If $\det(A)\neq0$ and there exists a positive vector solution $\vec{d}U$ of \eqref{eq3}, the condition $$
A\begin{pmatrix}
    c_1^2\\ \vdots\\ c_n^2
\end{pmatrix}\neq
\begin{pmatrix}
    1\\ \vdots\\ 1
\end{pmatrix}
$$
holds in Theorem \ref{T4}.
Thus, we have the following result.
\begin{corollary}
Let us assume that the interaction matrix $A=(a_{ij})_{1\leq i,j\leq n}$ satisfy $a_{ij}=a_{ji} \ge 0 $ for $1\leq i,j\leq n$ and $\det(A)\neq0$. 
If there exists a semi-vector nonnegative solution in $\R^nU$ which has Morse index 1, the Morse index of the positive vector solution in $\R^nU$, if exists, is at least 2.
\end{corollary}




\section{Upper bound of the Morse index for strong interspecies forces}
In this section, we additionally assume large attractive interaction forces between different components, that is, we consider the following problem 
\begin{equation} \label{nsystem}
	  - \Delta v_i + \lambda v_i =  \beta_{ii} v_i^3+\sum_{j\not=i}^n \beta \beta_{ij}v_iv_j^2
	  \quad \text{in } \Omega,  \quad v_i \in H\quad (1\leq i\leq n),
	\end{equation}
where $\beta_{ij}=\beta_{ji}\geq0$ for $1\leq i,j\leq n$ and $\beta > 0$ is large.
By a transformation $v_i = u_i/\sqrt{\beta},$ we get a limit system
\begin{equation}\label{eq2}
	  - \Delta u_i +\lambda u_i = \sum_{j\not=i}^n \beta_{ij}u_iu_j^2
	  \quad \text{in } \Omega,  \quad u_i \in H\quad (1\leq i\leq n).
\end{equation}
In this case, the interaction matrix $A=(a_{ij})_{1\leq i,j\leq n}$ is given as
\begin{equation} \label{diagzero}
    a_{ij}=\begin{cases}
        0\quad\mathrm{if\ }i=j,\\
        \beta_{ij}\quad\mathrm{if\ }i\neq j.
    \end{cases}
\end{equation}

We now aim to show that the Morse index of a positive vector solution $\vec{u}\in\R^nU$ to a system \eqref{eq3} is at most $(n-2)$, instead of $n$. We prove it by using Theorem \ref{T1} and the following lemma which extends Descartes' rule of signs.

For any polynomial $f(x)$ of order $n$, the number of positive real roots is less than or equal to the number of changes of sign of coefficients of $f(x)$. Also, the number of negative real roots is less than or equal to the number of changes of sign of coefficients of $f(-x)$. This is Descartes' rule of signs. The next result says that if a polynomial $f$ of order $n$ has exactly $n$ real roots, then the number of positive real roots of $f$ is exactly the number of changes of sign of coefficients of $f(x)$.


\begin{lemma}
\label{Des}
Assume that a polynomial $f$ of order $n$ has $n$ real roots. Then, the number of positive real roots of $f$ equals the number of changes of sign of coefficients of $f(x)$.
\end{lemma}
\begin{proof}
Define $k_0=0$ and let $f$ be $$f(x)=\sum_{i=0}^{l}a_{n-k_{i}}x^{n-k_i}=a_n x^n+a_{n-k_1}x^{n-k_1}+\cdots+a_{n-k_l}x^{n-k_l},$$ where all coefficients are nonzero and $0=k_0<k_1<k_2<\cdots<k_l \le n$. Then, $f(-x)$ is expressed as 
$$f(-x)=\sum_{i=0}^{l}(-1)^{n-k_i}a_{n-k_{i}}x^{n-k_i}.$$

Suppose $k_i-k_{i-1}=1$ for some $1\leq i\leq l$. Then
$$
a_{n-k_{i-1}}a_{n-k_i}>0\iff (-1)^{n-k_{i-1}}a_{n-k_{i-1}}(-1)^{n-k_i}a_{n-k_i}=-a_{n-k_{i-1}}a_{n-k_i}<0,
$$
so the sum of the number of changes of sign of coefficients of $f(x)$ and $f(-x)$ between the $(n-k_{i-1})$-th order term and the $(n-k_i)$-th order term is 1, which equals $(k_i-k_{i-1})$ in this case.

Now, suppose $k_i-k_{i-1}\geq 2$ for some $1\leq i\leq l$. Then, the sum of the number of changes of sign of coefficients of $f(x)$ and $f(-x)$ between the $(n-k_{i-1})$-th order term and the $(n-k_i)$-th order term is at most $2$, which is less than or equal to $(k_i-k_{i-1})$ in this case. 

Summing up, the sum of the number of changes of sign of coefficients of $f(x)$ and $f(-x)$ is at most $(k_1-k_0)+(k_2-k_1)+\cdots+(k_l-k_{l-1})=k_l$. On the other hand, $f$ has $n$ real roots, where $(n-k_l)$ elements of them are 0. Therefore, $f$ has $k_l$ nonzero real roots. By Descartes' rule of signs, it holds that
\begin{equation}
\begin{split}
    k_l&=(\#\mathrm{\ of\ positive\ real\ roots\ of\ }f)+(\#\mathrm{\ of\ negative\ real\ roots\ of\ }f)\\
    &\leq(\#\mathrm{\ of\ changes\ of\ sign\ of\ coefficients\ of } f(x))\\
    &\quad +(\#\mathrm{\ of\ changes\ of\ sign\ of\ coefficients\ of } f(-x)) \leq k_l
\end{split}    
\end{equation}
In particular, the number of positive real roots of $f$ equals the number of change of sign of coefficients of $f(x)$.
\end{proof}


\begin{thm}
\label{C3}
Let $n\geq 3$ and $A=(a_{ij})_{1\leq i,j\leq n}$ satisfy $a_{ij}=a_{ji}\geq0$ and $a_{ii}=0$ for $1\leq i,j\leq n$. We assume that we have a positive solution $\vec{u}\in\R^n U$ to  a system \eqref{eq3}. Then, if $\det (A)>0$, the Morse index of $\vec{u}$ equals $(n-2-2l)$ for some integer $l\geq 0$. If $\det (A)<0$, the Morse index of $\vec{u}$ equals $(n-3-2l)$ for some integer $l\geq 0$. 
\end{thm}
\begin{proof}
This corollary follows from Lemma \ref{Des} and Theorem \ref{T1} after we count the number of changes of sign of coefficients of the characteristic polynomial 
\begin{equation}
\det(\mu I-A)=\mu^n-\sum_{i<j}a_{ij}^2\mu^{n-2}-\sum_{i<j<k}2a_{ij}a_{jk}a_{ki}\mu^{n-3}+\cdots+(-1)^n\det (A).
\end{equation}
There is one change of sign of coefficients between $\mu^n$ and $\mu^{n-3}$ and at most $(n-3)$ changes of sign of coefficients between $\mu^{n-3}$ and the constant term. Therefore, there are total at most $(n-2)$ changes of sign of coefficients.

Let $k \le n-2$ be the number of positive eigenvalues $\{a_1,\cdots, a_k\}$ of $A$ and $\{-a_{k+1},\cdots,-a_n\}$ negative eigenvalues of $A$.
Then, we see that
\begin{equation}
\det(\mu I-A)=(\mu-a_1)\cdots (\mu-a_k)(\mu + a_{k+1})\cdots(\mu+a_n).
\end{equation}
Then, we see that $(-1)^k a_1\cdots a_n = (-1)^n\det (A).$
Thus,  $n-2-k$ is even if $\det (A)>0,$ and $n-2-k$ is odd if $\det (A)< 0$. 
This completes the proof.
\end{proof}

For $n=4,$ we have quite a simple characterization of the Morse index as we see in the following .
\begin{cor}
For $n=4$, let $A=(a_{ij})_{1\leq i,j\leq 4}$ satisfy $a_{ij}=a_{ji}\geq0$ and $a_{ii}=0$ for $1\leq i,j\leq 4$.
We assume that we have a positive vector solution $\vec{u}\in\R^n U$to a  system \eqref{eq3}. 
Then, 
\begin{itemize}
\item[(i)]
the Morse index of $\vec{u}$ is one if and only  if  each $\sqrt{a_{12}a_{34}}$, $\sqrt{a_{13}a_{24}}$, and $\sqrt{a_{14}a_{23}}$ is less than equal to the sum of the others;
\item[(ii)]
 the Morse index of $\vec{u}$ is two if and only if  one of $\sqrt{a_{12}a_{34}}$, $\sqrt{a_{13}a_{24}}$, and $\sqrt{a_{14}a_{23}}$ is strictly larger than the sum of the others.
 \end{itemize}
\end{cor}
\begin{proof} 
Since the largest eigenvalue of $A$ is positive, we see from Theorem \ref{C3} that
the Morse index of $\vec{u}$ is one or two.
We note that
\brr  & &  \det(\lambda I-A)  \\
& & =   \lambda^4  -\lambda^2(a_{12}^2+a_{13}^2+a_{14}^2+a_{23}^2 + a_{24}^2 + a_{34}^2) \\
& & \quad - 2\lambda(a_{12}a_{13}a_{23} +a_{12}a_{14}a_{24} + a_{13}a_{14}a_{34} +a_{23}a_{24}a_{34})  \\
& & \quad + \left(a_{12}a_{34}-\left(\sqrt{a_{13}a_{24}}-\sqrt{a_{14}a_{23}}\right)^2\right) \cdot \left(a_{12}a_{34}-\left(\sqrt{a_{13}a_{24}}+\sqrt{a_{14}a_{23}}\right)^2\right).\err
Theorem \ref{T1} and Lemma \ref{Des}  imply that the Morse index of $\vec{u}$ is one
iff $\det (A) \le 0$ and two iff $\det (A) > 0.$ Since the determimnant of $A$ is expressed as 
$$\det (A)=\left(a_{12}a_{34}-\left(\sqrt{a_{13}a_{24}}-\sqrt{a_{14}a_{23}}\right)^2\right) \cdot \left(a_{12}a_{34}-\left(\sqrt{a_{13}a_{24}}+\sqrt{a_{14}a_{23}}\right)^2\right),$$
the conclusion follows.
\end{proof}

We now prove that for some matrix $A=(a_{ij})_{1\leq i,j\leq n}$ such that $\det (A)>0$, $a_{ij}=a_{ji}\geq0$ and $a_{ii}=0$ for $1\leq i,j\leq n$, there exists a vector solution $\vec{u}\in\R^nU$ to \eqref{eq3} with Morse index $(n-2)$. To prove it, we need the following lemma.

\begin{lemma}
\label{L1}
Let $n\geq 3$. Also let $A=(a_{ij})_{1\leq i,j\leq n}$ satisfy $a_{ij}>0$ for $1\leq i<j\leq n$ and $a_{ii}=0$ for $1\leq i\leq n$. 
Then, there exists a vector $\vec{k}_n\in\mathbb{R}^n$ satisfying $A\vec{k}_n>\vec{0}\equiv(0,0,\cdots,0)$ componentwise and 
$\vec{k}_n\cdot A\vec{k}_n<0$.
\end{lemma}
\begin{proof}
Fix $\e_1>0$ small satisfying $a_{13}-\e_1a_{23},a_{14}-\e_1a_{24},\cdots,a_{1n}-\e_1a_{2n}>0$. Then, we define 
\[\e=\e_1^2 a_{12}a_{23}/a_{13}>0 \ \ \text{ and } \ \ \e_2=(\e_1 a_{12}+\e)/a_{13}>0.\]
 Let $\vec{k}_n=(1,-\e_1,\e_2,0,0,\cdots,0).$
We then have
\begin{equation}
A\vec{k}_n
=
\begin{pmatrix}
    -\e_1 a_{12}+\e_2 a_{13}\\
    a_{12}+\e_2 a_{23}\\
    a_{13}-\e_1 a_{23}\\
    a_{14}-\e_1 a_{24}+\e_2 a_{34}\\
    \cdots\\
    a_{1n}-\e_1 a_{2n}+\e_2 a_{3n}
\end{pmatrix}.
\end{equation}
By the definition of $\e_1,$ this vector is positive componentwise since \[-\e_1 a_{12}+\e_2 a_{13}=-\e_1 a_{12}+(\e_1 a_{12}+\e)=\e>0.\] Furthermore, we have
\begin{equation}
\begin{split}
\vec{k}_n\cdot A\vec{k}_n&=-2\e_1 a_{12}+2\e_2 a_{13}-2\e_1 \e_2 a_{23}\\
&=-2\e_1 a_{12}+2(\e_1 a_{12}+\e)-\frac{2\e_1 (\e_1 a_{12}+\e) a_{23}}{a_{13}}\\
&=2\e-\frac{2\e_1^2 a_{12}a_{23}}{a_{13}}-\frac{2\e_1\e a_{23}}{a_{13}}\\
&=-\frac{2\e_1\e a_{23}}{a_{13}}<0,
\end{split}
\end{equation}
so $\vec{k}_n$ is the desired vector.
\end{proof}

\begin{thm}
\label{MM3}
Let $n\geq 3$. Then, there exists a matrix $A=(a_{ij})_{1\leq i,j\leq n}$ such that $\det (A)>0$, $a_{ij}=a_{ji}>0$ for $1\leq i<j\leq n$ and $a_{ii}=0$ for $1\leq i\leq n$, for which a system \eqref{eq3} admits a positive vector solution of the form $\vec{u}=(c_1,c_2,\cdots,c_n)U$ of Morse index $(n-2)$.
\end{thm}
\begin{proof}
We prove this theorem by induction. Assume that there exists a matrix $A_n=(a_{ij})_{1\leq i,j\leq n}$ such that $\det (A_n)>0$, $a_{ij}=a_{ji}>0$ for $1\leq i<j\leq n$ and $a_{ii}=0$ for $1\leq i\leq n$, which a system \eqref{eq3} with the interaction matrix $A=A_n$ admits a positive vector solution of the form $(c_1,c_2,\cdots,c_n)U$ of which Morse index is $(n-2)$. For $n=3$, we can take
\begin{equation}
A_3=
\begin{pmatrix}
0&1&1\\
1&0&1\\
1&1&0
\end{pmatrix}.
\end{equation}
We aim to construct a matrix $A_{n+1}=(a_{ij})_{1\leq i,j\leq n+1}$ such that $\det (A_{n+1})>0$, $a_{ij}=a_{ji}>0$ for $1\leq i<j\leq n+1$ and $a_{ii}=0$ for $1\leq i\leq n+1$, for which a system \eqref{eq3} with interaction matrix $A=A_{n+1}$ admits a positive vector solution of the form $(d_1,d_2,\cdots,d_{n+1})U$ of which Morse index is $(n-1)$.

First, by Lemma \ref{L1}, there exists a vector $\vec{k}_n=(k_1,k_2,\cdots,k_n)\in\mathbb{R}^n$ satisfying $A_n\vec{k}_n>\vec{0}$ componentwise and $\vec{k}_n\cdot A_n\vec{k}_n<0$. Define $\vec{c}_n=(c_1^2,c_2^2,\cdots,c_n^2)$. We take a small $\e>0$ so that $c_i^2-\e k_i>0$ for $1\leq i\leq n$. Then let $d_i=(c_i^2-\e k_i)^{1/2}$ for $1\leq i\leq n$ and define $\vec{d}_n=(d_1^2,d_2^2,\cdots,d_n^2)=\vec{c}_n-\e\vec{k}_n$. Then, $(\e A_n\vec{k}_n)\cdot \vec{d}_n>0$ holds.
Now, we define
\begin{equation}
\vec{\alpha}_n=\frac{\e A_n\vec{k}_n}{(\e A_n\vec{k}_n)\cdot \vec{d}_n}\quad\mathrm{and}\quad d_{n+1}=((\e A_n\vec{k}_n)\cdot \vec{d}_n)^{1/2}.
\end{equation}
Note that $\vec{\alpha}_n>\vec{0}$ componentwise and $d_1,d_2,\cdots,d_{n+1}>0$. Define a $(n+1)\times(n+1)$ matrix $A_{n+1}$ by
\begin{equation}
A_{n+1}=
\begin{pmatrix}
A_n&\vec{\alpha}_n\\
\vec{\alpha}_n^T&0
\end{pmatrix}.
\end{equation}
Then, $A_{n+1}=(a_{ij})_{1\leq i,j\leq n+1}$ satisfies $a_{ij}=a_{ji}>0$ for $1\leq i<j\leq n+1$ and $a_{ii}=0$ for $1\leq i\leq n+1$.
We see that \eqref{Ac=1} holds for $A_{n+1}$ and $(\vec{d}_n,d_{n+1})$ as
\begin{equation}
\begin{split}
A_{n+1}
\begin{pmatrix}
d_1^2\\
d_2^2\\
\cdots\\
d_{n+1}^2
\end{pmatrix}
&=
\begin{pmatrix}
A_n&\vec{\alpha}_n\\
\vec{\alpha}_n^T&0
\end{pmatrix}
\begin{pmatrix}
\vec{d}_n\\
d_{n+1}^2
\end{pmatrix}
=
\begin{pmatrix}
A_n\vec{d}_n+d_{n+1}^2\vec{\alpha}_n\\
\vec{\alpha}_n\cdot\vec{d}_n
\end{pmatrix}\\
&=
\begin{pmatrix}
A_n\vec{c}_n-\e A_n\vec{k}_n+((\e A_n\vec{k}_n)\cdot\vec{d}_n)\vec{\alpha}_n\\
1
\end{pmatrix}
=
\begin{pmatrix}
A_n\vec{c}_n\\
1
\end{pmatrix}
=\begin{pmatrix}
1\\
1\\
\vdots\\
1
\end{pmatrix}.
\end{split}
\end{equation}
Therefore, $(d_1,d_2,\cdots,d_{n+1})U$ is indeed a positive vector solution for a system \eqref{eq3} with the interaction matrix $A=A_{n+1}$.

It now remains to show that the Morse index of $(d_1,d_2,\cdots,d_{n+1})U$ is $(n-1)$ and that $\det(A_{n+1})>0$. By Corollary \ref{C3}, it suffices to show that the Morse index of $(d_1,d_2,\cdots,d_{n+1})U$ is at least  $(n-2)$ and that $\det (A_{n+1})>0$. We prove the former one first. Since the Morse index of $(c_1,c_2,\cdots, c_n)U$ is $(n-2)$, we can take an orthogonal set $\{\vec{\phi}^1,\vec{\phi}^2,\cdots,\vec{\phi}^{n-2}\}\subset H^n$ which spans a subspace $W\subset H^n$ satisfying
\begin{equation}
I''((c_1,c_2,\cdots,c_n)U)(\vec{\phi},\vec{\phi})<0\quad\mathrm{for}\quad\vec{\phi}\in W\setminus\{0\}.
\end{equation}
Note that
\begin{equation}
\sup_{\substack{\lVert \vec{\phi}\rVert_{\Omega}^2=1\\ \vec{\phi}\in W}}I''((c_1,c_2,\cdots,c_n)U)(\vec{\phi},\vec{\phi})<0.
\end{equation}
We now fix $\vec{\phi}=(\phi_1,\phi_2,\cdots,\phi_n)\in\mathrm{span}\{\vec{\phi}^1,\vec{\phi}^2,\cdots,\vec{\phi}^{n-2}\}$ satisfying $\lVert\vec{\phi}\rVert_{\Omega}^2=1$. Since $|d_i^2-c_i^2|=\e|k_i|$ and
\begin{equation}
|d_{i}-c_{i}|=|(c_i^2-\e k_i)^{1/2}-c_i|=\frac{\e|k_i|}{|(c_i^2-\e k_i)^{1/2}+c_i|}\leq \frac{\e|k_i|}{|c_i|}
\end{equation}
holds, it follows that
\begin{equation}
\begin{split}
&I''((d_1,d_2,\cdots,d_{n+1})U)((\vec{\phi},0),(\vec{\phi},0))\\
&=\lVert(\vec{\phi},0)\rVert_{\Omega}^2-\int_{\Omega}U^2\sum_{i,j=1}^{n+1} a_{ij}(d_j^2\phi_i^2+2d_id_j\phi_i\phi_j)dx\\
&=\lVert\vec{\phi}\rVert_{\Omega}^2-\int_{\Omega}U^2\left(\sum_{i,j=1}^n a_{ij}(d_j^2\phi_i^2+2d_id_j\phi_i\phi_j)+\sum_{i=1}^n a_{i(n+1)}\cdot d_{n+1}^2\phi_i^2\right)dx\\
&<\lVert\vec{\phi}\rVert_{\Omega}^2-\int_{\Omega}U^2\sum_{i,j=1}^n a_{ij}(c_j^2\phi_i^2+2c_ic_j\phi_i\phi_j)dx+C\e \\
&=I''((c_1,c_2,\cdots,c_n)U)(\vec{\phi},\vec{\phi})+C\e,
\end{split}
\end{equation}
for some constant $C$ independent of $\vec{\phi}$ with $\lVert\vec{\phi}\rVert^2_{\Omega}=1$ and small enough $\e>0$. Therefore, for any small $\e>0$, it holds that
\begin{equation}
I''((d_1,d_2,\cdots,d_{n+1})U)((\vec{\phi},0),(\vec{\phi},0))<0\quad\mathrm{for}\quad\vec{\phi}\in W.
\end{equation}
Thus, the Morse index of $(d_1,d_2,\cdots,d_{n+1})U$ is at least $(n-2)$.

Finally, we show that $\det (A_{n+1})>0$. We define the $(i,j)$ cofactor of $A_n$ by $C_{ij}$. Then, we see that
\begin{equation}
\begin{split}
\det (A_{n+1})&=\det
\begin{pmatrix}
    A_n&\vec{\alpha}_n\\
    \vec{\alpha}_n^T&0
\end{pmatrix}
=\sum_{1\leq i,j\leq n}\alpha_i\alpha_j(-C_{ij})\\
&=\vec{\alpha}_n\cdot(-\det (A_n)\cdot A_n^{-1})\vec{\alpha}_n>0,
\end{split}
\end{equation}
since $\vec{\alpha}_n$ is parallel to $A_n\vec{k}_n$, which in turn satisfies $(A_n\vec{k}_n)\cdot (A_n^{-1}A_n\vec{k}_n)=\vec{k}_n\cdot A_n\vec{k}_n<0$.
This completes the proof.
\end{proof}

For any $1\leq m\leq n-2$, we can find a matrix $A=(a_{ij})_{1\leq i,j\leq n}$ satisfying $a_{ij}=a_{ji}\geq 0$ and $a_{ii}=0$ for $1\leq i,j\leq n$, such that there exists a corresponding positive vector solution $\vec{v}\in\R^nU$ of which the Morse index equals $m$. We omit the details of the proof, but the proof is essentially the same as the proof of Theorem \ref{T2}. The only difference is that we use the matrices $A_n$ constructed in Theorem \ref{MM3} instead of the identity matrices.

\begin{thm}
Let $1\leq m\leq n-2$. Then, there exists a matrix $A=(a_{ij})_{1\leq i,j\leq n}$ satisfying $a_{ij}=a_{ji}\geq 0$ and $a_{ii}=0$ for $1\leq i,j\leq n$ such that $\det (A)\neq 0$ and there exists a positive vector solution $\vec{u}\in\R^nU$ to a system \eqref{eq3} which the Morse index of $\vec{u}$ equals $m$.  
\end{thm}
\begin{remark}
The upper bound $n-2$ depends strongly on the vaishing property of $a_{ii}, i=1,\cdots, n,$ not on the other entries $a_{ij}, i\ne j.$
For $n=4,$ we consider a $4\times4$ matrix $A$ given by
\[A\equiv  \begin{pmatrix}  
0& 0 &1 &0 \\
 0&0 & 0& 1\\
  1&0 & 0&0 \\
  0&1 & 0& 0\\
\end{pmatrix}
\]
Then, for a least energy solution $U$ of \eqref{singleeq},
$(1,1,1,1)U$ is a positive vector solution of \eqref{eq2} with $n=4.$
Furthermore, the matrix $A$ has an eigenvalue $1$ with orthonomal eigenvectors $\{(1,0,1,0), (0,1,0,1)\}$.
Thus the Morse index of $(1,1,1,1)U$ is $2$.
We strongly believe that for any $k > 0$, there exist large $n > 0$ and $n\times n$ matrix $A$ with $a_{ij} \ge 0, 1\le i,j \le n$ and $a_{ij} =0, |i-j|\le k$
for which there exists a positive vector solution of \eqref{eq2} with the Morse index $n-2.$
\end{remark}

\section{Variational characterization of  the positive vector solutions}

In this section, we first give a variational characterization for a synchronized positive vector solution to be a ground state solution. Then we use it to establish a new existence result of positive vector solutions for the more general system \eqref{eq1} for which the linear frequencies are not assumed to be the same.

\begin{thm}\label{V1}
We assume that  $\lambda_1=\cdots=\lambda_n = \lambda.$ 
The following are equivalent:
\begin{itemize}
\item[(i)] the interaction matrix $A=(a_{ij})_{1\leq i,j\leq n}$ with $a_{ij} \ge 0$ has eigenvalues $\mu_1 > 0 \ge  \mu_2 \ge \cdots \ge \mu_n$ and the convex hull
of column vectors $\{\vec{a}_{i}\}_{i=1}^n$ with $\vec{a}_i = (a_{i1},\cdots,a_{in})$
contains an interior point $(e,\cdots,e)$ for some $e > 0;$
\item[(ii)] the following minimization
 \[ M_\lambda \equiv \inf\Big \{I(\vec{u}) \ \Big | \ I^\prime(\vec{u})(\vec{u}) = 0 \ \text{ for } \ \vec{u} \in \SH \setminus\{0\} \Big \}\]
 is attained by  a positive vector solution $\vec{c} U \in \R^n U $ of \eqref{eq3} where $U$ is a least energy solution  of \eqref{singleeq}.
\end{itemize}
\end{thm}
\begin{proof}
It is obvious from Theorem \ref{T1} that $(ii)$ implies $(i)$.

Suppose that $(i)$ holds.
Since the convex hull
of column vectors $\{\vec{a}_{i}\}_{i=1}^n$ 
contains an interior point $(c,\cdots,c)$ for some $c > 0,$
there exist $c_1, \cdots c_n > 0$ such that 
\begin{equation*}
\sum_{i=1}^n \vec{a}_{i}c_i^2 = A
\begin{pmatrix}
c_1^2\\
c_2^2\\
\vdots\\
c_n^2
\end{pmatrix}
=
\begin{pmatrix}
1\\
1\\
\vdots\\
1
\end{pmatrix}.
\end{equation*}
Then, for any least energy positive solution $U$ of \eqref{singleeq},
$\vec{c} U=(c_1,\cdots,c_n)U\in \R^nU$ is a positive vector solution.
On the other hand, it is standard  to show that there exists a nonnegative minimum point $\vec{u}$ of $I$ on $\{ \vec{u} \in \SH \setminus\{0\} \  |  \ I^\prime(\vec{u})(\vec{u}) = 0 \}.$
A synchronization result in \cite[Theorem 1]{Corr2} implies that $\vec{u} \in \R^nU$ for a least energy positive solution $U$ of \eqref{singleeq}. Let $\vec{u} = \vec{e} U$ for $\vec{e} = (e_1,\cdots,e_n)$ with $e_i \ge 0.$
Note that $\vec{e} U$ is a nonnegative solution of \eqref{eq3} with Morse index $1.$
If $\vec{e} U$ is not a positive vector solution of \eqref{eq3} and satisfies
$$
A\begin{pmatrix}
    e_1^2\\ \vdots\\ e_n^2
\end{pmatrix}\neq A\begin{pmatrix}
    c_1^2\\ \vdots\\ c_n^2
\end{pmatrix},
$$
then Theorem \ref{T4} implies that  the vector solution $\vec{c}U$ is a positive vector solution of 
\eqref{eq3} of Morse index $\ge 2.$ This contradicts Theorem \ref{T1}. On the other hand, if $\vec{e}U$ is not a positive vector solution of \eqref{eq3} and satisfies
$$
A\begin{pmatrix}
    e_1^2\\ \vdots\\ e_n^2
\end{pmatrix}=A\begin{pmatrix}
    c_1^2\\ \vdots\\ c_n^2
\end{pmatrix},
$$
then
$$
\begin{pmatrix}
    e_1^2-c_1^2\\ \vdots\\ e_n^2-c_n^2
\end{pmatrix}\perp \begin{pmatrix}
    1\\ \vdots\\ 1
\end{pmatrix}
$$
because the kernel of $A$ is perpendicular to the image of $A$. Therefore,
$$
I(\vec{c}U)=\sum_{i=1}^n c_i^2E_\lambda=\sum_{i=1}^n e_i^2E_\lambda=I(\vec{e}U),
$$
where
\[E_\lambda \equiv \frac12\int_\Omega|\nabla U|^2 + \lambda U^2 dx - \frac14\int_\Omega U^4 dx.\]
Hence $\vec{c}U$ is also a minimizer, and this completes the proof.
\end{proof}
Theorem  \ref {V1} implies as a perturbation result that if
the interaction matrix $A=(a_{ij})_{1\leq i,j\leq n}$ with $a_{ij} \ge 0$ has eigenvalues $\mu_1 > 0 >  \mu_2 \ge \cdots \ge \mu_n$ and the convex hull
of column vectors $\{\vec{a}_{i}\}_{i=1}^n$ with $\vec{a}_i = (a_{i1},\cdots,a_{in})$
contains an interior point $(c,\cdots,c)$ for some $c > 0,$ there exists
$\delta > 0$ such that for $\lambda_1,\cdots,\lambda_n > 0$ with $\max_{1\le i \ne j\le n}|\lambda_i-\lambda_j| \le \delta,$ 
\eqref{eq1}
has a positive vector solution, which is a minimizer of 
the following minimization
 \begin{equation} \label{LEMin}  M_{\lambda_1,\cdots,\lambda_n} \equiv \inf\Big \{  \tilde{I}(\vec{u}) \ \Big | \  
   \tilde{I}^\prime(\vec{u})(\vec{u}) = 0  \ \text{ for } \ \vec{u} \in \SH \setminus\{0\} \Big \},\end{equation}
  where \[
 \tilde{I} (\vec{u}) \equiv \frac{1}{2} \sum_{i=1}^n |\nabla u_i|_{2, \Omega}^2 + \lambda_i |u_i|_{2, \Omega}^2  - \frac{1}{4} \int_{\Omega} 
\sum_{i,j=1}^{n} a_{ij} u_i^2 u_j^2 \, dx.\]
Here we get a non-perturbative existence result for general $\lambda_1,\cdots,\lambda_n > 0$ using Theorem  \ref {V1}.
We refer to similar works in \cite{CFT,LLC,LWz2} for restricted class of interaction matrix $A,$ typically $a_{ij} = a$ for $i\ne j$.
Note that if
\begin{equation*}
A
\begin{pmatrix}
c_1^2\\
c_2^2\\
\vdots\\
c_n^2
\end{pmatrix}
=
\begin{pmatrix}
1\\
1\\
\vdots\\
1
\end{pmatrix},
\end{equation*}
we see that
$$
I(\vec{c}U)=\sum_{i=1}^n c_i^2 E_\lambda,
$$
where
\[E_\lambda \equiv \frac12\int_\Omega|\nabla U|^2 + \lambda U^2 dx - \frac14\int_\Omega U^4 dx.\]
Let $B=(b_{ij})$ be a $k\times k$ principal submatrix of $A$
and assume that there exists $d_1,\cdots,d_k > 0$ such that
$\sum_{j=1}^kb_{ij}d_j^2 = 1$ for $i =1,\cdots,k.$
Then, we have the following result.
\begin{prop} \label{compare}
For an interaction symmetric matrix $A = (a_{ij})$ with $a_{ij} > 0$
and a $k\times k$ principal submatrix $B=(b_{ij})$ of $A$,
we assume that $A$ has eigenvalue $\mu_n\le \mu_{n-1} \le \cdots \le \mu_2 < 0 < \mu_1$ and there exist $\alpha_1,\cdots,\alpha_n,\beta_1,\cdots,\beta_k > 0$ satisfying
\[A \begin{pmatrix}
    \alpha_1\\ \vdots\\ \alpha_n
\end{pmatrix} = \begin{pmatrix}
    1\\ \vdots\\ 1\end{pmatrix}, \qquad B \begin{pmatrix}
    \beta_1 \\ \vdots\\ \beta_k
\end{pmatrix} = \begin{pmatrix}
    1\\ \vdots\\ 1
\end{pmatrix}.
\]
Then, if $k< n,$
 \[ 1+ |\mu_2|\frac{\min_{1\le j\le n} \big(\sum_{i=1}^n(-1)^{i+j}\det(( a_{pq})_{p\neq i,q\neq j}) \big)^2 }{\det(A)\sum_{i,j=1}^n(-1)^{i+j}\det(( a_{pq})_{p\neq i,q\neq j})} 
 \le 1+ |\mu_2| \frac{\alpha_{k+1}^2 + \cdots +\alpha_n^2}{\sum_{i=1}^n \alpha_i} \le \frac{\sum_{i=1}^k \beta_i}{\sum_{i=1}^n \alpha_i}.\]
\end{prop}
\begin{proof}
If $B$ is singular, we can approximate $B$ by nonsingular matrices satisfying the same property.
Thus, we may assume that  $A$ and $B$ are nonsingular.
We write
\[  A=  \begin{pmatrix}
    B & C \\
    C^T & D \end{pmatrix},  \vec{\alpha} = \begin{pmatrix}
    \alpha_1\\ \vdots\\ \alpha_n  
\end{pmatrix} =  \begin{pmatrix}
    \vec{\alpha}^1\\ \vec{\alpha}^2  
\end{pmatrix} \in \R^k\times \R^{n-k},  \vec{\beta} = \begin{pmatrix}
    \beta_1\\ \vdots\\ \beta_k  
\end{pmatrix} \in \R^k,   \vec{1}_l = \begin{pmatrix}
    1\\ \vdots\\ 1 
\end{pmatrix} \in \R^l,\]
where $D$ is $(n-k)\times(n-k)$ matrix and $C$ is $k\times(n-k)$ matrix.
Then, we see that
\[ B\vec{\alpha}^1 + C\vec{\alpha}^2 = \vec{1}_k\]
\[ C^T\vec{\alpha}^1 + D\vec{\alpha}^2 = \vec{1}_{n-k}.\]
Since $\vec{\alpha}^1 + B^{-1}C\vec{\alpha}^2 = B^{-1}\vec{1}_k = \vec{\beta},$
we see from the second identity that
\[C^T\vec{\beta}-C^TB^{-1}C\vec{\alpha}^2 +D\vec{\alpha}^2 = \vec{1}_{n-k}.\]
Denoting $S = D-C^TB^{-1}C,$ the Schur complement of $B$, we see that
\[S\vec{\alpha}^2 =  \vec{1}_{n-k}- C^T\vec{\beta}.\]
Then, we get
\[\sum_{i=1}^n\alpha_i = \vec{1}_k^T \vec{\alpha}^1+ \vec{1}_{n-k}^T \vec{\alpha}^2 =  \vec{1}_k^T(\vec{\beta}-B^{-1}C\vec{\alpha}^2) + \vec{1}^T_{n-k} \vec{\alpha}^2=
\sum_{i=1}^k\beta_i -  \vec{\beta}^T C\vec{\alpha}^2 + \vec{1}^T_{n-k} \vec{\alpha}^2\]
\[ = \sum_{i=1}^k\beta_i + (\vec{1}_{n-k} - C^T\vec{\beta})^T\vec{\alpha}^2 = \sum_{i=1}^k\beta_i +(\vec{\alpha}^2)^TS\vec{\alpha}^2; \]
thus
\begin{equation} \label{dieies}
\sum_{i=1}^k\beta_i  - \sum_{i=1}^n\alpha_i = -(\vec{\alpha}^2)^TS\vec{\alpha}^2. \end{equation}
By an LDU decomposition of $A$, we see that for $S= D- C^TB^{-1}C,$
\[  A=  \begin{pmatrix}
    B & C \\
    C^T & D \end{pmatrix} = \begin{pmatrix}
    I_k & 0 \\
    C^TB^{-1} & I_{n-k}
\end{pmatrix} 
\begin{pmatrix}
    B & 0 \\
    0 & D- C^TB^{-1}C
\end{pmatrix}
\begin{pmatrix}
    I_k & B^{-1}C \\
      0 &  I_{n-k}
\end{pmatrix}; \]
\[  A^{-1}=  \begin{pmatrix}
    B^{-1}+B^{-1}CS^{-1}C^TB^{-1} & -B^{-1}CS^{-1}\\
    -S^{-1}C^TB^{-1} & S^{-1}\end{pmatrix}. \]
This implies that $S$ is nonsingular and $S^{-1}$ is a principal submatrix of $A^{-1}$.
Since all entries of $A$ are positive, we see from the minmax characterization of eigenvalues of $A$ and $B$ that $B$ also has only one positive eigenvalue.
The Sylvester law of inertia implies that the number of positive eigenvalues of $A$ is the sum of those of $B$ and $S$.
This implies that all eigenvalues of $S$ are strictly negative;
thus $(\vec{\alpha}^2)^TS\vec{\alpha}^2 < 0.$
Let $-\omega_k\le \cdots \le -\omega_1 < 0$ be eigenvalues of $S^{-1}.$
Since $A^{-1}$ has eigenvalues \[\nu_n\equiv 1/\mu_2 \le \cdots \le \nu_{2} \equiv 1/\mu_n < 0 < \nu_1 \equiv 1/\mu_1,\]
we see from the Cauchy interlacing theorem or the Poincar\'e separation theorem that 
\[ \nu_{n-k+i}  \le -\omega_i \le  \nu_i\ \ \text{ for } \  i=1,\cdots,k < n.\]
This implies that $\frac{-1}{\omega_k} \le \frac{1}{\nu_n} = \mu_2$;
thus the eigenvalues of $S$ is less than or equal to $\mu_2.$
This and \eqref{dieies} imply that
\begin{equation} \label{partes}|\mu_2| (\alpha_{k+1}^2 + \cdots +\alpha_n^2) \le \sum_{i=1}^k \beta_i - \sum_{i=1}^n \alpha_i.\end{equation}
Since $\vec{\alpha}$ is the inverse image of $\vec{1}_n$ by $A,$
$\alpha_i = \sum_{j=1}^n(-1)^{i+j}\text{det}(( a_{pq})_{p\neq i,q\neq j})/\det(A).$
Thus, we see that
\[ \frac{\big(\sum_{j=1}^n(-1)^{i+j}\det(( a_{pq})_{p\neq i,q\neq j})\big)^2 }{\det(A)\sum_{i,j=1}^n(-1)^{i+j}\det(( a_{pq})_{p\neq i,q\neq j})} 
= \frac{(\alpha_i)^2}{\sum_{i=1}^n\alpha_i}.\]
This and \eqref{partes} implies the claimed inequalities
\end{proof}
\begin{thm}\label{LET}
Suppose that the symmetric interaction matrix $A=(a_{ij})_{1\leq i,j\leq n}$ with $a_{ij} >  0$ has eigenvalues $\mu_1 > 0 >  \mu_2 \ge \cdots \ge \mu_n$ and
there exist $c_1,\cdots,c_n > 0$ satisfying
 \[A \begin{pmatrix}
    c_1^2\\ \vdots\\ c_n^2
\end{pmatrix} = \begin{pmatrix}
    1\\ \vdots\\ 1\end{pmatrix}.\]
    
Then, for $\lambda_1,\cdots,\lambda_n > 0,$ there exists a positive vector solution of \eqref{eq1}, which is a minimizer of \eqref{LEMin},
if  \begin{equation}\label{VarCondi}\Big(\frac{\max\{\lambda_1,\cdots,\lambda_n\}}{ \min\{\lambda_1,\cdots,\lambda_n\}}\Big)^{2-\frac{N}{2}}
< 1+ |\mu_2|\frac{\min_{1\le j\le n} \big(\sum_{i=1}^n(-1)^{i+j}\det(( a_{pq})_{p\neq i,q\neq j})\big)^2 }{\det(A)\sum_{i,j=1}^n(-1)^{i+j}\det(( a_{pq})_{p\neq i,q\neq j})}.\end{equation}

\end{thm}

\begin{proof} 
We define 
\[ \lambda^M \equiv \max\{\lambda_1,\cdots,\lambda_n\}, \ \ \ \lambda^m \equiv  \min\{\lambda_1,\cdots,\lambda_n\}.\]
We first note that
\begin{equation}\label{eCo} M_{\lambda_1,\cdots,\lambda_n} \le M_{\lambda^M,\cdots,\lambda^M} = M_{\lambda^M}= E_{\lambda^M}(c_1^2+\cdots+c_n^2),\end{equation}
where for any $\lambda > 0,$  $E_\lambda$ is a least energy level of \eqref{singleeq}. 
Suppose that a nonnegative minimizer $\vec{u}=(u_1,\cdots,u_n)$ of $M_{\lambda_1,\cdots,\lambda_n}$ is not a positive vector solution of \eqref{eq1}.
We may assume that $u_1,\cdots,u_l> 0$ and $u_{l+1}=\cdots=u_n = 0$ with $l< n.$ Let $B_l = (a_{ij})_{1\le i,j\le l}$ be a principal submatrix of $A.$
Then, we may define a minimization $M_{\lambda_1,\cdots,\lambda_l}$
as in \eqref{LEMin} with $B_l$ replacing $A.$
The vanishing property $u_{l+1}=\cdots=u_n = 0$ implies that 
\begin{equation} \label{eCom}
M_{\lambda_1,\cdots,\lambda_n} = \tilde{I}(\vec{u}) \ge  M_{\lambda_1,\cdots,\lambda_l} \ge M_{\lambda^m,\cdots,\lambda^m},\end{equation}
where we put  $\lambda^m$ $l-$times in $M_{\lambda^m,\cdots,\lambda^m}$.
By the synchronization result in \cite{Corr1,Corr2},
 $M_{\lambda^m,\cdots,\lambda^m}$ is attained, up to a permutation of components, by $(d_1,\cdots,d_k,0,\cdots,0)U$ with $k \le l,$
 where $U$ is a least energy solution of \eqref{singleeq} with $\lambda^m$ replacing $\lambda.$
 This implies that 
 \[M_{\lambda^m,\cdots,\lambda^m} = E_{\lambda^m}(d_1^2+\cdots+d_k^2).\]
 Then,  we get from \eqref{eCo} and \eqref{eCom} that
 \[E_{\lambda^M}(c_1^2+\cdots+c_n^2) \ge M_{\lambda_1,\cdots,\lambda_n}  \ge E_{\lambda^m}(d_1^2+\cdots+d_k^2) \]
Then, since $E_\lambda=\lambda^{2-\frac{N}{2}}E_1,$
we see from Proposition \ref{compare} with $\alpha_i =c_i^2$ and $\beta_j = d_j^2$ for $1\le i\le n$ and $1\le j \le k$ that 
\[\Big(\frac{\lambda^M}{\lambda^m}\Big)^{2-\frac{N}{2}} 
\ge \frac{d_1^2+\cdots+d_k^2}{c_1^2+\cdots+c_n^2} \ge 1+ |\mu_2|\frac{\min_{1\le j\le n} \big(\sum_{i=1}^n(-1)^{i+j}\det(( a_{pq})_{p\neq i,q\neq j})\big)^2 }{\det(A)\sum_{i,j=1}^n(-1)^{i+j}\det(( a_{pq})_{p\neq i,q\neq j})}.\] 
This contradicts \eqref{VarCondi}.
Thus we conclude that the minimizer $\vec{u}$ is a positive vector solution.
\end{proof}
Let $(c_1,\cdots,c_n)U =\vec{c} U \in \R^n U $ be a positive vector solution of  a system \eqref{eq3} with the interaction matrix $A=(a_{ij})_{1\leq i,j\leq n}$ satisfying $a_{ij}=a_{ji} \ge 0 $ for $1\leq i,j\leq n$.
Let $\gamma_1 \ge \cdots \ge \gamma_k > 0$ be the positive eigenvalues of $A$ with corresponding orthonormal eigenvectors 
$\vec{\eta}_1=(\eta_1^1,\cdots,\eta_n^1),\cdots,\vec{\eta}_k=(\eta_1^k,\cdots,\eta_n^k),$ respectively.
Then, we consider a set 
\[\widetilde{M}_k \equiv \{ \vec{u}=(u_1,\cdots,u_n) \in \SH \setminus\{0\} \ | \ I^\prime(\vec{u})\bigg (\frac{\eta_1^i}{(c_1)^2}u_1,\cdots,\frac{\eta_n^i}{(c_n)^2}u_n\bigg) = 0, 
\ \ i=1,\cdots,k\}.\]
We can show that a neighborhood $V_k$ of $\vec{u}=(c_1,\cdots,c_n)U$ in $\widetilde{M}_k$ is a smooth manifold.
Then, $\vec{u}=(c_1,\cdots,c_n)U$ would be a local minimizer of $I$ on $\widetilde{M}_k.$
Since the set $\widetilde{M}_k$ depends on the eigenvectors $\vec{\eta}^i, i =1,\cdots,k,$ and $c_1,\cdots,c_n$, it is not effective in the variational characterization of $\vec{u}=(c_1,\cdots,c_n)U.$
On the other hand,
in a pioneering paper \cite{LW1} and \cite{LW2} by Lin and Wei, when $A$ is close to the identity matrix, they considered a minimization of $I$ on
\[M_n \equiv \{ \vec{u}=(u_1,\cdots,u_n) \in \SH \ | \ J_i(\vec{u}) = 0, u_j \ne 0 \textup{ for }  1 \le i,j \le n  \},\]
where for each $l \in \{1,\cdots,n\}$ and $\vec{u} \in \SH,$ 
\[J_l(\vec{u}) \equiv I^\prime(\vec{u})(0,\cdots,0,u_l,0,\cdots,0).\]
For general nonsingular $A,$ it is not easy to prove that $M_n$ is smooth, 
which follows if \[\bigg(a_{ij} \int_\Omega u_i^2u_j^2 dx\bigg)_{1\le i,j\le n}\] is nonsingular for any $\vec{u} \in \M_n,$  
and that there exists a minimizer of $I$ on $M_n$.
As it is noted in \cite{AC}, $M_n$ is not closed. Furthermore there do not exist minima of $I$ over $M_n$ for certain $A$ (see \cite{S}).
On the other hand, Sirakov\cite{S} studied the minimization on $M_2$ for $n=2$ and general $\lambda_1, \lambda_2 > 0$ and found quite precise $0 < \nu_1 < \nu_2 < \infty$ depending on $\lambda_1,\lambda_2,a_{11},a_{22}$ such that for $a_{12} \in (0,\nu_1)\cup (\nu_2,\infty),$
the minimum is attained by a positive vector solution of \eqref{eq1}.
Here we characterize the solution $\vec{c}U$ by a local minimizer of $I$ over $M_n$ for general nonsingular $A$.
\begin{thm} \label{V2}
Let $(c_1,\cdots,c_n)U =\vec{c} U \in \R^n U $ be a positive vector solution of  a system \eqref{eq3} with the interaction matrix $A=(a_{ij})_{1\leq i,j\leq n}$ satisfying $a_{ij}=a_{ji} \ge 0 $ for $1\leq i,j\leq n$. Assume that $A =(a_{ij})$ is nonsingular and $U$ is a nondegenerate solution of \eqref{singleeq} in $H$.
Then,  an open neighborhood $V$ of $(c_1,\cdots,c_n)U$ in $M_n$ is a smooth manifold of codimension $n$ and $(c_1,\cdots,c_n)U$ is a strict local minimum point of $I$ on $V \subset M_n.$ 
\end{thm}
\begin{proof}
We first note that for $\vec{u} \in M_n$ and $m, l \in \{1,\cdots,n\},$
\[ J_m^\prime(\vec{u})(0,\cdots,0,u_l,0,\cdots,0) = -2a_{ml}\int_\Omega u_m^2 u_l^2dx.
\]
Thus, for $\vec{u} = (c_1,\cdots,c_n)U,$
we see that 
\[J_m^\prime(\vec{u})(0,\cdots,0,u_l,0,\cdots,0) = -2a_{ml}c_m^2c_l^2\int_\Omega U^4dx.
\]
Since $A=(a_{ij})_{1\le i,j\le n}$ is nonsingular, so is
$(a_{ml}c_m^2c_l^2)_{1\le m,l\le n}$.
Thus, there exists an open neighborhood $V$ of $(c_1,\cdots,c_n)U$ in $M_n$ such that
$V$ is a smooth manifold of codimension $n$ in $\SH.$ 
Note that for $1 \le m \le n$ and $\vec{\varphi} = (\varphi_1,\cdots,\varphi_n) \in \SH,$ 
\bea \label{rderi}& & J_m^\prime (\vec{c} U)(\vec{\varphi})  \nonumber  \\
&  &= 2  \int_{\Omega}c_m\nabla U \cdot \nabla \varphi_m + \lambda c_mU\varphi_m -\sum_{l=1}^n(a_{ml}c_l\varphi_l (c_m)^2+ a_{ml}(c_l)^2 c_m \varphi_m) U^3 dx  \nonumber\\
& & = -2 c_m^2\sum_{l=1}^n a_{ml}c_l\int_\Omega \varphi_l U^3 dx  \nonumber \\
& & = 
-2c_m^2\sum_{l=1}^n a_{ml}c_l \int_\Omega \nabla U\cdot \nabla \varphi_l +  \lambda U \varphi_l dx. \eea  
 Thus, we can identify by the Reisz representation theorem that
 \[ J_m^\prime(\vec{c}U) = -2c_m^2(a_{m1}c_1,\cdots,a_{mn}c_n)U \in \SH.\] 
 Since $(a_{ij}c_j)_{1\le i,j\le n}$ is nonsingular,
 this implies that the tangent space $T_{\vec{c}U}M_n$ is orthogonal to  $\vec{a}U$
 for any $\vec{a} \in \R^n.$
As in \eqref{twoderi}, we see that for any $\vec{\varphi} \in \SH,$
\[
I''(\vec{c}U)(\vec{\phi},\vec{\phi})
=\lVert\vec{\phi}\rVert_{\Omega}^2-\int_{\Omega}U^2\left(\sum_{i=1}^n \phi_i^2+\sum_{i,j=1}^n c_i\phi_i\cdot 2a_{ij}\cdot c_j\phi_j\right)dx,
\]
and that
for any $\vec{\varphi} \in \R^nU,$
\[
I''(\vec{c}U)(\vec{\phi},\vec{\phi})
= -\int_{\Omega}U^2\left(\sum_{i,j=1}^n c_i\phi_i\cdot 2a_{ij}\cdot c_j\phi_j\right)dx,
\]
Let $\mu_1 \ge \cdots \ge \mu_k > 0$ be the positive eigenvalues of $A$ with corresponding orthonormal eigenvectors 
$\vec{\eta}_1=(\eta_1^1,\cdots,\eta_n^1),\cdots,\vec{\eta}_k=(\eta_1^k,\cdots,\eta_n^k),$ respectively.
Then, we get that for $i \in \{1,\cdots,k\},$
\[I''(\vec{c}U)(\vec{\eta}_i \oslash \vec{c}  U,\vec{\eta}_i \oslash \vec{c}   U)
=-2\mu_i\int_{\Omega}U^4dx < 0.\]
Let $N$ be a $k$-dimensional subspace of $\SH$ spanned by $\vec{\eta}_i \oslash \vec{c}  U,
$ $i=1,\cdots,k.$
Then we see that
\[ I''(\vec{c}U)(\vec{\varphi},\vec{\varphi}) < 0 \ \ \textup{ for any } \vec{\varphi} \in N\setminus\{0\}.\]
Since a matrix $(a_{ij}c_j)_{1\le i,j\le n}$ is nonsingular, $\vec{\eta}_i \oslash \vec{c}  U$ should be in the linear span of 
\newline $\{(a_{k1}c_1,\cdots,a_{kn}c_n)U\}_{k=1}^n.$
Thus, for each $m =1,\cdots,k,$ $\vec{\eta}_m \oslash \vec{c}  U$ is orthogonal to 
the tangent space $T_{\vec{c}U}M_n$.
For any $ \vec{\psi} \in T_{\vec{c}U}M_n$ and $ \vec{\varphi}= \sum_{l=1}^k a_l\vec{\eta}_l \oslash \vec{c} U,$
we see that  
\brr
&& I''(\vec{c}U)(\vec{\varphi} +\vec{\psi},\vec{\varphi}+\vec{\psi})
\\
&& = \lVert\vec{\varphi} + \vec{\psi}\rVert_{\Omega}^2-\int_{\Omega}U^2\left(\sum_{i=1}^n (\varphi_i+\psi_i)^2+\sum_{i,j=1}^n c_i(\varphi_i+\psi_i)\cdot 2a_{ij}\cdot c_j(\varphi_j+\psi_j)\right)dx \\
&& = \lVert\vec{\psi}\rVert_{\Omega}^2-\int_{\Omega}U^2\left(\sum_{i=1}^n (\vec{\psi}_i)^2+ 2\sum_{i,j=1}^n c_i\psi_i\cdot a_{ij}\cdot c_j\psi_j\right)dx \\
&& \ \ \ -2\int_{\Omega}U^2\left(\sum_{i,j=1}^n c_i\varphi_i\cdot a_{ij}\cdot c_j \varphi_j 
 +2\sum_{i,j=1}^n c_i\varphi_i\cdot a_{ij}\cdot c_j \psi_j\right) dx\\
&& = I''(\vec{c}U)(\vec{\psi},\vec{\psi})  -2\int_{\Omega}U^2\left(\sum_{l=1}^k\sum_{i,j=1}^n (a_l)^2\eta^l_i\cdot a_{ij}\cdot \eta^l_j U^2
 + 2\sum_{l=1}^k\sum_{i,j=1}^n a_l\eta^l_i\cdot a_{ij}\cdot c_j \psi_j U\right) dx\\
 && = I''(\vec{c}U)(\vec{\psi},\vec{\psi})  -2\sum_{l=1}^k(a_l)^2\mu_l \int_\Omega U^4 dx  -4\int_{\Omega}U^3\sum_{l=1}^k\sum_{j=1}^n \mu_la_l\eta^l_j c_j \psi_j dx \\
&& = I''(\vec{c}U)(\vec{\psi},\vec{\psi})  -2\sum_{l=1}^k(a_l)^2\mu_l \int_\Omega U^4 dx  -4\sum_{l=1}^k\mu_la_l\sum_{j=1}^n\int_{\Omega} \nabla(\eta^l_j c_jU)\cdot \nabla \psi_j + \lambda \eta^l_j c_jU \psi_jdx.\err
Since the tangent space $T_{\vec{c}U}M_n$ is orthogonal to  $\vec{a}U$
 for any $\vec{a} \in \R^n,$ we see that for each $l =1,\cdots,k,$
 \[\sum_{j=1}^n\int_{\Omega} \nabla(\eta^l_j c_jU)\cdot \nabla \psi_j +  \lambda\eta^l_j c_jU \psi_jdx = 0.\]
Thus we see that $ \vec{\psi} \in T_{\vec{c}U}M_n$ and $ \vec{\varphi}= \sum_{l=1}^k a_l\vec{\eta}_l \oslash \vec{c} U,$
\begin{equation} \label{nonpositive} I''(\vec{c}U)(\vec{\varphi} +\vec{\psi},\vec{\varphi}+\vec{\psi}) = I''(\vec{c}U)(\vec{\psi},\vec{\psi})  -2\sum_{l=1}^k(a_l)^2\mu_l \int_\Omega U^4 dx.\end{equation}
By (iv) in Theorem \ref{T1}, we see that $I''(\vec{c}U)$ is an isomorphism on $\SH.$
Proposition 4.81 in \cite{JTSchwartz} says that there exists an isomorphism $T : \SH \to \SH$ and a projection $P$ on $\SH$ such that for any $\phi \in \SH,$ 
\[I''(\vec{c}U)(\vec{\phi},\vec{\phi}) = -\Vert PT\vec{\phi}\Vert +  \Vert (1-P)T\vec{\phi}\Vert.\]Since the Morse index of $\vec{c}U$ is $k,$  $P$ should be a projection to a $k$-dimensional space $G \subset \SH$.
This   implies that the maximal dimension of a subspace $K$ where $I''(\vec{c}U)(\vec{\phi},\vec{\phi}) < 0$ for $\vec{\phi} \in G \setminus\{0\}$ is $k.$
Then, we conclude from \eqref{nonpositive} that  
 \[I''(\vec{c}U)(\vec{\psi},\vec{\psi}) > 0 \ \ \textup{ for any  }  \ \  \vec{\psi} \in T_{\vec{c}U}M_n \setminus\{0\}.\]
This implies that $\vec{c}U$ is a strict local minimum of $I$ over $V.$
This completes the proof.
\end{proof}
For a perturbation result of Theorem \ref{V2},
we may consider a minimization of $\tilde{I}$ over
\[\tilde{M}_n \equiv \{ \vec{u}=(u_1,\cdots,u_n) \in \SH \ | \ \tilde{J}_i(\vec{u}) = 0, u_j \ne 0 \textup{ for }  1 \le i,j \le n  \},\]
where for each $l \in \{1,\cdots,n\}$ and $\vec{u} \in \SH,$ 
\[\tilde{J}_l(\vec{u}) \equiv \tilde{I}^\prime(\vec{u})(0,\cdots,0,u_l,0,\cdots,0).\]
Then, using the continuous dependence of $\tilde{I},\tilde{J}_1,\cdots,\tilde{J}_n$ and their derivatives for $\lambda_1,\cdots,\lambda_n > 0,$ 
we get the following result from Theorem \ref{V2} and the nondegeneracy result (iv) in Theorem \ref{T1}.
\begin{thm} \label{V2-1}
Let the interaction matrix $A=(a_{ij})_{1\leq i,j\leq n}$ satisfying $a_{ij}=a_{ji} \ge 0 $ for $1\leq i,j\leq n$ be nonsingular and $U$ is a nondegenerate solution of \eqref{singleeq} in $H$ with $\lambda > 0$.
Suppose that there exists a solution $\vec{\alpha}= (\alpha_1^2,\cdots,\alpha_n^2)$ with $\alpha_1,\cdots,\alpha_n > 0$ of $A\vec{\alpha} = \vec{1}_n.$
Then,  there exists $\delta > 0$ such that if $\max_{1\le i \le n}|\lambda_i-\lambda|\le \delta,$ there exist $(\alpha^\delta_1,\cdots,\alpha^\delta_n)U \in \tilde{M}_n$
 and an open neighborhood $\tilde{V}$ of $(\alpha^\delta_1,\cdots,\alpha^\delta_n)U$ in 
$\tilde{M}_n$ such that  $\tilde{V}$ is a smooth manifold of codimension $n$
$\tilde{I}$ has a strict local minimum point $\vec{u}$ in $\tilde{V} \subset \tilde{M}_n$
which is a positive vector solution of \eqref{eq1}.
\end{thm}

\begin{section}{\bf Synchronization/nonsynchronization of  positive vector solutions}
\subsection{Synchronization of  positive vector solutions}
Wei and Yao prove in  \cite{WY} that for $a_{12} > \max\{a_{11},a_{22}\},$ any positive vector solution of \eqref{eq3} with $n=2$ is synchronized.
For a case $n=3$ with $a_{11}=a_{22}=a_{33} = 0,$
if the following triangle inequality for $a_{12},a_{23},a_{31}$
\begin{equation} \label{triangle} a_{ij} <a_{jk} + a_{ki} \ \textup { for } \ \{i,j,k\}=\{1,2,3\} \end{equation}
and two among $\{a_{12}$, $a_{13}$, $a_{23}\}$ are same, it is proved  in  \cite[Proposition 7]{BKS} that any solution $\vec{u}$ of \eqref{eq3} is synchronized, that is, $\vec{u} \in \R^3U$ for a positive solution $U$ of \eqref{singleeq}.
In  both cases above for $n=2,3,$ the interaction matrix $A$ has only one positive eigenvalue.
In the following, we prove the synchronization for a general result for arbitrary $n \ge 2$ when the interaction matrix $A$ has only one positive eigenvalue.
\begin{thm}\label{SYN}
Assume that a symmetric matrix $A=(a_{ij})_{1\leq i,j\leq n}$ has exactly one positive eigenvalue and satisfies $a_{ij} \ge 0$ for any $1\leq i,j\leq n$, $\sum_{j=1}^na_{ij} > 0$ for each $i=1,\cdots,n.$
Then, any positive vector solution $\vec{u} =(u_1,\cdots,u_n)$ of \eqref{eq3}  is synchronized.
\end{thm}
\begin{proof}
Suppose that there exists a positive vector solution
$u=(u_1,\cdots,u_n)$ of \eqref{eq3}.
For $1\leq i,j\leq n$, since $u_i^2/u_j\in H$, testing the $j$-th equation of \eqref{eq3} by $u_i^2/u_j$, then subtracting the $i$-th equation of \eqref{eq3} tested by $u_i$, we obtain
\begin{equation}\label{picid}
\int_{\Omega}u_j^2\bigg|\nabla\bigg(\frac{u_i}{u_j}\bigg)\bigg|^2dx=\int_{\Omega}\sum_{k=1}^n (a_{ki}u_k^2-a_{kj}u_k^2)u_i^2dx.
\end{equation}
We define a symmetric matrix $G=(G_{ij})_{1\le i,j\le n}$ by
$G_{ij} = \int_{\Omega} u_i^2u_j^2 dx.$
Then, it is obvious that $G_{ij} > 0$ and
$G$ is nonnegative definite, that is, all eigenvalues of $G$ are nonnegative.
We define a matrix $F= AG = (F_{ij})_{1\le i,j\le n}.$
Note that $F_{ij} >  0$ for $1\le i,j\le n$.
Then the Perron-Frobenius Theorem implies that the largest positive eigenvalue $\omega > 0$ of $F$ has an eigenvector $\vec{d} = (d_1,\cdots,d_n)$ with $d_1,\cdots,d_n >  0$ and $\sum_{i}^n d_i =1.$
We see from \eqref{picid} that
\begin{equation}\label{picid1}
0 \le \int_{\Omega}u_j^2\bigg|\nabla\bigg(\frac{u_i}{u_j}\bigg)\bigg|^2dx=  (AG)_{ii} -(AG)_{ji} =  F_{ii} - F_{ji}.
\end{equation}
Then, it follows  from \eqref{picid1} that
\begin{equation}
\label{trinequal1}\text{tr} F - \omega = \sum_{i=1}^n F_{ii} - \sum_{i=1}^n (F\vec{d})_{i} =
\sum_{i,j}d_j(F_{ii} - F_{ij}) \ge 0. 
\end{equation}
On the other hand, since the nonzero eigenvalues of $F= AG= (AG^{1/2})G^{1/2}$ 
are same with those of a symmetric matrix $G^{1/2}	AG^{1/2}$ and the number of positive eigenvalue of the symmetric matrix $G^{1/2}AG^{1/2}$ is less than or equal to that of $A,$
 we deduce that $F$ has at most one positive eigenvalue.
This implies that $tr F \le \omega$; this and \eqref{trinequal1} imply that
$\text{tr} F = \omega.$
Applying the property $\text{tr} F = \omega$ to \eqref{picid1} and \eqref{trinequal1}, we see that $F_{ii}-F_{ij}= 0$ for any $1\le i,j\le n.$
Thus we see from \eqref{picid1} that $u_i/u_j$ is a constant for any $1\le i,j\le n.$
This completes the proof.
\end{proof}
The above general result gives us the following nonexistence result, which extends some previous nonexistence result for $\lambda_1=\cdots=\lambda_n$ and $n=2,3$ in \cite{BW},\cite{BKS},\cite{WY}.

\begin{cor}\label{non}
Suppose that a symmetric matrix $A=(a_{ij})_{1\leq i,j\leq n}$ has exactly one positive eigenvalue and satisfies $a_{ij} \ge 0$ for any $1\leq i,j\leq n$, $\sum_{j=1}^na_{ij} > 0$ for each $i=1,\cdots,n.$
Then, \eqref{eq3} has no positive vector solutions if there exists no vector $\vec{c}=(c_1,\cdots,c_n) \in \R^n$ satisfying $\sum_{j}a_{ij}c_j^2 =1 $ and $c_i \ne 0$ for each $i =1,\cdots,n.$
\end{cor}

When $a_{ij} > 0$ for any $1\le i,j \le n$ and $\Omega=\R^N,$ any positive vector solution of \eqref{eq3} is radially symmetric up to a translation.
Thus, we have the following uniqueness result.   
\begin{cor} Let $\Omega$ be a domain such that the scalar equation \eqref{singleeq} has a unique positive solution (up to a translation when $\Omega =\mathbb R^N$).
Suppose that $A$ has only one positive eigenvalue and all entries are positive.
Then, \eqref{eq3} has a unique positive vector solution (up to a translation when $\Omega =\mathbb R^N$) if there exists a unique vector $\vec{c}=(c_1,\cdots,c_n) \in \R^n$ satisfying $\sum_{j}a_{ij}c_j^2 =1 $ and $c_i \ne 0$ for each $i =1,\cdots,n$. 
\end{cor}

\end{section}
\subsection{Non-synchronization of positive vector solutions }
We now prove a non-synchronization result for a class $\mathscr{A}$ of nonsingular matrices $A$ with at least two positive eigenvalues and some other generic properties.
\begin{definition} We say $A=(a_{ij})_{1\leq i,j\leq n}\in\mathscr{A}$ if the following conditions hold.
\begin{enumerate}
\label{admissible}
    \item [(i)] $a_{ij}=a_{ji}>0$.
    \item [(ii)] There exist $c_1,\cdots,c_n>0$ satisfying $\sum_{j=1}a_{ij}c_j^2 = 1$ for each $i=1,\cdots,n$.
    \item [(iii)] $A$ has at least two positive eigenvalues.
    \item [(iv)] All principal submatrices of $A$, including $A$, are nonsingular.
    \item[(v)] All nonnegative synchronized solutions of \eqref{eq3} on a ball are non-degenerate. (In the view of Theorem \ref{T1} and \ref{T1-1}, it is equivalent to that for any nonnegative synchronized solution $\vec{d}=(d_1,\cdots,d_n)U$
    of \eqref{eq3}, 
    $\sum_{j=1}^n a_{ij}d_j^2\notin\{\omega_1,\omega_2,\cdots\}$ for all $i$ satisfying $d_i=0$, where $\{\omega_l\}_{l=1}^\infty$ are the eigenvalues of \eqref{speceq}.
\end{enumerate}
\end{definition}
 The conditions $(i)-(iii)$ in the class $\mathscr{A}$ are  generic in the sense that if there is an element $A=(a_{ij})_{1\leq i,j\leq n}$ satisfies $(i)$, $(ii)$, and $(iii)$, then for any small neighborhood $V$ of $A$ in the class of symmetric matrices, $V\cap \mathscr{A}$ is dense in $V$, which can be proved by a standard argument.
For $a \in \R^N, r > 0$, we define $B_r(a)= \{x \in \R^N \ | \ |x-a| < r\}.$
We denote $e_N = (0,\cdots,0,1) \in \R^N.$
\begin{thm}\label{NSYN}
Let $N=2,3$. For an interaction matrix $A=(a_{ij})_{1\le i,j\le n}$, we assume $A\in\mathscr{A}$.
Then, there exists a  domain $\Omega\subset\R^N$ such that  \eqref{eq3}  has a nonsynchronized positive vector solution.
\end{thm}
\begin{proof}
\textit{Step 1: Construction of a nonsyncrhonized solution $\vec{u}_{\e}$.} First of all, $A$ having at least two positive eigenvalues implies that the vector $\vec{x}=(x_1,\cdots,x_n)$ maximizing $\vec{x}^TA\vec{x}$ on
$$
\{\vec{x}:x_1+\cdots+x_n=1\ \mathrm{and}\ x_1,\cdots,x_n\geq0\}
$$
is not a positive vector. Fix such vector $\vec{x}$ having the most positive number of entries. By changing the coordinates if necessary, we let $\vec{x}=(x_1,\cdots,x_m,0,\cdots,0)$ for some $m<n$. Let $d_i=\sqrt{x_i/(\vec{x}^TA\vec{x})}$ for $i=1,\cdots,m$ and $d_i=0$ for $i=m+1,\cdots,n$. Then,
$$
\sum_{j=1}^m a_{ij}d_j^2=1
$$
for $i=1,\cdots,m$ and
$$
\sum_{j=1}^m a_{ij}d_j^2\leq 1
$$
for $i=m+1,\cdots,n$.
In particular, $\vec{d}U:=(d_1,\cdots,d_n)U$ is a solution of \eqref{eq3}, and the nondegeneracy result in Theorem \ref{T1-1} implies
$$
\sum_{j=1}^m a_{ij}d_j^2<1
$$
for $i=m+1,\cdots,n$
since the first eigenvalue of \eqref{speceq} is $1.$

Let $\Omega_0=B_1(0)\cup  B_1(3e_N)$. For small $\e > 0$, let $\Omega_\e$ be a domain with a smooth boundary such that
\[ \Omega_0 \subset \Omega_\e \subset  B_5(0),\quad\lim_{\e \to 0} |\Omega_\e \setminus \Omega_0| = 0.\]
Now, we set up to define compact maps of which we will calculate the Leray-Schauder degree. 
We define a restriction map $R_O:L^2(B_5(0))\rightarrow L^2(O)$ by $R_O(u)=u|_{O}$, and define a map $E_O:H^1_0(O)\rightarrow L^2(B_5(0))$ by the composition of the zero extension to $B_5(0)$ and the compact embedding $H_0^1(B_5(0)) \hookrightarrow L^2(B_5(0))$. 
We also define $(-\Delta+ \lambda)^{-1}_O:H^{-1}(O)\rightarrow H^1_0(O)$ by $f\mapsto u$ where $u\in H^1_0(O)$ is the unique solution of
$$
-\Delta u+ \lambda u = f \quad\mathrm{in}\quad O,  \ \ u = 0 \text{ on } \partial O
$$
for a given $f\in H^{-1}(O)$. 
Then, we define $K=(K_1,\cdots,K_n) : (L^2(B_5(0)))^n\rightarrow(L^2(B_5(0)))^n$ by
$$
K_i(u_1,\cdots,u_n)=E_{\Omega_0}\circ (-\Delta+ \lambda)^{-1}_{\Omega_0} \circ R_{\Omega_0}\big(\sum_{j=1}^n a_{ij}u_iu_j^2\big),
$$
and define $K_{\e}=(K_{1,\e},\cdots,K_{n,\e}) :(L^2(B_5(0)))^n\rightarrow(L^2(B_5(0)))^n$ by 
$$
K_{i,\e}(u_1,\cdots,u_n)=E_{\Omega_{\e}} \circ (-\Delta+ \lambda)^{-1}_{\Omega_{\e}}\circ R_{\Omega_{\e}}\big(\sum_{j=1}^n a_{ij}u_iu_j^2\big).
$$
We note that $K$ and $K_{\e}$ are compact due to the compactness of $E_{\Omega_0}.$ Now, we define $W=(w_1,\cdots,w_n)\in(H^1(B_5(0)))^{n}$ by
$$
W=\begin{cases}
    (c_1,\cdots,c_n)U_0\ \mathrm{in}\ B_1(0)\\
    \vec{d}U_1\ \mathrm{in}\ B_1(3e_N)\\
    0\ \mathrm{otherwise}
\end{cases},
$$
where $U_0$ is the solution of \eqref{singleeq} in $B_1(0)$ and $U_1$ is the solution of \eqref{singleeq} in $B_1(3e_N)$. 
 We define $N_R(W) =\{w\in(L^2(B_5(0)))^n : \lVert w-W\rVert_{(L^2(B_5(0)))^n}<R\}$.
 By the nondegeneracy results $(iv)$ of Theorem \ref{T1} and Theorem \ref{T1-1}, there is a small $R>0$ 
 such that $W$ is the only zero of $I-K$ in  $N_R(W)$ and the Leray-Schauder degree $\deg(I-K,N_R(W),(0,\cdots,0))$ is nonzero.

We claim that for small $\e>0$,
\[(I-(1-t)K_{\e}-tK)(u_1,\cdots,u_n)) \ne 0, \ \ t \in [0,1], \ \ (u_1,\cdots,u_n) \in \partial N_R(W).\]
To the contrary, suppose that there exist sequences $\{\e_n\}\subset(0,\infty)$, $\{t_n\}\subset[0,1]$, and $\{\vec{u}_n\}\subset\partial N_R(W)$ such that $\e_n\rightarrow 0$ and $(I-(1-t_n)K_{\e_n}-t_nK)(\vec{u}_n)=0$. Then, we see from the compactness of the map $E_{\Omega_0}$ and $E_{\Omega_\e}$ that $\vec{u}_n$ converges, up to a subsequence, to an element $\vec{u} \in L^2(B_5(0))$.
 Then, we see from the fact $\lim_{ n\to \infty}\e_n = 0$ that as $n \to \infty,$
 $K_{\e_n}(\vec{u}_n) \to K(\vec{u}) $ in $(L^2(B_5(0)))^n.$
This implies that $(I-K)(\vec{u}) = 0$ and $\vec{u} \in  \partial N_R(W)$;
this is a contradiction and proves the claim.

 Since $(1-t)(I-K_{\e})+t(I-K)$ is nonzero on $\partial N_R(W)$ for any $t \in [0,1]$ and small $\e>0$, we see from the homotopy invariance of the Leray Schauder degree that for small $\e>0$ and $\delta\geq0$,
$$
\deg(I-K_{\e},N_R(W),(0,0))=\deg(I-K,N_R(W),(0,0))\neq0.
$$
Therefore, for small $\e>0$, there exists a solution $\vec{u}_{\e}$ of \eqref{eq3} in $\Omega_{\e}$ which converges to $W$ in $(H^1_0(B_5(0)))^n$ as $\e\to0$. For small $\e>0$, such solutions must be nonsynchronized by the definition of $W$.

\textit{Step 2: Showing positivity of the nonsynchronized solution $\vec{u}_{\e}$.} It now suffices to show the positivity of $\vec{u}_{\e}$ for small $\e>0$. Let $\vec{u}_{\e}=(u_{1,\e},\cdots,u_{n,\e})$ and $W_0=(w_1,\cdots,w_n)$. Fix $i\in\{1,\dots,n\}$ and define the scalar Schr\"odinger operator
\begin{equation*}
 L_{i,\e}:=-\Delta+\lambda-\sum_{j=1}^n a_{ij}u_{j,\e}^2.
\end{equation*}
Note that $L_{i,\e}u_{i,\e}=0$.
Since $w_i\not\equiv0$, $u_{i,\e}\not\equiv0$ for small $\e>0$.  Hence the principal eigenvalue
$\mu_{i,\e}:=\lambda_1(L_{i,\e},\Omega_\e)$ satisfies
\begin{equation*}
 \mu_{i,\e}\leq0.
\end{equation*}

We now claim that $\mu_{i,\e}=0$. Assume by contradiction that there exists a sequence $\e\to0$ for which $\mu_{i,\e}<0$.  Let $\phi_\e>0$ be the corresponding principal eigenfunction, normalized by $ \|\phi_\e\|_{L^2(\Omega_\e)}=1.$
Its zero extension satisfies
\begin{equation}\label{eq:eigen-q}
 \int_{B_5(0)}\bigg(|\nabla\phi_\e|^2+\lambda\phi_\e^2
 -\sum_{j=1}^n a_{ij}u_{j,\e}^2\phi_\e^2\bigg)dx=\mu_{i,\e}<0.
\end{equation}
Because $2\leq N\leq3$ and $u_{j,\e}$ is bounded in $H_0^1(B_5(0))$,
$\sum_{j=1}^n a_{ij}u_{j,\e}^2$ is bounded in $L^p(B_5(0))$ for some $p>N/2$.  For any small $\delta > 0,$ the following estimate
\[
 \int_{B_5(0)} \sum_{j=1}^n a_{ij}u_{j,\e}^2\phi^2dx
 \leq \delta\|\nabla\phi\|_{L^2(B_5(0))}^2+C_\delta\|\phi\|_{L^2(B_5(0))}^2
\]
holds uniformly in $\e$.  Choosing $\delta$ small in \eqref{eq:eigen-q}, we conclude
that $\{\phi_\e\}$ is bounded in $H_0^1(B_5(0))$.  Thus, after passing to a subsequence,
\begin{equation}\label{eq:phi-conv}
 \phi_\e\rightharpoonup\phi\text{ in }H_0^1(B_5(0)),
 \quad
 \phi_\e\to\phi\text{ in }L^q(B_5(0))
\end{equation}
for every $q<2^*$; in particular, $\|\phi\|_{L^2(B_5(0))}=1$, $\phi\geq0$, and $\phi\in H_0^1(\Omega_0)$.
Now,
\[
 \sum_{j=1}^n a_{ij}u_{j,\e}^2\to \sum_{j=1}^n a_{ij}w_j^2
\]
strongly in $L^p(B_5(0))$ for some $p>N/2$.  Therefore, using
\eqref{eq:phi-conv} and weak lower semicontinuity,
\begin{equation*}
 q_i(\phi):=
 \int_{\Omega_0}\bigg(|\nabla\phi|^2+\lambda\phi^2-\sum_{j=1}^n a_{ij}w_j^2\phi^2\bigg)dx
 \leq\liminf_{\e\to0}\mu_{i,\e}\leq0.
\end{equation*}
We note that 
$$\sum_{j=1}^n a_{ij}w_j^2=U_0^2 \ \ \text{ in } \ \ B_1(0),$$
\begin{equation*}
\sum_{j=1}^n a_{ij}w_j^2=
 \begin{cases}
 U_1^2\mathrm{\ if\ }1\leq i\leq m,\\
 \sum_{j=1}^m a_{ij}d_j^2U_1^2\mathrm{\ if\ }m+1\leq i\leq n
 \end{cases} \ \ \text{ in } \ \ B_{1}(3e_N).
\end{equation*}
Recall that $\sum_{j=1}^m a_{ij}d_j^2<1$ for $m+1\leq i\leq n$. Therefore, $q_i(\phi)=0$, and in particular, we obtain
\begin{equation}\label{eq:phi-kernel}
 \phi=\tilde{c}_iU_0+\tilde{d}_iU_1,
 \quad \tilde{c}_i,\tilde{d}_i\geq0,
 \quad (\tilde{c}_i,\tilde{d}_i)\neq(0,0).
\end{equation}
On the other hand, $\phi_\e$ and $u_{i,\e}$
correspond to the distinct eigenvalues $\mu_{i,\e}<0$ and $0$. Therefore,
\[
 \int_{\Omega_\e}\phi_\e u_{i,\e}=0.
\]
Passing to the limit and using \eqref{eq:phi-kernel}, we get 
\[
 0=\int_{\Omega_0}\phi w_i
 = c_i\tilde{c}_i
 \int_{B_1(0)}U_0^2dx+ d_i\tilde{d}_i\int_{B_1(3e_N)}U_1^2dx>0.
\]
This is a contradiction and we conclude that $\mu_{i,\e}=0$.

Note that the principal eigenspace on the domain $\Omega_{\e}$ is one-dimensional and is spanned by a strictly positive eigenfunction.  By $L_{i,\e}u_{i,\e}=0$,
$u_{i,\e}$ is a scalar multiple of the eigenfunction.  Because $u_{i,\e}\to {w_i}$ and $w_i$ is positive on $B_1(0)$, $u_{i,\e}>0$ in $\Omega_\e$ for all sufficiently small $\e>0$.  Since $i$ was arbitrary, this proves the positivity of $\vec{u}_{\e}$.
\end{proof}

\begin{remark}
    For $n=2$, the assumptions $(i)$, $(ii)$, and $(iii)$ in Definition \ref{admissible} imply $0< a_{12}<\min\{a_{11},a_{22}\}$, which in turn, imply $(iv)$ and $(v)$. Therefore, if $0 < a_{12}<\min\{a_{11},a_{22}\}$, then there exists a domain where there exists a nonsynchronized positive vector solution for some domain $\Omega$. Also for this case $0<a_{12}<\min\{a_{11},a_{22}\}$, when $\Omega =\mathbb R^N$, synchronization (therefore uniqueness) of positive vector solutions have been given recently in \cite{CLWY} (see also earlier partial results in \cite{WY, CZ1, ZW}). If $\min\{a_{11},a_{22}\}\leq a_{12}\leq \max\{a_{11},a_{22}\}$, then there exists no positive vector solution unless $a_{11}=a_{12}=a_{22}$ (e.g.,\cite{BW}).
     Corollary \ref{non} applies to the range $\sqrt{a_{11} a_{22} }\leq a_{12} \leq a_{22}$ for non-existence of positive vector solutions. If $a_{12}>\max\{a_{11},a_{22}\}$, any positive vector solution is synchronized, which is shown in \cite{WY} and now a special case of Theorem \ref{SYN}.
\end{remark}

\begin{remark}
    The assumptions $(iv)$ and $(v)$ in Definition \ref{admissible} can be weakened as in the following statement: 
    
    there exist a vector $\vec{d}=(d_1,\cdots,d_n)$ and an index set $S\subsetneq\{1,\cdots,n\}$ such that $d_i>0$ for all $i\in S$, $d_i=0$ for all $i\notin S$, the principal submatrix $(a_{ij})_{i,j\in S}$ is nonsingular, $\sum_{j=1}^n a_{ij}d_j^2=1$ for all $i\in S$, and $\sum_{j=1}^n a_{ij}d_j^2<1$ for all $i\notin S$.

    The matrix $A$ having two or more positive eigenvalues is a necessary condition for such $\vec{d}$ and $S$ to exist. However, it is not a sufficient condition. A counterexample where $A$ has two or more positive eigenvalues but no such $\vec{d}$ and $S$ satisfying $\sum_{j=1}^n a_{ij}d_j^2<1$ for all $i\notin S$ exists is given as
    $$
    A=\begin{pmatrix}
        \frac12&1&1\\
        1&1&\frac12\\
        1&\frac12&1
    \end{pmatrix}.
    $$
    In fact, $A$ having two positive eigenvalues imply the existence of such $\vec{d}$ and $S$ with $\sum_{j=1}^n a_{ij}d_j^2\leq 1$ for all $i\notin S$ instead of $\sum_{j=1}^n a_{ij}d_j^2<1$ for all $i\notin S$. The possibility $\sum_{j=1}^n a_{ij}d_j^2=1$ for some $i\notin S$ is what blocks us from removing the condition $(v)$ in Definition \ref{admissible}.
\end{remark}

\noindent{\bf Acknowledgement}. 
The first author was supported by the National Research Foundation of Korea(NRF) grant funded by the Korea government(MSIT)(No.NRF-2023R1A2C1005734).
The third author is grateful to the Department of Mathematical Sciences at KAIST for hosting his visit in 2023 when the work was initiated.


\begin{thebibliography}{1}
\footnotesize
\bibitem{AC}
{A. Ambrosetti, E. Colorado,}
 { Standing waves of some coupled nonlinear Schr\"odinger equations.}
{\sl J. Lond. Math. Soc.} {\bf 75} (2007), 67--82.


\bibitem{ACR} A. Ambrosetti, E. Colorado, D. Ruiz,
    {Multi-bump solitons to linearly coupled systems of nonlinear Schr\"odinger equations.}
    {\sl Calc. Var. PDEs} {\bf 30} (2007) p85--112.



\bibitem{BDW} T. Bartsch, E. N. Dancer, Z.-Q. Wang,
    { A Liouville theorem, a-priori bounds, and bifurcating branches of positive solutions for a nonlinear elliptic system.}
    {\sl Calc. Var. PDEs} {\bf 37} (2010) 345--361.

\bibitem{BW}
   {T. Bartsch, Z.-Q. Wang,}
 { Note on ground states of nonlinear Schr\"odinger systems.}
  {\sl J. Part. Diff. Equ.} {\bf 19} (2006), 200--207.

\bibitem{BWW} T. Bartsch, Z.-Q.  Wang, J. Wei,
 {   Bound states for a coupled Schr\"odinger system.}
    {\sl J. Fixed Point Theory Appl.} {\bf 2} (2007) 353--367.


\bibitem{B1} J. Byeon,
{ Existence of large  positive solutions
of some nonlinear elliptic equations on singularly perturbed domains}.
Comm. in P. D. E. {\bf   22} (1997), 1731 -- 1769.

\bibitem{B2} J. Byeon, {Standing waves of nonlinear Schrödinger
systems with all attractive forces.}
A survey paper in SPECIAL ISSUE : 100 years of the Journal of the London Mathematical Society,
{\sl J. Lond. Math. Soc.} (2026)

\bibitem{BKS} J. Byeon, O. Kwon, J. Seok, 
{ Positive vector solutions for nonlinear Schr\"odinger systems with strong interspecies attractive forces}. J. Math. Pures Appl. {\bf 143} (2020), 73--115.

\bibitem{CLLL} S. Chang, C.S. Lin, T.C. Lin, W. Lin,
 {Segregated nodal domains of
two-dimensional multispecies Bose-Einstein condensates}. {\sl Phys. D},
{\bf 196} (2004), 341--361.


\bibitem{CLWY} H. Chen, Y.Liu, J. Wei, W. Yang,
{ On Sirakov’s equal-frequency uniqueness conjecture}, arXiv:2607.28279v1.




\bibitem{CZ1} Z. Chen and  W. Zou, 
An optimal constant for the existence of least energy
solutions of a coupled Schr\"odinger system,
{\sl Calc. Var. Partial Differential Equations} {\bf 48} (2013) 695--711.

\bibitem{CTV} M. Conti, S. Terracini, G. Verzini,
{ Nehari's problem and competing species system.}
    {\sl Ann. I. H. Poincar\'e} {\bf 19} (2002), 871--888.
    
  \bibitem{Corr1} S. Correia,
{\em Characterization of ground-states for a system of M
coupled semilinear Schr\"odinger equations
and applications,}
{\sl J. Differential Equations} {\bf 260}(2016), 3302--3326.

\bibitem{Corr2}  S. Correia, 
{\em Ground-states for systems of  M  coupled semilinear Schr\"odinger equations with attraction-repulsion effects: characterization and perturbation results,}
{\sl Nonlinear Anal.} {\bf 140} (2016), 112--129.  

\bibitem{CFT}  S. Correia, F. Oliveira, H. Tavares, 
 Semitrivial vs. fully nontrivial ground states in cooperative cubic Schr\"odinger systems with  $d \ge 3$  equations
{\sl J. Funct. Anal.} {\bf 271} (2016),  2247--2273.


\bibitem{D}
E.N. Dancer,
{ The effect of domain shape on the number of positive solutions of certain nonlinear equations}.
J. Differential Equations {\bf 74} (1988),120--156.

\bibitem{DWW} E. N. Dancer, J. Wei, T. Weth,
    {A priori bounds versus multiple existence of positive solutions for
    a nonlinear Schr\"odinger system.}
    {\sl Ann. I. H. Poincar\'e} {\bf 27} (2010), 953--969.


\bibitem{dFL} D.G. de Figueiredo,  O. Lopes,
{ Solitary waves for some nonlinear Schr\"odinger systems,}
{\sl Ann. Inst. H. Poincaré C Anal. Non Linéaire}{\bf  25} (2008), 149--161.


\bibitem{FMT} P. Felmer,  S. Martínez,  K. Tanaka 
{Uniqueness of radially symmetric positive solutions for $-\Delta u + u= u^p$ 
 in an annulus}
{\sl J. of Differential Equations}{bf 245}(2008), 1198--1209.


\bibitem{GNN} B. Gidas, W.-M. Ni, L. Nirenberg,
{Symmetry and related properties via the maximum principle.}
{\sl Comm. Math. Phys.} {\bf 68} (1979) no. 3, 209--243.


\bibitem{GT} D. Gilbarg, N. S. Trudinger,
Elliptic partial differential equations of second order.
Second edition. 224. Springer-Verlag, Berlin, (1983).


\bibitem{IT} N. Ikoma, K. Tanaka,
{A local mountain pass type result for a system
of nonlinear Schr\"odinger equations.}
{\sl Calc. Var. Partial Differ.}{\bf 40} (2011) 449--480

\bibitem{LW1}
  {T.-C. Lin, J. Wei,}
  { Ground state of $N$ coupled nonlinear Schr\"odinger equations
in $\R^n, n\leq3$.}
 {\sl Comm. Math. Phys.} {\bf 255} (2005), 629--653.

\bibitem{LW2}
    { T.-C. Lin, J. Wei,}
    { Spikes in two coupled nonlinear Schr\"{o}dinger equations.}
    {\sl Ann. Inst. H. Poincar\'{e} } {\bf 22} (2005), 403--439.

\bibitem{L} P. L. Lions,
{The concentration-compactness principle in the calculus of variations.
The locally compact case, part II.}
{\sl Ann. Inst. H. Poincar\'{e} } {\bf 1} (1984), 223--283.



\bibitem{LcyW1} C. Liu, Z.-Q. Wang, A complete classification of ground-states for a
coupled nonlinear Schr\"odinger system.
{\sl Communications On
Pure And Applied Analysis}
{\bf 16} (2017), 115--130.

\bibitem{LLC} 
H. Liu, Z. Liu and J. Chang, Existence and uniquiness of positive solutions of nonlinear
Schr\"oodinger systems, {\sl Proceedings of the Royal Society of Edinburgh: Section A Mathematics},
{\bf 145} (2015), 365–-390.

\bibitem{LLW} J. Liu, X. Liu, Z.-Q. Wang, {Multiple mixed states of nodal solutions for nonlinear Schr\"odinger systems.}
{\sl Calc. Var. PDEs}, {\bf 52} (2015), 565--586.


\bibitem{LWz1} Z. Liu, Z.-Q. Wang,
    {Multiple bound states of nonlinear Schr\"odinger systems.}
    {\sl Comm. Math. Phys.} {\bf 282} (2008) 721--731.

\bibitem{LWz2} Z. Liu, Z.-Q. Wang,
   { Ground states and bound states of a nonlinear Schr\"odinger system.}
    {\sl Advanced Nonlinear Studies} {\bf 10} (2010) 175--193

\bibitem{MMP}
    {L.A. Maia, E. Montefusco, B. Pellacci,}
     {Positive solutions for a weakly coupled nonlinear Schr\"{o}dinger system. }
     {\sl J. Diff. Equ.} {\bf 299} (2006), 743--767.

\bibitem{MS} M. Mitchell, M. Segev, {Self-trapping of inconherent
white light. } {\sl Nature}, {\bf 387} (1997), 880--882.

\bibitem{MPS}
     {E. Montefusco, B. Pellacci, M. Squassina,}
    {Semiclassical states for weakly coupled nonlinear Schr\"{o}dinger systems. }
    {\sl J. Eur. Math. Soc.} {\bf 10} (2008), 41--71.

\bibitem{NR} B. Noris, M. Ramos,
   { Existence and bounds of positive solutions for a nonlinear Schr\"odinger system.}
    {\sl Proceedings of the AMS} {\bf 138} (2010) 1681--1692.

\bibitem{NTTV}
{B. Noris, H. Tavares, S. Terracini, G. Verzini,}
{ Uniform H\"older bounds for nonlinear Schr\"odinger systems with strong competition.}
 {\sl Comm. Pure and Appl. Math.} {\bf 63} (2010), 267--302.

\bibitem{PW} S. Peng, Z.-Q. Wang, {Segregated and Synchronized Vector Solutions for Nonlinear Schr\"{o}dinger System.,}
{\sl Arch. Ration. Mech. Anal.} {\bf 208} (2013), 305--339.

\bibitem{R}
Ch. R\"uegg et al, {Bose-Einstein condensation of the triple states
in the magnetic insulator TlCuCl$_3$.} 
{\sl Nature}, {\bf 423} (2003), 62--65.


\bibitem{SW1} Y. Sato, Z.-Q. Wang,
    {On the multiple existence of semi-positive solutions for a nonlinear Schr\"odinger system.}
	{\sl Ann. Inst. H. Poincare } {\bf 30} (2013), 1--22.



\bibitem{S} B. Sirakov,
   { Least energy solitary waves for a system of nonlinear Schr\"odinger equations in $\R^n$.}
    {\sl Comm. Math. Phys.} {\bf 271} (2007), 199--221.

\bibitem{S1} N. Soave, {On existence and phase separation of solitary waves
for nonlinear Schr\"odinger systems modelling
simultaneous cooperation and competition.} {\sl Calc. Var. PDEs} {\bf 53 } (2015), 689--718.

\bibitem{ST} N. Soave and H. Tavares,
{New existence and symmetry results for least energy positive solutions of Schr\"odinger systems with mixed competition and cooperation terms.}
{\sl J. Differential Equations} {\bf 261} (2016), 505–537.


\bibitem{JTSchwartz}
J. T. Schwartz, 
Nonlinear functional analysis, 
Notes on Mathematics and its Applications,
Gordon and Breach Science Publishers, New York-London-Paris, 1969.



\bibitem{TTVW}
H. Tavares, S. Terracini, G. Verzini, T. Weth,
{Existence and nonexistence of entire solutions for non-cooperative cubic elliptic systems.}
{\sl Comm. Partial Differential Equations} {\bf 36} (2011), 1988--2010.


\bibitem{TV}
     {S. Terracini, G. Verzini,}
     {Multipulse Phase in $k$-mixtures of Bose-Einstein condensates.}
     {\sl Arch. Rat. Mech. Anal.} {\bf 194} (2009), 717--741.

\bibitem{TW1}
{R. Tian, Z.-Q. Wang,}
{ Multiple solitary wave solutions of nonlinear Schr\"{o}dinger systems.}
{\sl Topo. Meth. Non. Anal.} {\bf 37} (2011), 203--223.

\bibitem{WW1}
Z.-Q. Wang, M. Willem, {Partial symmetry of vector solutions for elliptic systems.} {\sl Journal d' Analyse Mathematique}, {\bf 122} (2014), 69--85.

\bibitem{WW11}
{J. Wei, T. Weth, T.} {Nonradial symmetric bound states for
a system of two coupled Schr\"odinger equations.} {\sl Rend. Lincei Mat.
Appl.} {\bf 18} (2007),  279--293.

\bibitem{WW12}
{J. Wei, T. Weth, T.}
{Radial solutions and phase separation in a system of two
coupled Schr\"odinger equations.}
{\sl Arch. Rat. Mech. Anal.} {\bf 190} (2008), 83--106.

\bibitem{WY}
	J. Wei, W. Yao, {Uniqueness of positive solutions to some coupled nonlinear Schr\"{o}dinger equations.} {\sl Commun. Pure Appl. Anal.} {\bf 11} (2012), 1003-1011.


\bibitem{WZZ}
J. Wei, X. Zhong, W. Zou, On Sirakov's open problem and related topics. Ann. Sc. Norm. Super. Pisa Cl. Sci. {\bf 23} (2022), 959--992.

\bibitem{ZW} L. Zhou and Z.-Q. Wang, 
Uniqueness of positive solutions to some Schr\"odinger systems, 
Nonlinear Anal. 195 (2020), 111750. 

\end{thebibliography}
\end{document}